\documentclass[11pt]{article}
\usepackage{amsmath,amsthm,amssymb}
\usepackage[usenames,dvipsnames]{xcolor}
\usepackage{enumerate}
\usepackage{graphicx}
\usepackage{cite}
\usepackage{comment}
\usepackage{oands}
\usepackage{tikz}
\usepackage{changepage}
\usepackage{bbm}
\usepackage{mathtools}
\usepackage[margin=1in]{geometry}
\usepackage[pagewise,mathlines]{lineno}
\usepackage{appendix}
\usepackage{stmaryrd}
\usepackage{multicol}
\usepackage{microtype}
\usepackage[colorinlistoftodos]{todonotes}

\usepackage[pdftitle={Tightness of supercritical Liouville first passage percolation},
  pdfauthor={Jian Ding and Ewain Gwynne},
colorlinks=true,linkcolor=NavyBlue,urlcolor=RoyalBlue,citecolor=PineGreen,bookmarks=true,bookmarksopen=true,bookmarksopenlevel=2,unicode=true,linktocpage]{hyperref}

\theoremstyle{plain}
\newtheorem{thm}{Theorem}[section]

\newtheorem{lem}[thm]{Lemma}
\newtheorem{prop}[thm]{Proposition}

\def\@rst #1 #2other{#1}
\newcommand\MR[1]{\relax\ifhmode\unskip\spacefactor3000 \space\fi
  \MRhref{\expandafter\@rst #1 other}{#1}}
\newcommand{\MRhref}[2]{\href{http://www.ams.org/mathscinet-getitem?mr=#1}{MR#2}}

\theoremstyle{definition}
\newtheorem{defn}[thm]{Definition}
\newtheorem{remark}[thm]{Remark}

\numberwithin{equation}{section}

\newcommand{\dsb}{\begin{adjustwidth}{2.5em}{0pt}
\begin{footnotesize}}
\newcommand{\dse}{\end{footnotesize}
\end{adjustwidth}}

\newcommand{\ssb}{\begin{adjustwidth}{2.5em}{0pt}}
\newcommand{\sse}{\end{adjustwidth}}

\newcommand{\aryb}{\begin{eqnarray*}}
\newcommand{\arye}{\end{eqnarray*}}
\def\alb#1\ale{\begin{align*}#1\end{align*}}
\def\allb#1\alle{\begin{align}#1\end{align}}
\newcommand{\eqb}{\begin{equation}}
\newcommand{\eqe}{\end{equation}}
\newcommand{\eqbn}{\begin{equation*}}
\newcommand{\eqen}{\end{equation*}}

\newcommand{\BB}{\mathbbm}
\newcommand{\ol}{\overline}

\newcommand{\op}{\operatorname}

\newcommand{\bd}{\mathbf}

\newcommand{\frk}{\mathfrak}
\newcommand{\eqD}{\overset{d}{=}}
\newcommand{\ep}{\varepsilon}
\newcommand{\rta}{\rightarrow}

\newcommand{\wt}{\widetilde}
\newcommand{\wh}{\widehat}
\newcommand{\mcl}{\mathcal}

\newcommand{\bdy}{\partial}
\newcommand{\rng}{\mathring}

\newcommand{\ccL}{{\mathbf{c}_{\mathrm L}}}
\newcommand{\ccM}{{\mathbf{c}_{\mathrm M}}}

\let\originalleft\left
\let\originalright\right
\renewcommand{\left}{\mathopen{}\mathclose\bgroup\originalleft}
\renewcommand{\right}{\aftergroup\egroup\originalright}

\title{Tightness of supercritical Liouville first passage percolation}
 \date{ }
 \author{
\begin{tabular}{c} Jian Ding\\[-5pt]\small University of Pennsylvania \end{tabular}
\begin{tabular}{c} Ewain Gwynne\\[-5pt]\small University of Chicago \end{tabular}
}

\begin{document}

\maketitle

\begin{abstract} 
\emph{Liouville first passage percolation (LFPP)} with parameter $\xi  >0$ is the family of random distance functions $\{D_h^\epsilon\}_{\epsilon >0}$ on the plane obtained by integrating $e^{\xi h_\epsilon}$ along paths, where $h_\epsilon$ for $\epsilon >0$ is a smooth mollification of the planar Gaussian free field. 
Previous work by Ding-Dub\'edat-Dunlap-Falconet and Gwynne-Miller has shown that there is a critical value $\xi_{\mathrm{crit}} > 0$ such that for $\xi < \xi_{\mathrm{crit}}$, LFPP converges under appropriate re-scaling to a random metric on the plane which induces the same topology as the Euclidean metric (the so-called $\gamma$-\emph{Liouville quantum gravity metric} for $\gamma  = \gamma(\xi)\in (0,2)$). 

We show that for all $\xi > 0$, the LFPP metrics are tight with respect to the topology on lower semicontinuous functions. 
For $\xi > \xi_{\mathrm{crit}}$, every possible subsequential limit $D_h$ is a metric on the plane which does \emph{not} induce the Euclidean topology: rather, there is an uncountable, dense, Lebesgue measure-zero set of points $z\in\mathbb C $ such that $D_h(z,w) = \infty$ for every $w\in\mathbb C\setminus \{z\}$. 
We expect that these subsequential limiting metrics are related to Liouville quantum gravity with matter central charge in $(1,25)$. 
\end{abstract}

\tableofcontents

\section{Introduction}
\label{sec-intro}

\subsection{Definition of Liouville first passage percolation}
\label{sec-lfpp}

Let $h$ be the whole-plane Gaussian free field (see, e.g., the expository articles~\cite{shef-gff,berestycki-lqg-notes,pw-gff-notes} for more on the GFF). 
For $t  > 0$ and $z\in\BB C$, we define the heat kernel $p_t(z) := \frac{1}{2\pi t} e^{-|z|^2/2t}$ and we denote its convolution with $h$ by
\eqb \label{eqn-gff-convolve}
h_\ep^*(z) := (h*p_{\ep^2/2})(z) = \int_{\BB C} h(w) p_{\ep^2/2} (z  - w) \, dw^2 ,\quad \forall z\in \BB C  
\eqe
where the integral is interpreted in the sense of distributional pairing. 

For a parameter $\xi > 0$, we define the $\ep$-\emph{Liouville first passage percolation} (LFPP) metric associated with $h$ by
\eqb \label{eqn-gff-lfpp}
D_{h}^\ep(z,w) := \inf_P \int_0^1 e^{\xi h_\ep^*(P(t))} |P'(t)| \,dt ,\quad \forall z,w\in\BB C 
\eqe
where the infimum is over all piecewise continuously differentiable paths $P : [0,1]\rta \BB C$ from $z$ to $w$. We will be interested in (subsequential) limits of the re-normalized metrics $\frk a_\ep^{-1} D_h^\ep$, where the normalizing constant is defined by
\eqb \label{eqn-gff-constant}
\frk a_\ep := \text{median of} \: \inf\left\{ \int_0^1 e^{\xi h_\ep^*(P(t))} |P'(t)| \,dt  : \text{$P$ is a left-right crossing of $[0,1]^2$} \right\}      .
\eqe  
Here, by a left-right crossing of $[0,1]^2$ we mean a piecewise continuously differentiable path $P : [0,1]\rta [0,1]^2$ joining the left and right boundaries of $[0,1]^2$. 

The goal of this paper is to prove that the metrics $\frk a_\ep^{-1} D_h^\ep$ admit subsequential scaling limits when the parameter $\xi$ lies in the supercritical phase.  
The phase transition for LFPP is described in terms of its distance exponent, the existence of which is provided by the following proposition. 

\begin{prop} \label{prop-Q}
For each $\xi  >0$, there exists $Q = Q(\xi) > 0$ such that
\eqbn
\frk a_\ep = \ep^{1-\xi  Q + o_\ep(1)}  ,\quad \text{as $\ep\rta 0$}. 
\eqen
Furthermore, $\xi \mapsto Q(\xi)$ is continuous, strictly decreasing on $(0,0.7)$, non-increasing on $(0,\infty)$, and satisfies $\lim_{\xi \rta \infty} Q(\xi) =0$. 
\end{prop}

We will prove Proposition~\ref{prop-Q} in Section~\ref{sec-Q-wn} (see also the end of Section~\ref{sec-gff-compare}). The existence of $Q(\xi)$ follows from a subadditivity argument, the fact that $Q(\xi) > 0$ follows from~\cite{lfpp-pos}, and the other asserted properties of $Q(\xi)$ follow from results in~\cite{gp-lfpp-bounds}. We remark that Proposition~\ref{prop-Q} in the subcritical phase (see definitions just below) follows from~\cite[Theorem 1.5]{dg-lqg-dim}, so the result is only new in the supercritical phase. 

The value of $Q(\xi)$ is not known explicitly except when $\xi = 1/\sqrt 6$, in which case $Q = 5/\sqrt 6$~\cite{dg-lqg-dim}.\footnote{As per the discussion in Section~\ref{sec-lqg} below, $\xi =1/\sqrt 6$ corresponds to Liouville quantum gravity with parameter $\gamma=\sqrt{8/3}$ (equivalently, matter central charge $\ccM = 0$) and the fact that $Q(1/\sqrt 6) = 5/\sqrt 6$ is a consequence of the fact that $\sqrt{8/3}$-LQG has Hausdorff dimension 4.}
See~\cite{dg-lqg-dim,gp-lfpp-bounds,ang-discrete-lfpp} for bounds\footnote{The bounds in~\cite{dg-lqg-dim,gp-lfpp-bounds,ang-discrete-lfpp} are stated for LFPP defined using slightly different approximations of the GFF from the one defined in~\eqref{eqn-gff-convolve}. However, it is not hard to show using basic comparison lemmas from~\cite{dg-lqg-dim,ang-discrete-lfpp,ding-goswami-watabiki} that the different variants of LFPP have the same distance exponents.}
for $Q(\xi)$. 

We define
\eqb \label{eqn-xi-crit}
\xi_{\op{crit}} := \inf\left\{ \xi > 0 : Q(\xi) = 2 \right\} .
\eqe
The best currently known bounds for $\xi_{\op{crit}}$ come from~\cite[Theorem 2.3]{gp-lfpp-bounds}, which gives
\eqb \label{eqn-xi-bounds}
0.4135 \leq \xi_{\op{crit}} \leq 0.4189 .
\eqe
We do not have a conjecture as to the value of $\xi_{\op{crit}}$ (but see~\cite[Section 1.3]{dg-lqg-dim} for some speculation).  

We call $(0,\xi_{\op{crit}})$ the \emph{subcritical phase} and $(  \xi_{\op{crit}},  \infty  )$ the \emph{supercritical phase}. It was shown in~\cite{dddf-lfpp} that in the subcritical phase $\xi \in (0,\xi_{\op{crit}})$, the re-scaled LFPP metrics $\frk a_\ep^{-1} D_h^\ep$ are tight w.r.t.\ the topology of uniform convergence on compact subsets of $\BB C\times \BB C$. Moreover, every possible subsequential limit is a metric on $\BB C$ which induces the same topology as the Euclidean metric. 
Subsequently, it was shown in~\cite{gm-uniqueness} (building on~\cite{local-metrics,lqg-metric-estimates,gm-confluence}) that the subsequential limiting metric is unique. This limiting metric can be thought of as the Riemannian distance function associated with a so-called \emph{Liouville quantum gravity surface} with matter central charge $\ccM  = 25-6Q^2 \in (-\infty,1)$, or equivalently with coupling constant $\gamma \in (0,2)$ satisfying $Q = 2/\gamma + \gamma/2$. See Section~\ref{sec-lqg} for further discussion. 

The main results of this paper, stated just below, give the tightness of $\frk a_\ep^{-1} D_h^\ep$ for all $\xi > 0$ and some basic properties of the subsequential limiting metrics. 
In the supercritical phase $\xi  > \xi_{\op{crit}}$, the subsequential limiting metric $D_h$ does \emph{not} induce the same topology as the Euclidean metric. Rather, there is an uncountable, dense, but zero-Lebesgue measure set of ``singular points" which lie at infinite $D_h$-distance from every other point. 
As we will explain in Section~\ref{sec-lqg}, we expect that in the supercritical phase $D_h$ is closely related to Liouville quantum gravity with matter central charge $\ccM  = 25-6Q^2 \in (1,25)$.  
For $\xi = \xi_{\op{crit}}$, which corresponds to $\gamma$-Liouville quantum gravity with $\gamma=2$, we expect (but do not prove) that the subsequential limiting metric induces the same topology as the Euclidean metric.

\bigskip

\noindent \textbf{Acknowledgments.} We thank an anonymous referee for helpful comments on an earlier version of this article. We thank Jason Miller and Josh Pfeffer for helpful discussions and Jason Miller for sharing his code for simulating LFPP with us. J.D.\ was partially supported by NSF grant DMS-1757479. E.G.\ was supported by a Clay research fellowship and a Trinity college, Cambridge junior research fellowship.

\subsection{Main results}
\label{sec-main-results}

Since we do not expect that the limit of $\frk a_\ep^{-1} D_h^\ep$ is a continuous function on $\BB C\times\BB C$ when $\xi  > \xi_{\op{crit}} $, we cannot expect tightness of these metrics with respect to the local uniform topology. 
Instead, we will show tightness with respect to a slight modification of the topology on lower semicontinuous functions on $\BB C \times \BB C$ introduced by Beer in~\cite{beer-usc} (actually, Beer treats the case of upper semicontinuous functions but everything works the same for lower semicontinuous functions by symmetry). Under this topology, a sequence of lower semicontinuous functions $f_n : \BB C\times \BB C \rta \BB R \cup \{\pm\infty\}$ converges to a lower semicontinuous function $f$ if and only if
\begin{enumerate}[(A)]
\item If $\{(z_n,w_n)\}_{n\in\BB N}$ is a sequence of points in $\BB C\times \BB C$ such that $(z_n,w_n)\rta (z,w)$, then $\liminf_{n\rta\infty} f_n(z_n,w_n) \geq f(z,w)$. \label{item-semicont-liminf} 
\item For each $(z,w)\in\BB  C\times \BB C$, there is a sequence $(z_n,w_n) \rta (z,w)$ such that $\lim_{n\rta\infty} f_n(z_n,w_n) = f(z,w)$.  \label{item-semicont-lim}
\end{enumerate}
It is easily verified that if $f_n \rta f$ in the above sense and each $f_n$ is lower semicontinuous, then $f$ is also lower semicontinuous.  

It follows from~\cite[Lemma 1.5]{beer-usc} that the above topology is the same as the one induced by the metric $\BB d_{\op{lsc}}$ defined as follows.   
Let $\phi : \BB R \rta (0,1)$ be an increasing homeomorphism. Set $\phi(-\infty) =0$ and $\phi(\infty) = 1$.  
We endow $\BB R \cup \{\pm \infty\}$ with the metric $d_\phi(s,t) = |\phi(s) -\phi(t)|$, so that $\BB R\cup \{\pm\infty\}$ is homeomorphic to $[0,1]$. 
Let $\mcl K$ be the space of compact subsets of $\BB C\times  \BB C \times ( \BB R \cup \{\pm \infty\} )$ equipped with the Hausdorff distance $\BB d_{\op{Haus}}$ induced by the product of the Euclidean metric on $\BB C\times \BB C$ and the metric $d_\phi$ on $\BB R \cup \{\pm \infty\}$. 
If $f :\BB C\times \BB C \rta \BB R \cup \{\pm\infty\}$ is lower semicontinuous, then the ``overgraph" 
\eqbn
U(f)  := \left\{ (z,w,t) \in \BB C\times  \BB C \times ( \BB R \cup \{\pm \infty\} ) : f(z,w) \leq t\right\} 
\eqen
is closed, so for each $r > 0$ the set $U_r(f) := U(f) \cap \ol{B_r(0) \times  B_r(0) \times  ( \BB R \cup \{\pm \infty\} )}$ is compact. 
We then set
\eqbn
\BB d_{\op{lsc}}(f,g) := \int_0^\infty e^{-r}\left( \BB d_{\op{Haus}}(U_r(f) , U_r(g)) \wedge 1 \right) \,dr .
\eqen

\begin{thm} \label{thm-lfpp-tight}
Let $\xi  >0$. For every sequence of positive $\ep$-values tending to zero, there is a subsequence $\mcl E$ for which the following is true. 
\begin{enumerate}
\item $\frk a_\ep^{-1} D_h^\ep$ converges in law to a lower semicontinuous function $D_h : \BB C\times \BB C\rta [0,\infty]$ w.r.t.\ the above topology as $\ep\rta 0$ along $\mcl E$.  \label{item-tight}
\item Each possible subsequential limit $D_h$ is a metric on $\BB C$, except that pairs of points are allowed to have infinite distance from each other. \label{item-metric}
\item $\{ \frk a_\ep^{-1} D_h^\ep(u,v)\}_{u,v\in\BB Q^2} \rta \{ D_h(u,v)\}_{u,v\in\BB Q^2}$ in law as $\ep\rta 0$ along $\mcl E$, jointly with the convergence $\frk a_\ep^{-1} D_h^\ep \rta D_h$. \label{item-rational}
\item For each rational $r>0$, the limit $\frk c_r := \lim_{\mcl E\ni\ep} r \frk a_{\ep/r} / \frk a_\ep$ exists and satisfies $\frk c_r = r^{\xi Q + o_r(1)}$ as $r\rta 0$ or $r\rta\infty$. \label{item-constant}
\end{enumerate} 
\end{thm}

We will also establish a number of properties of the subsequential limiting metric $D_h$ of Theorem~\ref{thm-lfpp-tight}. 
Let us first note that (by the Prokhorov theorem) after possibly passing to a further subsequence of the one in Theorem~\ref{thm-lfpp-tight} we can arrange that the joint law of $(h, \frk a_\ep^{-1} D_h^\ep)$ converges to a coupling $(h, D_h)$ (here the first coordinate is given the distributional topology). 
Following~\cite{hmp-thick-pts}, for $\alpha\geq 0$ we say that $z\in\BB C$ is an \emph{$\alpha$-thick point} of $h$ if
\eqb \label{eqn-thick-pt}
\liminf_{r \rta 0} \frac{h_r(z)}{\log r}  \geq \alpha ,
\eqe
where here $h_r(z)$ is the average over $h$ over the circle of radius $r$ centered at $z$. It is shown in~\cite{hmp-thick-pts} that for $\alpha \in (0,2)$, a.s.\ the set of $\alpha$-thick points has Hausdorff dimension $2-\alpha^2/2$ and for $\alpha > 2$, a.s.\ the set of $\alpha$-thick points is empty. 

\begin{thm} \label{thm-ssl-properties}
Let $(h,D_h)$ be a subsequential limiting coupling of $h$ with a random metric as in Theorem~\ref{thm-lfpp-tight}. Almost surely, the following is true.
\begin{enumerate}
\item $D_h(z,w) < \infty$ for Lebesgue-a.e.\ $(z,w) \in\BB C\times \BB C$. \label{item-finite-ae}
\item Every $D_h$-bounded subset of $\BB C$ is also Euclidean-bounded. \label{item-bounded}
\item The identity map from $\BB C$, equipped with the metric $D_h$ to $\BB C$, equipped with the Euclidean metric is locally H\"older continuous with any exponent less than $[\xi(Q+2)]^{-1}$. If $\xi  > \xi_{\op{crit}} $, the inverse of this map is not continuous. \label{item-holder} 
\item Say that $z\in\BB C$ is a \emph{singular point} for $D_h$ if $D_h(z,w) = \infty$ for every $w\in\BB C\setminus\{z\}$. 
Then $D_h$ is a complete, finite-valued metric on $\BB C\setminus \{ \text{singular points} \}$. \label{item-singularity}
\item Any two non-singular points $z,w\in\BB C$ can be joined by a $D_h$-geodesic (i.e., a  path of $D_h$-length exactly $D_h(z,w)$).   \label{item-geodesic}
\item If $\xi > \xi_{\op{crit}}  $ and $\alpha > Q$, then a.s.\ each $\alpha$-thick point $z$ of $h$ is a singular point, i.e., it satisfies $D_h(z,w) =\infty$ for every $w\in\BB C\setminus \{z\}$. \label{item-thick-pt}
\end{enumerate}
\end{thm}

Assertion~\ref{item-holder} should be compared to~\cite[Theorem 1.7]{dddf-lfpp}, which shows that in the subcritical phase the identity map from $(\BB C , D_h)$ to $(\BB C , |\cdot|)$ is locally H\"older continuous with any exponent less than $[\xi(Q+2)]^{-1}$ and the inverse of this map is H\"older continuous with any exponent less than $\xi(Q-2)$. The latter H\"older exponent goes to zero as $\xi \rta \xi_{\op{crit}}$, so it is natural that the inverse map is not continuous for $\xi > \xi_{\op{crit}}$. 

Assertion~\ref{item-thick-pt} implies that in the supercritical phase, the set of singular points for $D_h$ uncountable and dense. We can visualize these singular points as infinite ``spikes". However, $D_h$-distances between typical points are still finite by assertion~\ref{item-finite-ae}. This is because two typical points are joined by a $D_h$-geodesic which avoids the singular points.  Assertion~\ref{item-thick-pt} has several interesting consequences, for example the following. 
\begin{itemize}
\item $(\BB C,D_h)$ has infinitely many ``ends" in the sense that the complement of a $D_h$-metric ball centered at a typical point has infinitely many connected components of infinite $D_h$-diameter (a proof of this assertion appears in the subsequent paper~\cite{pfeffer-supercritical-lqg}). 
\item A $D_h$-metric ball cannot contain a Euclidean-open set (since every Euclidean-open set contains a singular point). 
\item The restriction of $D_h$ to $ \BB C\setminus \{\text{singular points}\}$ does not induce the Euclidean topology. This is because any $z\in\BB C$ can be expressed as a Euclidean limit of points $z_n$ which are not themselves singular points, but which are close enough to singular points so that $D_h(z_n,z) \rta \infty$.  
\item $(\BB C \setminus \{\text{singular points}\} , D_h)$ is not locally simply connected. This is because any Jordan loop in $\BB C$ surrounds a singular point, so cannot be $D_h$-continuously contracted to a point. In particular, $(\BB C \setminus \{\text{singular points}\}  , D_h)$ is not a topological manifold.
\end{itemize} 

As a partial converse to assertion~\ref{item-thick-pt}, one can show that points $z$ such that $\limsup_{r \rta 0} h_r(z) / \log r < Q$ are not singular points; see~\cite{pfeffer-supercritical-lqg}. We do not know if points of thickness exactly $Q$ are singular points, and we expect that determining this would require much more delicate estimates than the ones in the present paper. 
 
See Figure~\ref{fig-sim} for a simulation of supercritical LFPP metric balls.

\begin{figure}[ht!]
\begin{center}

\includegraphics[width=0.45\textwidth]{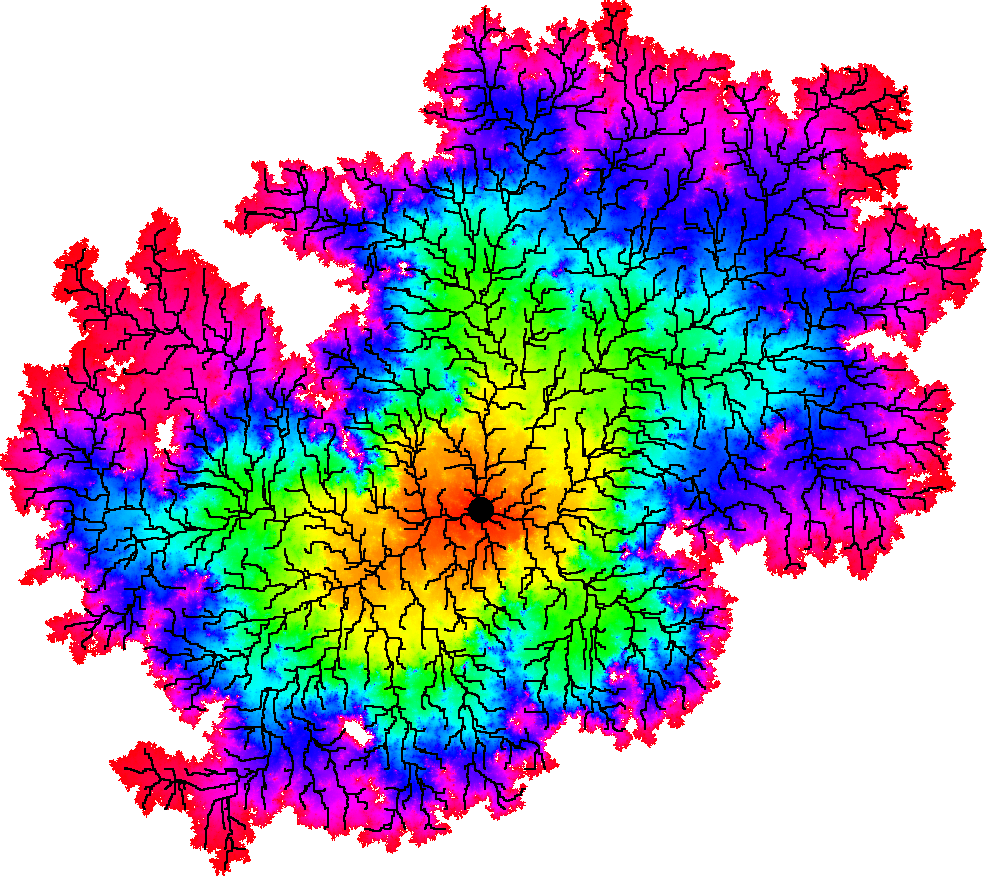} \hspace{0.05\textwidth}
\includegraphics[width=0.45\textwidth]{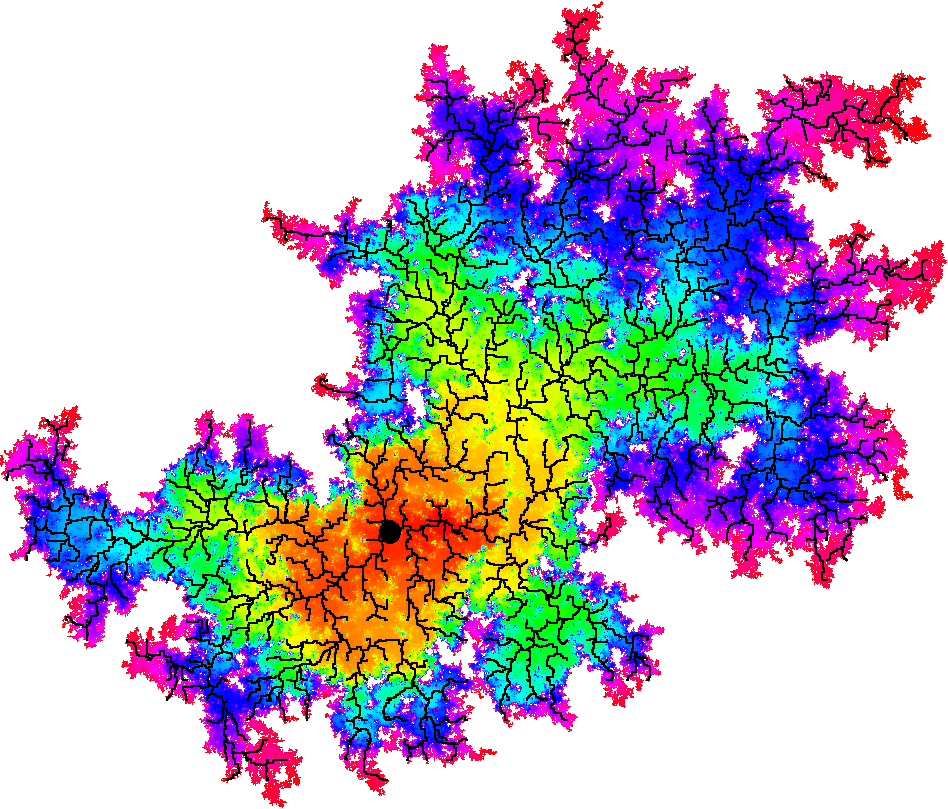}
\vspace{0.002\textheight}
\includegraphics[width=0.45\textwidth]{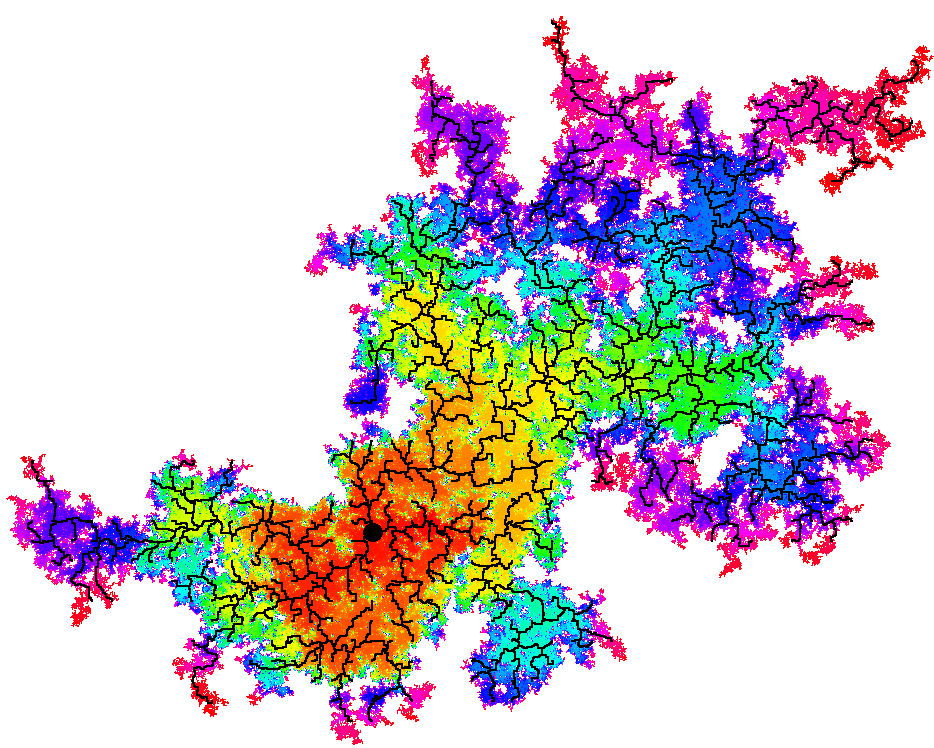} \hspace{0.05\textwidth}
\includegraphics[width=0.45\textwidth]{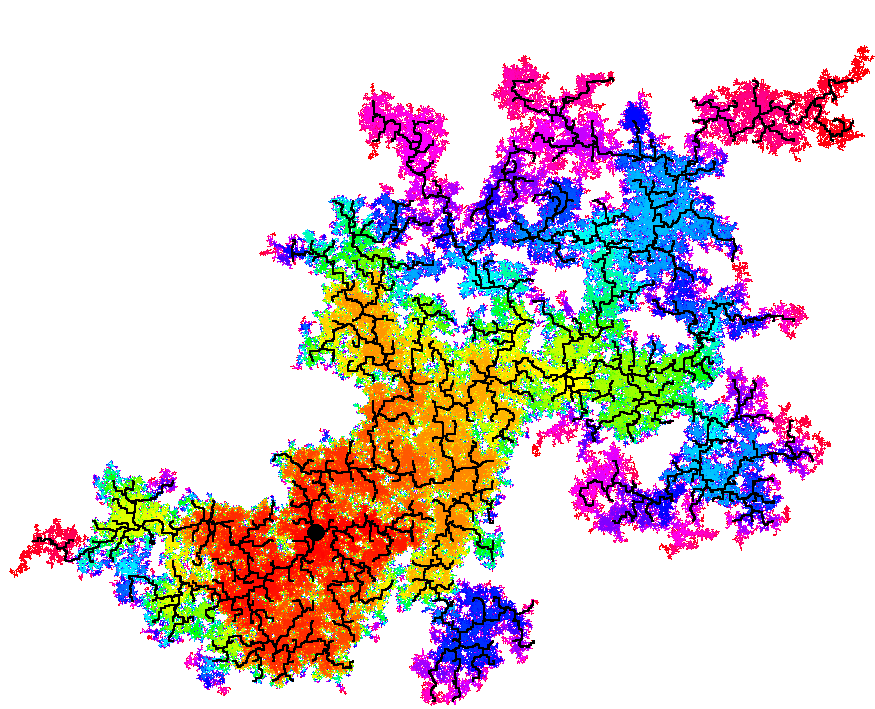} 

\end{center}
\vspace{-0.03\textheight}
\caption{\label{fig-sim} 
Simulations of LFPP metric balls for $\xi = 0.5$ (top left), $\xi = 0.9$ (top right), $\xi = 1.3$ (bottom left), and $\xi = 1.7$ (bottom right), produced from the same GFF instance. The colors indicate distance to the center point (marked with a black dot) and the black curves are geodesics from the center point to other points in the ball. 
Note that all of these values of $\xi$ are supercritical by~\eqref{eqn-xi-bounds}. In particular, the metric balls for the subsequential limiting metrics have empty Euclidean interior (despite the appearance of the figures). 
The simulations were produced using LFPP w.r.t.\ a discrete GFF on a $1024 \times 1024$ subset of $\BB Z^2$.  It is believed that this variant of LFPP falls into the same universality class as the variant in~\eqref{eqn-gff-convolve}. The geodesics go from the center of the metric ball to points in the intersection of the metric ball with the grid $20\BB Z^2$. The code for the simulations was provided by J.\ Miller.
}
\end{figure}

In the supercritical phase, $D_h$ satisfies many properties which are either similar to the properties of the limit of subcritical LFPP which were established in~\cite{dddf-lfpp,lqg-metric-estimates,gm-confluence,gm-uniqueness,gm-coord-change,gp-kpz} or are related to properties of LQG with $\ccM \in (1,25)$ which are discussed in~\cite{ghpr-central-charge}. Several such properties are established in the subsequent works~\cite{pfeffer-supercritical-lqg,dg-confluence}: 
\begin{itemize}
\item \textbf{Measurability:} $D_h$ is a.s.\ given by a measurable function of $h$ (c.f.~\cite[Lemma 2.20]{lqg-metric-estimates}). 
\item \textbf{Weyl scaling:} adding a continuous function $f$ to $h$ corresponds to scaling the $D_h$-length of each path by a factor of $e^{\xi f}$ (c.f.~\cite[Lemma 2.12]{lqg-metric-estimates}).  
\item \textbf{Moments:} for any fixed $z,w\in\BB C$, $D_h(z,w)$ has finite moments up to order $2Q/\xi$. More generally, if $\alpha,\beta\in (-\infty,Q)$ and $h^{\alpha,\beta} = h -\alpha\log|\cdot-z|  - \beta\log|\cdot-w|$, then $D_{h^{\alpha,\beta}}(z,w)$ has finite moments up to order $\frac{2 }{\xi}( Q - \max\{\alpha,\beta\})$ (c.f.~\cite[Theorem 1.11]{lqg-metric-estimates}). 
\item \textbf{Confluence of geodesics:} two $D_h$-geodesics with the same starting point and different target points typically coincide for a non-trivial initial time interval (c.f.~\cite{gm-confluence}).
\item \textbf{Hausdorff dimension:} a.s.\ the Hausdorff dimension of $(\BB C , D_h)$ is infinite (c.f.~\cite[Theorem 1.6]{ghpr-central-charge}). 
\item \textbf{KPZ formula:} if $X \subset\BB C$ is a random fractal sampled independently of $h$, then the Hausdorff dimensions of $X$ w.r.t.\ $D_h$ and w.r.t.\ the Euclidean metric are related by the variant of the KPZ formula from~\cite[Theorem 1.5]{ghpr-central-charge} (note that this formula gives that the Hausdorff dimension of $X$ w.r.t.\ $D_h$ is infinite if the dimension of $X$ w.r.t.\ the Euclidean metric is sufficiently close to 2). 
\end{itemize}
We also expect $D_h$ to satisfy the following further properties, which have not been proven yet.   
\begin{itemize} 
\item \textbf{Uniqueness:} the metric $D_h$ is uniquely characterized (up to multiplication by a deterministic positive constant) by a list of axioms similar to the list in~\cite{gm-uniqueness}. 
\item \textbf{Coordinate change:} if $a\in\BB C\setminus \{0\}$ and $b\in\BB C$, then a.s.\ $D_h(az+b,aw+b) = D_{h(a\cdot+b) + Q\log|a|}(z,w)$ for all $z,w\in\BB C$ (c.f.~\cite{gm-uniqueness}). More generally, if we extend the definition of $D_h$ to the case when $h$ is a field on a general open domain $U\subset\BB C$ (e.g., via local absolute continuity) then for a conformal map $\phi : \wt U \rta U$ a.s.\ $D_h(\phi(z) , \phi(w)) = D_{h\circ\phi + Q\log|\phi'|}(z,w)$ for all $z,w\in \wt U$ (c.f.~\cite{gm-coord-change}).  
\end{itemize}
It may be possible to prove these last two properties by adapting the arguments of~\cite{gm-uniqueness,gm-coord-change} (which prove analogous results in the subcritical case), but these papers use the fact that the metric induces the Euclidean topology, so non-trivial new ideas would be needed. Alternatively, it would also be of interest to find a completely different proof of these properties. 
 
\begin{remark} \label{remark-gh}
The reader might wonder whether one has the convergence $\frk a_\ep^{-1} D_h^\ep \rta D_h$ in law with respect to some local variant of the Gromov-Hausdorff distance.
We expect that no such convergence statement holds when $\xi  >\xi_{\op{crit}}$.  
The reason for this is that for $\xi > \xi_{\op{crit}}$, $D_h$-metric balls are \emph{not} $D_h$-compact (this is proven in~\cite{pfeffer-supercritical-lqg}). 
\end{remark}

\subsection{Connection to Liouville quantum gravity}
\label{sec-lqg}
 
Liouville quantum gravity (LQG) is a one-parameter family of random surfaces which describe two-dimensional quantum gravity coupled with conformal matter fields. LQG was first introduced by Polyakov~\cite{polyakov-qg1} in order to define a ``sum over Riemannian metrics" in two dimensions, which he was interested in for the purposes of bosonic string theory. 

One way to define LQG is in terms of the so-called \emph{matter central charge} $\ccM \in (-\infty,25)$.  
Let $\mcl D$ be a simply connected topological surface.\footnote{LQG can also be defined for non-simply connected surfaces, but we consider only the simply connected case for simplicity. See~\cite{drv-torus,grv-higher-genus} for works concerning on LQG on non-simply connected surfaces.}
For a Riemannian metric $g$ on $\mcl D$, let $\Delta_g$ be its Laplace-Beltrami operator.  
Heuristically speaking, an LQG surface with matter central charge $\ccM$ is the random two-dimensional Riemannian manifold $(\mcl D  ,g)$ sampled from the ``Lebesgue measure on the space of Riemannian metrics on $\mcl D$ weighted by $(\det \Delta_g)^{-\ccM/2}$". This definition is far from making literal sense, but see~\cite{apps-central-charge} for some progress on interpreting it rigorously. 
In physics, one thinks of an LQG surface as representing ``gravity coupled to matter fields". The parameter $\ccM$ is the central charge of the conformal field theory given by these matter fields, and $(\det\Delta_g)^{-\ccM/2}$ can be thought of as the associated partition function. 

We refer to the case when $\ccM \in (-\infty,1)$ (resp.\ $\ccM \in (1,25)$) as the \emph{subcritical} (resp.\ \emph{supercritical}) phase. As we will see below, these phases correspond to the subcritical and supercritical phases of LFPP. We refer to Figure~\ref{fig-c-phases-table} for a table of the relationship between the parameters in the subcritical and supercritical phases.

\begin{figure}[t!]
 \begin{center}
\includegraphics[scale=1]{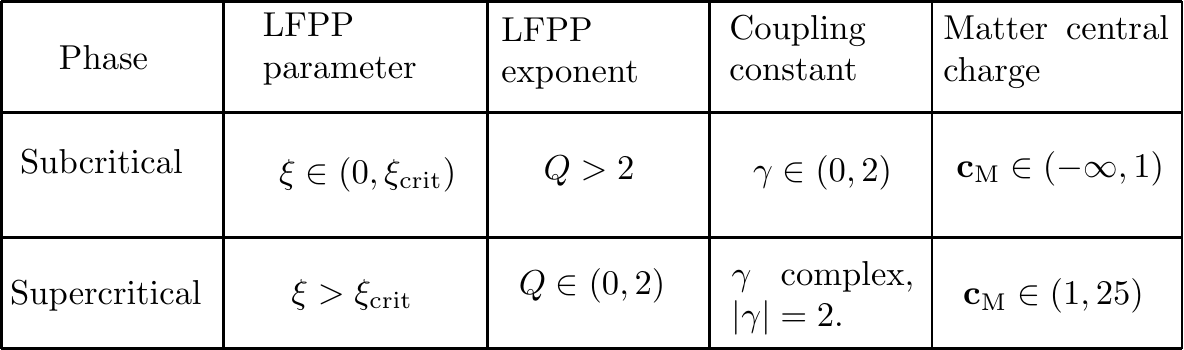}
\vspace{-0.01\textheight}
\caption{Table of the parameter ranges for LQG / LFPP in the subcritical and supercritical phases.
}\label{fig-c-phases-table}
\end{center}
\vspace{-1em}
\end{figure} 

\subsubsection*{LQG metric tensor in the subcritical phase}

The so-called \emph{DDK ansatz}~\cite{david-conformal-gauge,dk-qg} is a heuristic argument which allows us to describe the Riemannian metric tensor of an LQG surface directly in the subcritical (or critical) case, $\ccM \in (-\infty,1]$. The DDK ansatz implies that the Riemmanian metric tensor of an LQG surface with $\ccM \in (-\infty,1]$ should be given by
\eqb \label{eqn-lqg-tensor}
e^{\gamma h} \,(dx^2  + dy^2)
\eqe
where $h$ is a variant of the Gaussian free field, $dx^2+dy^2$ is a fixed smooth metric tensor (e.g., the Euclidean metric tensor if $\mcl D\subset \BB C$), and  the \emph{coupling constant} $\gamma \in (0,2]$ is related to $\ccM$ by
\eqb \label{eqn-c-gamma}
\ccM = 25  - 6\left(\frac{2}{\gamma} + \frac{\gamma}{2} \right)^2.
\eqe

\begin{remark} \label{remark-ccL}
Another way of thinking about the DDK ansatz (which is closer to the original physics phrasing) is that the partition function of LQG can be obtained by integrating $(\det \Delta_g)^{-\ccM/2}$ times the partition function of \emph{Liouville conformal field theory} (LCFT) over the moduli space of $\mcl D$. The central charge of the LCFT corresponding to LQG with matter central charge $\ccM$ is $\ccL = 26 - \ccM$. See~\cite{dkrv-lqg-sphere,drv-torus,hrv-disk,grv-higher-genus} for rigorous constructions of LCFT on various surfaces.
\end{remark}

The Riemannian metric tensor~\eqref{eqn-lqg-tensor} does not make literal sense since $h$ is a distribution (generalized function), not an actual function. 
However, one can define various objects associated with an LQG surface rigorously via regularization procedures. For example, one can construct the volume form, a.k.a.\ the LQG area measure ``$\mu_h = e^{\gamma h} \,dx\,dy$" (where $dx\,dy$ denotes Lebesgue measure)~\cite{kahane,shef-kpz,rhodes-vargas-log-kpz} (see~\cite{shef-deriv-mart,shef-renormalization} for the critical case $\ccM=1$, $\gamma=2$). Similarly, one can construct a natural diffusion on LQG, called Liouville Brownian motion~\cite{berestycki-lbm,grv-lbm} (see~\cite{rhodes-vargas-criticality} for the critical case). 

As explained in~\cite{dg-lqg-dim}, for $\gamma \in (0,2]$ a natural way to approximate the Riemannian distance function associated with~\eqref{eqn-lqg-tensor} is to consider LFPP with parameter $\xi = \gamma/d_\gamma$, where $d_\gamma > 2$ is the dimension exponent associated with $\gamma$-LQG, as defined in~\cite{dzz-heat-kernel,dg-lqg-dim}. It is shown in~\cite{gp-kpz} that $d_\gamma$ is the Hausdorff dimension of a $\gamma$-LQG surface, viewed as a metric space.  
For this choice of $\xi$, one has~\cite[Theorem 1.5]{dg-lqg-dim}
\eqb \label{eqn-Q-gamma}
Q = Q(\xi) = 2/\gamma + \gamma/2 .
\eqe 
It is shown in~\cite[Proposition 1.7]{dg-lqg-dim} that $\gamma/d_\gamma$ is strictly increasing in $\gamma$, which means that $\gamma \in (0,2)$ corresponds exactly to $\xi \in (0,\xi_{\op{crit}})$ and
\eqb
\xi_{\op{crit}} = \frac{2}{d_2} . 
\eqe
It is shown in~\cite{dddf-lfpp,gm-uniqueness} that subcritical LFPP converges to a random metric on $\BB C$ which can be interpreted as the Riemannian distance function associated with LQG for $\ccM  \in (-\infty,1)$. 

\subsubsection*{LQG in the supercritical phase}

LQG in the supercritcal phase $\ccM \in (1,25)$ is much less well-understood than in the subcritical and critical cases. 
Part of the reason for this is that when $\ccM \in (1,25)$, the coupling constant $\gamma$ in~\eqref{eqn-c-gamma} is complex with modulus 2. 
Consequently, analytic continuations of certain formulas from the case $\ccM < 1$ to the case $\ccM \in (1,25)$ (such as the KPZ formula~\cite{kpz-scaling,shef-kpz} or predictions for the Hausdorff dimension of LQG~\cite{watabiki-lqg,dg-lqg-dim}) yield nonsensical complex answers.
Moreover, it is not clear whether there is any natural notion of ``volume form" or ``diffusion" associated with supercritical LQG. 
The recent paper~\cite{jsw-decompositions} shows how to make sense of random distributions of the form ``$e^{\gamma h} \,dx\,dy$" for complex $\gamma$, but $|\gamma|=2$ falls outside the feasible region for the techniques of that paper. See~\cite{ghpr-central-charge,apps-central-charge} for further discussion and references concerning supercritical LQG. 

However, it is expected that there is a notion of Riemannian distance function (metric) associated with supercritical LQG. 
The paper~\cite{ghpr-central-charge} provides one possible approximation procedure for such a metric, based on a family of random tilings $\{\mcl S_h^\ep\}_{\ep > 0}$ of the plane by dyadic squares constructed from the GFF, depending on the central charge $\ccM$. The collection of squares $\{\mcl S_h^\ep\}_{\ep > 0}$ is a.s.\ locally finite for $\ccM \in (-\infty, 1)$. In contrast, for $\ccM \in (1,25)$ there is an uncountable, zero-Lebesgue measure set of ``singular points" $z\in \BB C$ such that every neighborhood of $z$ contains infinitely many small squares of $\mcl S_h^\ep$ (these singular points are analogous to the points which are at infinite $D_h$-distance from every other point in the setting of Theorem~\ref{thm-ssl-properties}). It is conjectured in~\cite{ghpr-central-charge} that the graph distance in the adjacency graph of squares of $\mcl S_h^\ep$, suitably re-scaled, converges to the metric associated with LQG for all $\ccM \in (-\infty,25)$. 

Another possible approximation procedure is supercritical LFPP. 
Indeed, if $\gamma$ and $\ccM$ are related by~\eqref{eqn-c-gamma} then for $\ccM \in (1,25)$ the parameter $Q$ from~\eqref{eqn-Q-gamma} lies in $(0,2)$. 
Hence, analytically continuing the relationship between LFPP and LQG to the supercritical phase shows that LQG for $\ccM \in (1,25)$ should correspond to LFPP with $\xi  > \xi_{\op{crit}} $. More precisely,~\eqref{eqn-c-gamma} and~\eqref{eqn-Q-gamma} suggest that $\xi$ and $\ccM$ should be related by
\eqb \label{eqn-xi-c}
\ccM = 25 - 6 Q(\xi)^2 .
\eqe
Further evidence of this relationship comes from the dyadic tiling model of~\cite{ghpr-central-charge}. As discussed in~\cite[Section 2.3]{ghpr-central-charge}, it is expected that if $\mcl S_h^\ep$ is the dyadic tiling above for a given value of $\ccM$, then the $\mcl S_h^\ep$-graph distance between the squares containing two typical points of $\BB C$ is of order $\ep^{-\xi + o (1)}$, where $\xi$ is as in~\eqref{eqn-xi-c}. 
This is analogous to the relationship between LFPP and Liouville graph distance in the subcritical phase, which was established in~\cite[Theorem 1.5]{dg-lqg-dim}.
Consequently, in the supercritical phase the subsequential limiting metrics $D_h$ of Theorem~\ref{thm-lfpp-tight} are candidates for the distance function associated with LQG with $\ccM \in (1,25)$.

\begin{remark} \label{remark-tree}
Many works~\cite{cates-branched-polymer,david-c>1-barrier,adjt-c-ge1,ckr-c-ge1,bh-c-ge1-matrix,dfj-critical-behavior,bj-potts-sim,adf-critical-dimensions} have suggested that LQG surfaces with $\ccM > 1$ should correspond to ``branched polymers", i.e., they should look like continuum random trees. 
At a first glance, this seems to be incompatible with the results of this paper since the metric $D_h$ of Theorem~\ref{thm-lfpp-tight} is not tree-like. 
However, as explained in~\cite[Section 2.2]{ghpr-central-charge}, the tree-like behavior should only arise when the surfaces are in some sense conditioned to have finite volume.
In our setting, we are not imposing any constraints which force the LQG surface to have finite volume so we get a non-trivial metric structure. We expect that similar statements hold if we condition the surfaces to have finite diameter instead of finite volume.
\end{remark}

\begin{remark} \label{remark-atomic} 
We do \emph{not} expect that LQG with $\ccM   \in (1,25)$ and supercritical LFPP are related to the purely atomic LQG measures for $\gamma  > 2$ considered in~\cite{dup-dual-lqg,bjrv-gmt-duality,rhodes-vargas-review,wedges}. 
Indeed, by~\eqref{eqn-c-gamma} the matter central charge corresponding to $\gamma  >2$ is the same as the matter central charge corresponding to the dual parameter $\gamma' = 4/\gamma \in (0,2)$, so lies in $(-\infty,1)$. Matter central charge in $(1,25)$ corresponds to a complex value of $\gamma$. 
\end{remark}

\subsection{Outline}
\label{sec-outline}

Here we will give a rough outline of the content of the rest of the paper. More detailed outlines of the more involved arguments in the paper can be found at the beginnings of their respective sections and subsections. We will also comment on the similarities and differences between the arguments in this paper and those in~\cite{ding-dunlap-lqg-fpp,df-lqg-metric,ding-dunlap-lgd,dddf-lfpp}, which prove tightness for various approximations of LQG distances in the subcritical phase.

In \textbf{Section~\ref{sec-prelim}}, we fix some notation, then introduce a variant of LFPP based on the white noise decomposition of the GFF which we will work with throughout much of the paper. In particular, we let $W$ be a space-time white noise on $\BB C\times \BB R$ and for $m,n\in\BB N$ with $m < n$ we let $\Phi_{m,n}(z) := \sqrt\pi \int_{2^{-2n}}^{2^{-2m}} \int_{\BB C} p_{t/2}(z-w) W(dw,dt)$ where $p_{t/2}$ is the heat kernel. 
We let $D_{m,n}$ be the LFPP metric associated with $\Phi_{m,n}$, i.e., it is defined as in~\eqref{eqn-gff-lfpp} but with $\Phi_{m,n}$ in place of $h_\ep^*$. 
As we will see in Section~\ref{sec-pt-tight}, $D_{0,n}$ is a good approximation for $D_h^\ep$. 
However, due to the exact spatial independence properties of the white noise, $D_{0,n}$ is sometimes easier to work with than $D_h^\ep$. 

We will also record several estimates for $D_{0,n}$ which were proven in~\cite{dddf-lfpp} (for general values of $\xi$). In particular, this includes a concentration bound for $D_{0,n}$-distances between sets (Proposition~\ref{prop-perc-estimate}) which will play a crucial role in our arguments. We note that, unlike in other works proving tightness results for approximations of LQG distances, we do not need to prove an RSW estimate or any a priori estimates for distances across rectangles. The reason is that we can re-use the relevant estimates from~\cite{dddf-lfpp}, which were proven for general $\xi  >0$. Finally, we will establish a variant of Proposition~\ref{prop-Q} for $D_{0,n}$ (Proposition~\ref{prop-Q-wn}) using a subadditivity argument. Proposition~\ref{prop-Q} will be deduced from this variant in Section~\ref{sec-gff-compare}. 

In \textbf{Section~\ref{sec-efron-stein}}, we will establish a concentration result for the log left-right crossing distance $\log L_{0,n}$, where
\eqbn
L_{0,n} := \inf\left\{ \int_0^1 e^{\xi \Phi_{0,n}(P(t))} |P'(t)| \,dt  : \text{$P$ is a left-right crossing of $[0,1]^2$} \right\}
\eqen
(c.f.~\eqref{eqn-gff-constant}). 
This is the most technically involved part of the and will be done using an inductive argument based on the Efron-Stein inequality. 
See Section~\ref{sec-ef-outline} for a detailed outline of the argument. 

The papers~\cite{ding-dunlap-lqg-fpp,df-lqg-metric,ding-dunlap-lgd,dddf-lfpp} also use the Efron-Stein inequality and induction to prove similar concentration statements (in the subcritical phase). But, the arguments in these papers are otherwise very different from the ones in the present paper.
Here, we will briefly explain the differences between our argument and the one in~\cite{dddf-lfpp}. 
 
We want to prove a variance bound for $\log L_{0,n}$ by induction on $n$. To do this, we fix a large integer $K$ (which is independent from $n$) and assume that we have proven a variance bound for $\log L_{0,n-K}$. Let $\mcl S_K$ be the set of $2^{-K} \times 2^{-K}$ dyadic squares $S \subset [0,1]^2$. 
In~\cite{dddf-lfpp}, the authors use the Efron-Stein inequality to reduce to bounding $\sum_{S\in\mcl S_K} \BB E[Z_S^2]$, where $Z_S$ is the change in $\log L_{0,n}$ when we re-sample the restriction of the field to the square $S$, leaving the other squares fixed. 
A necessary condition in order to bound this sum is as follows. If $P  :[0,L_{0,n}] \rta [0,1]^2$ is a $D_{0,n}$-geodesic between the left and right boundaries of $[0,1]^2$, then the maximum over all $S\in\mcl S_K$ of the time that $P$ spends in $S$ is negligible in comparison to $L_{0,n}$.
In~\cite{dddf-lfpp}, this condition is achieved by a crude upper bound on the maximum of the field $\Phi_{0,K}$ (see~\cite[Proposition 21]{dddf-lfpp}). 

In the supercritical case, the bound for the maximum of $\Phi_{0,K}$ is not sufficient: indeed, there will typically be squares $S\in\mcl S_K$ such that the $D_{0,n}$-distance between the inner and outer boundaries of the annular region $B_{2^{-K}}(S) \setminus S$ (here $B_{2^{-K}}(S)$ is the Euclidean $2^{-K}$-neighborhood of $S$) is larger than $L_{0,n}$. 
In order to get around this difficulty, instead of applying the Efron-Stein inequality to bound $\op{Var}\left[ \log L_{0,n} \right]$ directly, we will instead apply it to bound $\op{Var}\left[ \BB E\left[ \log L_{0,n} \,|\, \Phi_{K,n}\right] \right]$ (actually, for technical reasons we will work with a slightly modified version of $\Phi_{0,n}$).
We will bound $\BB E\left[ \op{Var}\left[ \log L_{0,n} \,|\, \Phi_{K,n} \right] \right]$ separately using a Gaussian concentration inequality, then combine the bounds to get our needed bound for $\op{Var}\left[ \log L_{0,n}\right]$. 

For our application of the Efron-Stein inequality, we will show that the contribution to $\BB E\left[ \log L_{0,n} \,|\, \Phi_{K,n}\right]$ from each square $S$ is negligible in comparison to $\BB E\left[ \log L_{0,n} \,|\, \Phi_{K,n}\right]$. This can be done because, even though there will typically be \emph{some} squares $S$ for which the $D_{0,n}$-distance across $B_{2^{-K}}(S) \setminus S$ is larger than $L_{0,n}$, for a fixed square $S$ the $D_{0,n}$-distance across $B_{2^{-K}}(S) \setminus S$ is typically much smaller than $L_{0,n}$. We can show (using the our inductive hypothesis) that this continues to be the case even if we condition on $\Phi_{K,n}$. 
This allows us to get much better bounds than we would get by just looking at the maximum of the coarse field, but requires us to argue in a quite different manner from~\cite{dddf-lfpp}. 

Thanks to the concentration result for $D_{0,n}$-distances between sets from~\cite{dddf-lfpp} (Proposition~\ref{prop-perc-estimate}), once the concentration of $\log L_{0,n}$ is established, we immediately obtain strong up-to-constants tail bounds for $D_{0,n}$-distances between general compact sets (see Proposition~\ref{prop-perc-estimate'}). The purpose of Sections~\ref{sec-pt-tight} and~\ref{sec-lfpp-tight} is to use these tail bounds together with various ``concatenating paths" arguments to establish our main results, Theorems~\ref{thm-lfpp-tight} and~\ref{thm-ssl-properties}. The arguments in these sections are more standard than the ones in Section~\ref{sec-efron-stein}, but some care is still needed due to the existence of singular points (which makes it impossible to get uniform convergence for our metrics).  

In \textbf{Section~\ref{sec-pt-tight}}, we will use the concentration of $\log L_{0,n}$ to establish the tightness of $D_{0,n}(z,w)$ (appropriately re-scaled) for fixed points $z,w\in\BB C$. This is done by summing estimates for the $D_{0,n}$-distances between non-trivial connected sets over dyadic scales surrounding each of $z$ and $w$. We will then transfer our estimates for $D_{0,n}$ to estimates for $D_h^\ep$ using comparison lemmas for different approximations of the GFF. We will work with $D_h^\ep$ for the rest of the paper.  

In \textbf{Section~\ref{sec-lfpp-tight}}, we will consider a sequence of $\ep$-values tending to zero along which a certain countable collection of functionals of the re-scaled LFPP metrics $\frk a_\ep^{-1} D_h^\ep$ converge jointly. This will include the $\frk a_\ep^{-1} D_h^\ep$-distances between rational points as well as $\frk a_\ep^{-1}$-distances across and around annuli whose boundaries are circles with rational radii centered at rational points. We will then use these functionals to construct a metric $D_h$ which satisfies the conditions of Theorems~\ref{thm-lfpp-tight} and Theorem~\ref{thm-ssl-properties}. The proofs in this section have essentially no similarities to the arguments in~\cite{ding-dunlap-lqg-fpp,df-lqg-metric,ding-dunlap-lgd,dddf-lfpp}. This is because we are showing convergence with respect to the topology on lower semi-continuous functions rather than the uniform topology; and because our limiting metric can take on infinite values. 

Appendix~\ref{sec-appendix} contains some basic estimates for Gaussian processes which are needed for our proofs.

\section{Preliminaries}
\label{sec-prelim}

\subsection{Basic notation}
\label{sec-basic}

\subsubsection*{Integers and asymptotics}

\noindent
We write $\BB N = \{1,2,3,\dots\}$ and $\BB N_0 = \BB N \cup \{0\}$.
\medskip

\noindent
For $a < b$, we define the discrete interval $[a,b]_{\BB Z}:= [a,b]\cap\BB Z$.
\medskip

\noindent
If $f  :(0,\infty) \rta \BB R$ and $g : (0,\infty) \rta (0,\infty)$, we say that $f(\ep) = O_\ep(g(\ep))$ (resp.\ $f(\ep) = o_\ep(g(\ep))$) as $\ep\rta 0$ if $f(\ep)/g(\ep)$ remains bounded (resp.\ tends to zero) as $\ep\rta 0$. We similarly define $O(\cdot)$ and $o(\cdot)$ errors as a parameter goes to infinity.
\medskip

\noindent
If $f,g : (0,\infty) \rta [0,\infty)$, we say that $f(\ep) \preceq g(\ep)$ if there is a constant $C>0$ (independent from $\ep$ and possibly from other parameters of interest) such that $f(\ep) \leq  C g(\ep)$. We write $f(\ep) \asymp g(\ep)$ if $f(\ep) \preceq g(\ep)$ and $g(\ep) \preceq f(\ep)$.
\medskip
 
\noindent
We will often specify any requirements on the dependencies on rates of convergence in $O(\cdot)$ and $o(\cdot)$ errors, implicit constants in $\preceq$, etc., in the statements of lemmas/propositions/theorems, in which case we implicitly require that errors, implicit constants, etc., appearing in the proof satisfy the same dependencies.

\subsubsection*{Metrics}

\noindent
Let $(X,D)$ be a metric space.
\medskip

\noindent
For a curve $P : [a,b] \rta X$, the \emph{$D$-length} of $P$ is defined by 
\eqbn
\op{len}\left( P ; D  \right) := \sup_{T} \sum_{i=1}^{\# T} D(P(t_i) , P(t_{i-1})) 
\eqen
where the supremum is over all partitions $T : a= t_0 < \dots < t_{\# T} = b$ of $[a,b]$. Note that the $D$-length of a curve may be infinite.
\medskip

\noindent
For $Y\subset X$, the \emph{internal metric of $D$ on $Y$} is defined by
\eqb \label{eqn-internal-def}
D(x,y ; Y)  := \inf_{P \subset Y} \op{len}\left(P ; D \right) ,\quad \forall x,y\in Y 
\eqe 
where the infimum is over all paths $P$ in $Y$ from $x$ to $y$. 
Then $D(\cdot,\cdot ; Y)$ is a metric on $Y$, except that it is allowed to take infinite values.  
\medskip
 
\noindent
We say that $D$ is a \emph{length metric} if for each $x,y\in X$ and each $\delta > 0$, there exists a curve of $D$-length at most $D(x,y) + \delta$ from $x$ to $y$. 
We say that $D$ is a \emph{geodesic metric} if for each $x,y\in X$, there exists a curve of $D$-length exactly $D(x,y)$ from $x$ to $y$.
By~\cite[Proposition 2.5.19]{bbi-metric-geometry}, if $(X,D)$ is compact and $D$ is a length metric then $D$ is a geodesic metric.
\medskip

\subsubsection*{Subsets of the plane}

\noindent
We write $\BB S = [0,1]^2$ for the unit square.
\medskip

\noindent
For a rectangle $R\subset\BB C$ with sides parallel to the coordinate axes, we write $\bdy_{\op L} R$ and $\bdy_{\op R} R$ for its left and right sides.
\medskip

\noindent
For a set $A\subset\BB C$ and $r>0$, we write
\eqbn
B_r(A) := \{z\in \BB C : \text{Euclidean distance from $z$ to $A$}  < r\} 
\eqen
For $z\in\BB C$ we write $B_r(z) = B_r(\{z\})$ for the Euclidean ball of radius $r$ centered at $z$. 
\medskip

\noindent
An \emph{annular region} is a bounded open set $A\subset\BB C$ such that $A$ is homeomorphic to an open or closed Euclidean annulus. If $A$ is an annular region, then $\bdy A$ has two connected components, one of which disconnects the other from $\infty$. We call these components the outer and inner boundaries of $A$, respectively.

\begin{defn}[Distance across and around annuli] \label{def-around-across} 
Let $D$ be a length metric on $\BB C$. 
For an annular region $A \subset\BB C$, we define $D\left(\text{across $A$}\right)$ to be the $D $-distance between the inner and outer boundaries of $A$.
We define $D \left(\text{around $A$}\right)$ to be the infimum of the $D $-lengths of a path in $A$ which disconnect the inner and outer boundaries of $A$. 
\end{defn}

\noindent
Note that both $D(\text{across $A$})$ and $D(\text{around $A$})$ are determined by the internal metric of $D$ on $A$.

\subsection{White-noise decomposition}
\label{sec-wn-decomp}

Here we will introduce a variant of LFPP defined using the white noise decomposition of the GFF, which we will use in place of $D_h^\ep$ for all of Section~\ref{sec-efron-stein} and part of Section~\ref{sec-pt-tight}.

Let $W$ be a space-time white noise on $ \BB C \times [0,1]$, so that $W$ is a distribution (generalized function) and for any $L^2$ functions $f , g : \BB C\times [0,1] \rta \BB R$, the random variable $\int_0^1 \int_{\BB C} f(u,t) g(u,t) W(du , dt)$ (with the integral interpreted in the sense of distributional pairing) is centered Gaussian with variance $\int_0^1 \int_{\BB C} f(u,t) g(u,t) \, du\, dt$. 
Also let
\eqb \label{eqn-heat-kernel}
p_t(z) := \frac{1}{2\pi t} e^{-\frac{|z|^2}{2t} } 
\eqe
be the heat kernel on $\BB C$. 
For $m  ,n \in \BB N_0$ with $m \leq n$, we define
\eqb \label{eqn-phi-def}
\Phi_{m,n}(z) := \int_{2^{-2n}}^{2^{-2m}} \int_{\BB C}  p_{\frac{t}{2}}(z-u) W(du,dt) .
\eqe
Then $\Phi_{m,n}$ is a smooth centered Gaussian process with variance
\eqb
\op{Var} \Phi_{m,n}(z) = (n-m) \log 2.
\eqe

It is clear from the definition that for any $z\in\BB C$, $\Phi_{m,n}(\cdot-z) \eqD \Phi_{m,n}(\cdot)$. The processes $\Phi_{m,n}$ also have the following scale invariance property (see~\cite[Section 2.1]{dddf-lfpp} for an explanation): 
\eqb \label{eqn-phi-scale}
\Phi_{m,n}(2^k\cdot) \eqD \Phi_{m+k , n+k}(\cdot) .
\eqe
We will also need the following basic modulus of continuity estimate for $\Phi_{0,n}$, which is proven in~\cite[Proposition 3]{dddf-lfpp}.

\begin{lem}[\cite{dddf-lfpp}] \label{lem-phi-cont}
For each bounded open set $U\subset\BB C$, there are constants $c_0,c_1> 0$ depending on $U$ such that for each $n\in\BB N$, 
\eqb
\BB P\left[ 2^{-n} \sup_{z\in U} |\nabla \Phi_{0,n}| > x \right] \leq c_0 e^{-c_1 x^2} ,\quad \forall x > 0 .
\eqe
\end{lem}
\begin{proof}
The special case when $\ol U = \BB S$ is the unit square is~\cite[Proposition 3]{dddf-lfpp}. The case of a general choice of $U$ follows by covering $U$ by translates of $[0,1]^2$ and using the translation invariance of the law of $\Phi_{0,n}$. 
\end{proof}

For any $z,w\in\BB C$, the correlation of $\Phi_{m,n}(z)$ and $\Phi_{m,n}(w)$ is positive. In some of our arguments this property is undesirable since we want to use percolation techniques and/or the Efron-Stein inequality, for which exact long-range independence is useful. 
We therefore define a truncated version $\Psi_{m,n}$ of $\Phi_{m,n}$ which lacks the scale invariance property~\eqref{eqn-phi-scale} but has a finite range of dependence. 
In particular, let $\chi : [0,1] \rta \BB C$ be a smooth, radially symmetric bump function which is identically equal to 1 on $B_1(0)$ and identically equal to 0 on $\BB C\setminus B_2(0)$. Also fix a positive constant $\ep_0 \in (0,1/100)$ which will be chosen later (in a universal manner). Using the notation~\eqref{eqn-heat-kernel}, we define
\eqb \label{eqn-heat-kernel-trunc}
\wt p_t(z) := p_t(z) \chi(z/\sigma_t) \quad\text{where} \quad \sigma_t := \frac{1}{100} \sqrt t \left| \log_2 t \right|^{\ep_0}. 
\eqe
The reason for the choice of $\sigma_t$ is that a Brownian motion is unlikely to travel distance more than $\sigma_t$ in $t$ units of time, so the mass of $p_t(z)$ is concentrated in the $\sigma_t$-neighborhood of the origin. This makes it so that $\wt p_t$ is a good approximation for $p_t$. 

For $m,n\in\BB N$ with $m < n$, we define
\eqb \label{eqn-phi-def-trunc}
\Psi_{m,n}(z) := \int_{2^{-2n}}^{2^{-2m}} \int_{\BB C}  \wt p_{\frac{t}{2}}(z-u) W(du,dt) .
\eqe
Since $\wt p_t$ is smooth, $\Psi_{m,n}$ is a smooth centered Gaussian process. 
Furthermore, since $\wt p_t$ is supported on the Euclidean ball of radius $2 \sigma_t$ centered at zero, it follows that $\Psi_{m,n}(z)$ and $\Psi_{m,n}(w)$ are independent if $|z-w| \geq 2 \sigma_{2^{-2m}} $, which is the case provided $|z-w| \geq 2^{-m} m^{\ep_0}$. 

The following lemma, which is~\cite[Proposition 5]{dddf-lfpp}, allows us to transfer results between $\Phi_{m,n}$ and $\Psi_{m,n}$. 

\begin{lem}[\cite{dddf-lfpp}] \label{lem-field-compare}
Let $U\subset\BB C$ be a bounded open set. There are constants $c_0,c_1 > 0$ depending only on $U$ such that for $x > 0$, 
\eqb
\BB P\left[ \sup_{n \in\BB N} \sup_{z\in \ol U} |\Phi_{0,n}(z) - \Psi_{0,n}(z)| > x \right] \leq c_0 e^{-c_1 x^2} .
\eqe
\end{lem}
\begin{proof}
The special case when $\ol U = \BB S$ is the unit square is~\cite[Proposition 5]{dddf-lfpp}. The case of a general choice of $U$ follows by covering $U$ by translates of $[0,1]^2$ and using the translation invariance of the joint law of $\Phi_{0,n}$ and $\Psi_{0,n}$.
\end{proof}

\subsection{Basic definitions for white noise LFPP}
\label{sec-lfpp-def}

Fix $\xi > 0$.
For $m, n\in\BB N$ with $m\leq n$, we define the LFPP metric associated with the field $\Phi_{m,n}$ of~\eqref{eqn-phi-def} by
\eqb \label{eqn-lfpp-def}
D_{m,n}(z,w) := \inf_P \int_0^1 e^{\xi \Phi_{m,n}(P(t)) } |P'(t)| \,dt
\eqe
where the infimum is over all piecewise continuously differentiable paths $P : [0,1] \rta \BB C$ joining $z$ and $w$. 
We also define 
\eqb \label{eqn-bdy-dist-def}
L_{0,n} := D_{0,n}\left( \bdy_{\op L} \BB S , \bdy_{\op R} \BB S  ; \BB S\right) ,
\eqe
i.e., $L_{0,n}$ is the $D_{0,n}$-distance between the left and right boundaries of the unit square restrict to paths which are contained in the (closed) unit square.
The random variable $L_{0,n}$ will be the main observable considered in our proofs. 

For $p \in [0,1]$ and $n\in\BB N$, define the $p$th quantile
\eqb \label{eqn-quantile-def}
\ell_n(p) := \inf\left\{ x \in \BB R : \BB P\left[ L_{0,n} \leq x \right] \leq p \right\}. 
\eqe
It is easy to see from the fact that $\Phi_{0,n}$ is a Gaussian process with non-zero variances that in fact $\BB P[L_{0,n} \leq \ell_n(p)] = p$. 
We define the median 
\eqb \label{eqn-median-def}
\lambda_n := \ell_n(1/2) , 
\eqe
which will be the normalizing factor for distances in our scaling limit results (note that $\lambda_n$ is defined in the same way as $\frk a_\ep$ from~\eqref{eqn-gff-constant} but with $\Phi_{0,n}$ in place of $h_\ep^*$). 
We also define the maximal quantile ratio
\eqb \label{eqn-quantile-max}
\Lambda_n(p) := \max_{k \leq n} \frac{\ell_n(1-p)}{\ell_n(p)} ,\quad \forall p \in (0,1/2) .
\eqe

We define $\wt D_{m,n}$, $\wt L_{0,n}$, $\wt \ell_n(p)$, and $\wt\Lambda_n(p)$ in the same manner as about but with the field $\Psi_{m,n}$ of~\eqref{eqn-phi-def-trunc} in place of $\Phi_{m,n}$. 
The starting point of our proofs is some a priori estimates for $D_{0,n}$ and $\wt D_{0,n}$. These estimates were established (for general values of $\xi$) in~\cite[Section 4]{dddf-lfpp} using comparison to percolation. 

\begin{prop} \label{prop-perc-estimate}
For each $\xi > 0$, there is a constant $\frk p \in (0,1/2)$ such that the following is true. 
Let $U\subset \BB C$ be an open set and let $K_1,K_2\subset U$ be disjoint compact connected sets which are not singletons. 
Recall the definition of the internal metric $D_{0,n}(\cdot,\cdot ; U)$ on $U$ from Section~\ref{sec-basic}. 
There are constants $c_0,c_1 >0$ depending only on $U,K_1,K_2,\xi$ such that for $n\in\BB N$ and $T > 3$, 
\eqb \label{eqn-perc-lower}
\BB P\left[ D_{0,n}(K_1,K_2 ; U)  < T^{-1} \ell_n(\frk p) \right] \leq c_0 e^{-c_1 (\log T)^2} 
\eqe
and
\eqb \label{eqn-perc-upper}
\BB P\left[ D_{0,n}(K_1,K_2 ; U)  > T \Lambda_n(\frk p) \ell_n(\frk p) \right] \leq c_0 e^{-c_1 (\log T)^2 / \log\log T} .
\eqe
The same is true with $\wt D_{0,n}$ in place of $D_{0,n}$. 
\end{prop}

It is easy to see from Proposition~\ref{prop-perc-estimate} that both estimates from the proposition also hold with $D_{0,n}(K_1,K_2;U)$ replaced by either $D_{0,n}(\text{around $A$})$ or $D_{0,n}(\text{across $A$})$ for an annular region $A$ (Definition~\ref{def-around-across}), with the constants $c_0,c_1$ depending only on $A$ and $\xi$. 

\begin{proof}[Proof of Proposition~\ref{prop-perc-estimate}]
It is shown in~\cite[Corollary 17 and Proposition 18]{dddf-lfpp} that there is a constant $\frk p \in (0,1/2)$ as in the proposition statement and constants $a_0 ,a_1 > 0$ depending only on $\xi$ such that if $R := [0,1] \times [0,3]$, then for $n\in\BB N$ and $T>3$, 
\eqb \label{eqn-perc-lower0}
\BB P\left[ D_{0,n}(\bdy_{\op L} R , \bdy_{\op R} R ; R )  < T^{-1} \ell_n(\frk p) \right] \leq c_0 e^{-c_1 (\log T)^2} 
\eqe
and
\eqb \label{eqn-perc-upper0}
\BB P\left[ D_{0,n}(\bdy_{\op U} R , \bdy_{\op T} R ; R )  > T \Lambda_n(\frk p) \ell_n(\frk p) \right] \leq c_0 e^{-c_1 (\log T)^2 / \log\log T} .
\eqe

To deduce~\eqref{eqn-perc-lower} from~\eqref{eqn-perc-lower0}, we choose a finite collection of rectangles $R_1,\dots,R_K \subset U$ (in a manner depending only on $U,K_1,K_2$) which each have aspect ratio 3 with the following property: any path from $K_1$ to $K_2$ in $U$ must cross one of the $R_k$'s in the ``easy" direction (i.e., it must cross between the two longer sides of $R_k$). We then apply~\eqref{eqn-perc-lower0} together with the scale, translation, and rotation invariance properties of $\Phi_{0,n}$ to simultaneously lower-bound the distance between the two longer sides of $R_k$ for each $k\in [1,K]_{\BB Z}$. This gives~\eqref{eqn-perc-lower}.  

To deduce~\eqref{eqn-perc-upper} from~\eqref{eqn-perc-upper0}, we apply a similar argument. We look at a collection of rectangles $R_1,\dots,R_K \subset U$ with aspect ratio 3 with the following property: if $\pi_k$ for $k\in [1,K]_{\BB Z}$ is a path in $R_k$ between the two shorter sides of $R_k$, then the union of the $\pi_k$'s contains a path from $K_1$ to $K_2$. We then apply~\eqref{eqn-perc-upper0} to upper bound the $D_{0,n}$-distance between the two shorter sides of each $R_k$ and thereby deduce~\eqref{eqn-perc-upper}. 

The bounds with $\wt D_{0,n}$ in place of $D_{0,n}$ follow from the bounds for $D_{0,n}$ combined with Lemma~\ref{lem-field-compare}. 
\end{proof}

\subsection{Existence of an exponent for white noise LFPP}
\label{sec-Q-wn}

We will need a variant of Proposition~\ref{prop-Q} for the white noise LFPP metrics $D_{0,n}$.  

\begin{prop} \label{prop-Q-wn}
For each $\xi  >0$, there exists $Q = Q(\xi) > 0$ such that
\eqb \label{eqn-Q-wn}
\lambda_n = 2^{-n(1-\xi Q ) + o_n(n) } ,\quad \text{as $n\rta\infty$}. 
\eqe 
Furthermore, $\xi \mapsto Q(\xi)$ is continuous, strictly decreasing on $(0,0.7)$, non-increasing on $(0,\infty)$, and satisfies $\lim_{\xi \rta \infty} Q(\xi) =0$. 
\end{prop}

From now on we define $Q(\xi)$ as in Proposition~\ref{prop-Q-wn}. 
We will show that Proposition~\ref{prop-Q} holds (with the same value of $Q$) in Section~\ref{sec-gff-compare}. 

In this subsection, we will establish all of Proposition~\ref{prop-Q-wn} except for the statement that $Q > 0$, which will be proven in Section~\ref{sec-Q-pos} using the results of~\cite{lfpp-pos}. 
To prove the existence of $Q$, we will use a subadditivity argument. 
We first need an a priori estimate for the maximal quantile ratio from~\eqref{eqn-quantile-max}. 

\begin{lem} \label{lem-a-priori-conc}
Let $\frk p$ be as in Proposition~\ref{prop-perc-estimate}. 
There is a constant $c > 0$ depending only on $\xi$ such that for each $n\in\BB N$,
\eqb \label{eqn-a-priori-conc}
\Lambda_n(\frk p) \leq e^{c\sqrt n}  .
\eqe
\end{lem}
\begin{proof}
The random variable $\log L_{0,n}$ is a $\xi$-Lipschitz function of the continuous centered Gaussian process $\Phi_{0,n}$. 
Since $\op{Var} \Phi_{0,n}(z) = n \log 2$ for each $z\in\BB C$, we can apply Lemma~\ref{lem-gaussian-var} to get
\eqb
\op{Var} \log L_{0,n} \preceq n 
\eqe
with the implicit constant depending only on $\xi$. 
We now obtain~\eqref{eqn-a-priori-conc} by a trivial bound for quantile rations in terms of variance (see, e.g.,~\cite[Lemma 3.2]{dddf-lfpp} applied with $Z = \log L_{0,n}$). 
\end{proof}

Instead of proving the subadditivity of $\log\lambda_n$ directly, we will instead prove subadditivity for a slightly different quantity which is easier to work with. 
For a square $S$, let $\mcl A_S$ be the closed square annulus between $\bdy S$ and the boundary of the square with the same center and three times the radius of $S$. 
For $n\in\BB N$, let
\eqb \label{eqn-mu-def}
\mu_n := \BB E\left[ D_{0,n}\left(\text{around $\mcl A_{\BB S} $}\right) \right] , 
\eqe
where here we recall that $\BB S$ is the closed unit square. The following lemma allows us to compare $\lambda_n$ and $\mu_n$.

\begin{lem} \label{lem-Q-compare}
For each $p > 0$, there are constants $c_0 , c_1 > 0$ depending only on $p$ and $\xi$ such that
\eqb \label{eqn-Q-compare}
c_0 e^{-c_1\sqrt n} \lambda_n^p
\leq \BB E\left[ D_{0,n}\left(\text{around $\mcl A_{\BB S}$}\right)^p \right] 
\leq c_0 e^{c_1 \sqrt n} \lambda_n^p .
\eqe
\end{lem}
\begin{proof}
By integrating the estimates of Proposition~\ref{prop-perc-estimate}, we obtain
\alb
\ell_n(\frk p)^p  
\preceq \BB E\left[ D_{0,n}\left(\text{around $\mcl A_{\BB S}$}\right)^p \right] 
\preceq \left(\Lambda_n(\frk p) \ell_n(\frk p)\right)^p .
\ale
Combining this with Lemma~\ref{lem-a-priori-conc} yields~\eqref{eqn-Q-compare}.
\end{proof}

For $m\in\BB N$, define the collection of squares
\eqb
\wh{\mcl S}_m := \left\{\text{$2^{-m} \times 2^{-m}$ dyadic squares contained in $\mcl A_{\BB S}$ }\right\} .
\eqe
The reason for the hat in the notation is to avoid confusion with the collection of dyadic squares $\mcl S_K$ used in Section~\ref{sec-efron-stein}. 
Due to the scaling property of LFPP, we have the following formula, which will be a key input in the proof of the sub-multiplicativity of $\mu_n$. 

\begin{lem} \label{lem-Q-scale}
Let $m,n\in \BB N$ with $m < n$. For each $S\in\wh{\mcl S}_m$, 
\eqb
\BB E\left[ D_{m,n}\left(\text{around $\mcl A_S$}\right) \right] = 2^{-m} \mu_{ n-m} .
\eqe
\end{lem}
\begin{proof}
By the scaling property~\eqref{eqn-phi-scale} of $\Phi_{m,n}$ and the translation invariance of the law of $\Phi_{m,n}$, 
\eqbn
D_{m,n}\left(\text{around $\mcl A_S$} \right) \eqD 2^{-m} D_{0,n-m}\left(\text{around $\mcl A_{\BB S}$}\right) .
\eqen
Taking expectations now gives the lemma statement. 
\end{proof}

We can now prove the existence of an exponent for $\mu_n$. 

\begin{lem} \label{lem-mu-exponent}
Let $\mu_n$ for $n\in\BB N$ be as in~\eqref{eqn-mu-def}. 
For each $\xi > 0$, there exists $\alpha = \alpha(\xi) \in \BB R$ such that 
\eqb \label{eqn-mu-exponent}
\lim_{n\rta\infty} \frac{\log\mu_n}{\log n} = - \alpha .
\eqe
\end{lem}
\begin{proof}
We will show that for $n,m\in\BB N$ with $n > m$, 
\eqb \label{eqn-mu-subadd} 
\mu_n \preceq 2^{ 2n^{2/3}} \mu_m \mu_{n-m}  
\eqe
with the implicit constant depending only on $\xi$. By a version of Fekete's subadditivity lemma with an error term (see, e.g.,~\cite[Lemma 6.4.10]{dembo-ld}) applied to $\log \mu_n$, this implies~\eqref{eqn-mu-exponent}. Throughout the proof, $c_0$ and $c_1$ denote positive constants depending only on $\xi$ which may change from line to line. 
\medskip

\noindent\textit{Step 1: regularity event for $\Phi_{0,m}$.}
Let $U$ be a bounded open subset of $\BB C$ which contains $\mcl A_S$. For $T > 0$, let 
\eqb
E_{m,n}  := \left\{ 2^{-m} \sup_{z \in U} |\nabla \Phi_{0,n}(z)| \leq \frac{\log 2}{\xi} n^{2/3}  \right\}  
\eqe
so that by Lemma~\ref{lem-phi-cont},
\eqb \label{eqn-Q-reg-prob}
\BB P\left[ E_{m,n} \right] \geq 1 - c_0 \exp\left( - c_1  n^{4/3} \right).
\eqe 
\medskip

\noindent\textit{Step 2: lower bound for $D_{0,m}$ in terms of a sum over $2^{-m}\times 2^{-m}$ squares. }
Let $P_{0,m}$ be a path around $\mcl A_{\BB S}$ with $D_{0,m}$-length $D_{0,m}\left(\text{around $\mcl A_{\BB S}$}\right)$, parametrized by its $D_{0,m}$-length. 
Let $t_0 = 0$ and let $S_0 \in \wh{\mcl S}_m$ be chosen so that $P_{0,m}(0)  \in S_0$.
Inductively, suppose that $j\in\BB N$ and $t_{j-1}$ and $S_{j-1}$ have been defined. Let $t_j$ be the first time $t \geq t_{j-1}$ for which $P_{0,m}(t) \notin B_{2^{-m-1}}(S_{j-1})$, or $t_j = D_{0,m}\left(\text{around $\mcl A_{\BB S}$}\right)$ if no such time exists.
Also let $S_j\in\wh{\mcl S}_m$ be chosen so that $P_{0,m}(t_j) \in S_j$. Note that $S_j$ necessarily shares a corner or a side with $S_{j-1}$. 
Let
\eqb
J := \min\left\{j\in \BB N : t_j = D_{0,m}\left(\text{around $\mcl A_{\BB S}$}\right) \right\} .
\eqe

By the definition of $D_{0,m}$, on $E_{m,n}$ it holds for $j\in [0,J-1]_{\BB Z}$ that
\allb 
t_{j+1} - t_j 
&\geq D_{0,m}\left( \text{across $B_{2^{-m-1}}(S_j) \setminus S_j$}\right) \notag\\
&\geq 2^{-m} \exp\left(\xi \min_{z \in B_{2^{-m-1}}(S_j)} \Phi_{0,m}(z) \right) \notag\\
&\succeq  2^{-m - n^{2/3}} e^{\xi \Phi_{0,m}(v_{S_j}) }  , 
\alle
where $v_{S_j}$ is the center of $S_j$ and the implicit constant in $\succeq$ is universal. 
Therefore, on $E_{m,n}$,
\eqb \label{eqn-m-lower}
D_{0,m}\left(\text{around $\mcl A_{\BB S}$}\right) \succeq T^{-1} 2^{-m - n^{2/3} } \sum_{j=0}^{J-1}   e^{\xi \Phi_{0,m}(v_{S_{j}}) } .
\eqe
\medskip

\noindent\textit{Step 3: upper bound for $D_{0,n}$ in terms of a sum over $2^{-m}\times 2^{-m}$ squares. }
Since the squares $S_j$ and $S_{j+1}$ share a corner or a side for each $j \in [0,J-1]_{\BB Z}$, if $\pi_j$ is a path around $\mcl A_{S_{j}}$ for $j \in [0,J-1]_{\BB Z}$, then the union of the paths $\pi_j$ contains a path around $\mcl A_{\BB S}$. Consequently,
\eqb
D_{0,n}\left(\text{around $\mcl A_{\BB S}$} \right) 
\leq \sum_{j=0}^{J-1} D_{0,n}\left(\text{around $\mcl A_{S_{j}}$}\right) .
\eqe
Since $\Phi_{0,n} = \Phi_{0,m} + \Phi_{m,n}$, it follows that on $E_{m,n}$,
\allb
D_{0,n}\left(\text{around $\mcl A_{S_{j}}$}\right)
\leq \exp\left(\xi \max_{z\in\mcl A_{S_{j}}} \Phi_{0,m}(z) \right) D_{m,n}\left(\text{around $\mcl A_{S_{j}}$}\right)
\preceq 2^{n^{2/3}} e^{\xi \Phi_{0,m}(v_{S_{j}})} D_{m,n}\left(\text{around $\mcl A_{S_{j}}$}\right) .
\alle
Hence, on $E_{m,n}$,
\eqb \label{eqn-n-upper} 
D_{0,n}\left(\text{around $\mcl A_{\BB S}$} \right)  \preceq  T \sum_{j=0}^{J-1} e^{\xi \Phi_{0,m}(v_{S_{j}})} D_{m,n}\left(\text{around $\mcl A_{S_{j}}$}\right)  .
\eqe
\medskip

\noindent\textit{Step 4: comparison of $\mu_n$ and $\mu_m\mu_{n-m}$.}
The event $E_{m,n}$ and the squares $S_j$ for $j\in[0,J]_{\BB Z}$ are a.s.\ determined by $\Phi_{0,m}$, which is independent from $D_{m,n}$. Therefore, 
\allb \label{eqn-mu-upper}
\BB E\left[ D_{0,n}\left(\text{around $\mcl A_{\BB S}$}\right) \BB 1_{E_{m,n}} \right]
&\preceq 2^{n^{2/3}} \BB E\left[  \sum_{j=0}^{J-1} e^{\xi \Phi_{0,m}(v_{S_{j}})} \BB 1_{E_{m,n}} \right] \BB E\left[  D_{m,n}\left(\text{around $\mcl A_{S_{j}}$}\right) \right] \quad \text{(by~\eqref{eqn-n-upper})} \notag\\
&\preceq   2^{m + 2n^{2/3}} \BB E\left[ D_{0,m}\left(\text{around $\mcl A_{\BB S}$}\right)  \right]  \times 2^{-m} \mu_{n-m} \quad \text{(by~\eqref{eqn-m-lower} and Lemma~\ref{lem-Q-scale})} \notag\\
&= 2^{ 2n^{2/3}} \mu_m \mu_{n-m} .
\alle

We need to show that the left side of~\eqref{eqn-mu-upper} is not too much smaller than $\mu_n$. 
We have 
\eqb \label{eqn-mu-split}
\mu_n 
 \leq \BB E\left[ D_{0,n}\left(\text{around $\mcl A_{\BB S}$}\right) \BB 1_{E_{m,n}} \right] + \BB E\left[ D_{0,n}\left(\text{around $\mcl A_{\BB S}$}\right) \BB 1_{E_{m,n}^c} \right]  .
\eqe
By the Cauchy-Schwarz inequality followed by Lemma~\ref{lem-Q-compare} and~\eqref{eqn-Q-reg-prob},
\allb \label{eqn-mu-error}
&\BB E\left[ D_{0,n}\left(\text{around $\mcl A_{\BB S}$}\right) \BB 1_{E_{m,n}^c} \right] \notag\\
&\qquad \leq \BB E\left[ D_{0,n}\left(\text{around $\mcl A_{\BB S}$}\right)^2 \right]^{1/2} \BB P\left[ E_{m,n}^c\right]^{1/2} \quad \text{(by Cauchy-Schwarz)} \notag\\
&\qquad \leq c_0 e^{c_1\sqrt n} \lambda_n \times c_0 e^{-c_1 n^{4/3} } \quad \text{(by Lemma~\ref{lem-Q-compare} and~\eqref{eqn-Q-reg-prob})} \notag\\
&\qquad \leq c_0 e^{-c_1 n^{4/3}} \mu_n \quad \text{(by Lemma~\ref{lem-Q-compare})} .
\alle
By plugging~\eqref{eqn-mu-error} into~\eqref{eqn-mu-split}, we obtain
\eqbn
\mu_n \leq    \BB E\left[ D_{0,n}\left(\text{around $\mcl A_{\BB S}$}\right) \BB 1_{E_{m,n}} \right] +    c_0 e^{-c_1 n^{4/3}} \mu_n    .
\eqen
Since $c_0 e^{-c_1 n^{4/3}} \leq 1$ for large enough $n$, we can re-arrange this last inequality to get that for large enough $n$, 
\eqb \label{eqn-mu-compare}
\mu_n \leq 2 \BB E\left[ D_{0,n}\left(\text{around $\mcl A_{\BB S}$}\right) \BB 1_{E_{m,n}} \right] .
\eqe

Plugging~\eqref{eqn-mu-compare} into~\eqref{eqn-mu-upper} gives~\eqref{eqn-mu-subadd} when $n$ is sufficiently large. By possibly increasing the implicit constant to deal with finitely many small values of $n$, we obtain~\eqref{eqn-mu-subadd} in general.
\end{proof}

\begin{proof}[Proof of Proposition~\ref{prop-Q-wn}]
Let $\alpha=\alpha(\xi)$ be as in Lemma~\ref{lem-mu-exponent}. By Lemmas~\ref{lem-mu-exponent} and~\ref{lem-Q-compare}, we have $\lambda_n = 2^{o_n(n)} \mu_n = 2^{-\alpha n + o_n(n)}$. Therefore,~\eqref{eqn-Q-wn} holds with $Q := (1-\alpha)/\xi$.  

We prove that $Q > 0$ for all $\xi > 0$ in Lemma~\ref{lem-Q-pos} below. 
The remaining properties of $Q$ asserted in the proposition statement are essentially proven in~\cite{gp-lfpp-bounds}, but the results there are for LFPP defined using the circle average process of a GFF instead of the white noise approximation. The median left-right crossing distance of $\BB S$ for the two variants of LFPP can be compared up to multiplicative errors of the form $2^{o_n(n)}$ due to~\cite[Proposition 3.3]{ding-goswami-watabiki}. 
Hence, we can apply~\cite[Lemma 1.1]{gp-lfpp-bounds} to get that $\xi\mapsto Q(\xi)$ is and~\cite[Lemma 4.1]{gp-lfpp-bounds} to get that $\xi \mapsto Q(\xi)$ is strictly decreasing on $(0,0.7)$, non-increasing on $(0,\infty)$, and satisfies $\lim_{\xi \rta \infty} Q(\xi) =0$. 
\end{proof}

\subsection{Positivity of $Q$}
\label{sec-Q-pos}

To conclude the proof of Proposition~\ref{prop-Q-wn} it remains to check the following. 

\begin{lem} \label{lem-Q-pos}
There are universal constants $c_0,c_1 > 0$ such that with $Q = Q(\xi)$ as in Proposition~\ref{prop-Q-wn},  
\eqb
Q(\xi) \geq c_0 \xi^{-1} e^{-c_1 \xi} ,\quad\forall \xi > 0 .
\eqe
In particular, $Q  > 0$ for all $\xi > 0$.
\end{lem}

Lemma~\ref{lem-Q-pos} will be extracted from~\cite[Theorem 1.1]{lfpp-pos}, which gives similar result for LFPP defined using the discrete GFF.
To explain this result, for $n\in\BB N$, let $\mcl B_n := [-2^n , 2^n]_{\BB Z}^2$ and let $h_n^{\BB Z}$ be the discrete Gaussian free field on $\mcl B_n$, with zero boundary conditions, normalized so that $\BB E[h_n^{\BB Z}(z) h_n^{\BB Z}(w)] = \frac{\pi}{2} \op{Gr}_{\mcl B_n}(z,w)$ for each $z,w\in \mcl B_n$, where $\op{Gr}_{\mcl B_n}$ is the discrete Green's function.  
For $\xi > 0$, we define the \emph{discrete LFPP} metric with parameter $\xi$ associated with $h_n^{\BB Z}$ by
\eqbn
D_n^{\BB Z}(z,w) = \inf_{P : z\rta w} \sum_{j=1}^{|P|} e^{\xi h_n^{\BB Z}(P(j))}  ,\quad\forall z,w\in \mcl B_n
\eqen
where the infimum is over all nearest-neighbor paths $P: [0,|P|]_{\BB Z} \rta \mcl B_n$ with $P(0) = z$ and $P(|P|) = w$.  
We define the discrete square annulus
\eqbn
\mcl A_n  :=   [2^{n-1/2} , 2^{n-1/2}]_{\BB Z}^2  \setminus  [2^{n-1} , 2^{n-1}]_{\BB Z}^2    \subset \mcl B_n
\eqen
and we define $D_n^{\BB Z}\left(\text{across $\mcl A_n $} \right)$ to be the $D_n^{\BB Z}$-distance between the inner and outer boundaries of $\mcl A_n $. 

It is shown in~\cite[Theorem 1.1]{lfpp-pos} that there are universal constants $c_0,c_1 >0$ such that
\eqb \label{eqn-lfpp-pos}
\lim_{n\rta\infty} \BB P\left[ D_n\left(\text{across $\mcl A_n $} \right) \geq  \exp\left( c_0 e^{-c_1 \xi} n \right) \right] \geq \frac12 .
\eqe
To deduce Lemma~\ref{lem-Q-pos} from~\eqref{eqn-lfpp-pos} we need to establish that distances in discrete LFPP can be described in terms of $Q$.

\begin{lem} \label{lem-discrete-lfpp}
With $Q$ as in Proposition~\ref{prop-Q-wn}, it holds for each $\delta > 0$ that
\eqb \label{eqn-discrete-lfpp}
\lim_{n\rta\infty} \BB P\left[ 2^{n(\xi Q - \delta)} \leq  D_n^{\BB Z}\left(\text{across $\mcl A_n $} \right) \leq 2^{n(\xi Q + \delta)} \right] = 1 .
\eqe
\end{lem}
\begin{proof}
Let $\mcl A := [2^{-1/2} , 2^{-1/2} ] \setminus [2^{-1} , 2^{-1}]$ be the continuum analog of $\mcl A_n $. 
By Proposition~\ref{prop-Q-wn}, we have $\lambda_n = 2^{-n(1-\xi Q)  + o_n(1)}$. By combining this with Proposition~\ref{prop-perc-estimate} and Lemma~\ref{lem-a-priori-conc} (to say that $\ell_n(\frk p) \asymp \lambda_n$ and $\Lambda_n(\frk p) \asymp 1$), it follows that for each $\delta >0$, 
\eqb \label{eqn-wn-annulus}
\lim_{n\rta\infty} \BB P\left[ 2^{-n(1-\xi Q + \delta)} \leq  D_{0,n}\left(\text{across $\mcl A$} \right) \leq 2^{-n(1-\xi Q - \delta)} \right] = 1 .
\eqe

We now want to apply the main result of~\cite{ang-discrete-lfpp} to transfer from~\eqref{eqn-wn-annulus} to~\eqref{eqn-discrete-lfpp}. However,~\cite{ang-discrete-lfpp} considers LFPP defined using the circle average process of the GFF rather than the white noise approximation, so we need an intermediate step. 
Let $\rng h$ be a zero-boundary GFF on $[-1,1]^2$ and define $ D_n^{\rng h}$ in the same manner as $D_{0,n}$ but with the radius $2^{-n}$ circle average process $\rng h_{2^{-n}}$ for $\rng h$ in place of $\Phi_{0,n}$. By applying~\cite[Proposition 3.3]{ding-goswami-watabiki}, to compare $\rng h_{2^{-n}}$ and $\Phi_{0,n}$, we see that~\eqref{eqn-wn-annulus} remains true with $D_n^{\rng h}$ in place of $D_{0,n}$. 

We now apply~\cite[Theorem 1.4]{ang-discrete-lfpp} to deduce~\eqref{eqn-discrete-lfpp} from the version of~\eqref{eqn-wn-annulus} with $D_n^{\rng h}$ in place of $D_{0,n}$.
Note that in our setting space is re-scaled by a factor of $2^{-n}$ as compared to the setting of~\cite[Theorem 1.4]{ang-discrete-lfpp} (which considers a GFF on $[0,n]^2$ and averages over circles of radius 1). This is the reason why the estimates in~\eqref{eqn-discrete-lfpp} and~\eqref{eqn-wn-annulus} differ by a factor of $2^{-n}$. We also note that the factor of $\sqrt{\pi/2}$ in~\cite{ang-discrete-lfpp} does not appear in our setting due to our normalization of the discrete GFF.
\end{proof}

\begin{proof}[Proof of Lemma~\ref{lem-Q-pos}]
This is immediate from Lemma~\ref{lem-discrete-lfpp} and~\eqref{eqn-lfpp-pos}. 
\end{proof}

\section{Concentration of left/right crossing distance}
\label{sec-efron-stein}
\subsection{Statement and setup}
\label{sec-ef-setup}

The goal of this section is to prove the following proposition. 

\begin{prop} \label{prop-quantile-ratio}
Let $\xi > 0$, let $\frk p$ be the constant from Proposition~\ref{prop-perc-estimate}, and define the maximal quantile ratio $\Lambda_n(\frk p)$ as in~\eqref{eqn-quantile-max}. 
We have
\eqb \label{eqn-quantile-ratio}
\sup_{n\in\BB N} \Lambda_n(\frk p) < \infty .
\eqe
\end{prop}

As an immediate corollary of Propositions~\ref{prop-perc-estimate} and~\ref{prop-quantile-ratio}, we get the following improvement on Proposition~\ref{prop-perc-estimate}.

\begin{prop}  \label{prop-perc-estimate'}
Let $\xi  > 0$ and for $n\in\BB N$ let $\lambda_n$ be the median $D_{0,n}$-distance across $\BB S$, as in~\eqref{eqn-median-def}. 
Let $U\subset \BB C$ be an open set and let $K_1,K_2\subset U$ be disjoint compact connected sets which are not singletons. 
There are constants $c_0,c_1 >0$ depending only on $U,K_1,K_2,\xi$ such that for $n\in\BB N$ and $T > 3$, 
\eqb \label{eqn-perc-lower'}
\BB P\left[ D_{0,n}(K_1,K_2 ; U)  < T^{-1} \lambda_n \right] \leq c_0 e^{-c_1 (\log T)^2} 
\eqe
and
\eqb \label{eqn-perc-upper'}
\BB P\left[ D_{0,n}(K_1,K_2 ; U)  > T \lambda_n \right] \leq c_0 e^{-c_1 (\log T)^2 / \log\log T} .
\eqe
The same is true with $\wt D_{0,n}$ in place of $D_{0,n}$. 
In particular, the random variables $\lambda_n^{-1} D_{0,n}(K_1,K_2;U)$ and their reciporicals as $n$ varies are tight. 
\end{prop}
\begin{proof}
This follows since Proposition~\ref{prop-quantile-ratio} implies that the quantities $  \ell_n(\frk p)$ and $  \Lambda_n(\frk p) \ell_n(\frk p)$ appearing in Proposition~\ref{prop-perc-estimate} are bounded above and below by $ \xi$-dependent constants times $\lambda_n$. 
\end{proof}

Due to Lemma~\ref{lem-field-compare}, Proposition~\ref{prop-quantile-ratio} is equivalent to the analogous statement with distances defined using the truncated field $\Psi_{0,n}$ of~\eqref{eqn-phi-def-trunc} instead of the original field $\Phi_{0,n}$. Recall that objects defined with $\Psi_{0,n}$ instead of $\Phi_{0,n}$ are denoted by a tilde.
For the proof of Proposition~\ref{prop-quantile-ratio}, we will mostly work with $\Psi_{0,n}$ since the long-range independence of $\Psi_{0,n}$ is useful for our application of the Efron-Stein inequality. 
We will bound the quantile ratio $ \Lambda_n(\frk p)$ in terms of the variance of the log of the left-right crossing distance $\wt L_{0,n}$. This will be accomplished using the following elementary lemma.

\begin{lem} \label{lem-induct-quantile}
There is a constant $c > 0$ depending only on $\xi$ such that for each $n\in\BB N$, 
\eqb \label{eqn-induct-quantile}
\frac{ \ell_n(1-\frk p)}{  \ell_n(\frk p)} \leq \exp\left( c \sqrt{\op{Var} \left[ \log \wt L_{0,n}  \right]} \right) .
\eqe
\end{lem}
\begin{proof}
By exponentiating both sides of the estimate from~\cite[Lemma 22]{dddf-lfpp}, applied with $Z = \log \wt L_{0,n}$, we get
\eqb \label{eqn-induct-quantile-tilde}
\frac{ \wt\ell_n(1 - p)}{  \wt\ell_n(p)} \leq  \exp\left( \frac{\sqrt 2}{p} \sqrt{ \op{Var} \left[ \log \wt L_{0,n}  \right] } \right) ,\quad\forall p \in (0,1/2)
\eqe 
where $\wt\ell_n(\cdot)$ is the quantile function for $\wt L_{0,n}$. By Lemma~\ref{lem-field-compare}, $\ell_n(\frk p) \geq \wt\ell_n(\wt{\frk p})$ and $\ell_n(1-\frk p) \leq \wt\ell_n(1-\wt{\frk p})$ for $\wt{\frk p} \in (0,1/2)$ depending only on $\frk p , \xi$ (hence only on $\xi$). Therefore~\eqref{eqn-induct-quantile} follows from~\eqref{eqn-induct-quantile-tilde} applied with $p = \wt{\frk p}$. 
\end{proof}

In light of Lemma~\ref{lem-induct-quantile}, to prove Proposition~\ref{prop-quantile-ratio} we need a uniform upper bound for $\op{Var}\left[ \log \wt L_{0,n} \right]$. To prove such a bound we fix a large constant $K$ (to be chosen later, independently from $n$) and we use the decomposition
\eqb \label{eqn-var-decomp}
\op{Var}\left[ \log \wt L_{0,n} \right]
= \BB E\left[ \op{Var}\left[ \log \wt L_{0,n} \,|\, \Psi_{K,n} \right] \right] + 
\op{Var}\left[ \BB E\left[ \log \wt L_{0,n} \,|\, \Psi_{K,n} \right] \right] .
\eqe
The expectation of the conditional variance is easy to control using a Gaussian concentration bound, as explained in the following lemma.

\begin{lem} \label{lem-cond-var}
Almost surely, 
\eqb \label{eqn-cond-var}
\op{Var}\left[ \log \wt L_{0,n} \,|\, \Psi_{K,n} \right] \preceq K 
\eqe
with a deterministic implicit constant depending only on $\xi$. 
\end{lem}
\begin{proof}
If we condition on $\Psi_{K,n}$, then under the conditional law the random variable $\log \wt L_{0,n}$ is a measurable functional of the continuous centered Gaussian process $\Psi_{0,K}|_{\BB S}$, which satisfies $\op{Var} \Psi_{0,K}(z) \leq K \log 2$ for each $z\in \BB S$. 
Furthermore, if $f$ is a continuous function on $\BB S$ then adding $f$ to $\Psi_{0,K}$ can increase or decrease the value of $\wt L_{0,n}$ by a factor of at most $e^{\xi \|f\|_\infty}$, where $\|f\|_\infty$ is the $L^\infty$ norm. From this, we infer that $\log \wt L_{0,n}$ is a $\xi$-Lipschitz continuous function of $\Psi_{0,k}|_{\BB S}$ w.r.t.\ the $L^\infty$ norm under the conditional law given $\Psi_{K,n}$. 
By Lemma~\ref{lem-gaussian-var} (and a simple approximation argument to reduce to the case of a finite-dimensional Gaussian vector), we now obtain~\eqref{eqn-cond-var}. 
\end{proof}

The main difficulty in the proof is bounding the variance of the conditional expectation in~\eqref{eqn-var-decomp}. 
This will be done using the Efron-Stein inequality and induction. 
Let us first describe the setup for the Efron-Stein inequality.

\begin{defn} \label{def-dyadic-square}
Let $\ep_0$ be the positive constant from~\eqref{eqn-heat-kernel-trunc}. 
We define $\mcl S_K$ to be the set of $2^{-K} \times 2^{-K}$ dyadic squares $S$ which are contained in the $2^{-K+1} K^{ \ep_0}$-neighborhood of the Euclidean unit square $\BB S$. 
\end{defn}

As in~\cite[Equation (2.18)]{dddf-lfpp}, for a $2^{-K} \times 2^{-K}$ dyadic square $S \in \mcl S_K$ and $z\in\BB C$, let
\eqb \label{eqn-block-fn-def}
\psi_{K,n}^S(z) := \int_{2^{-2n}}^{2^{-2K}} \int_S \wt p_{t/2}(z-u) W(du,dt) ,
\eqe
where $\wt p_t$ is the truncated heat kernel as in~\eqref{eqn-heat-kernel-trunc}. 
By the spatial independence property of the white noise $W$, the random functions $\psi_{K,n}^S$ for different choices of $S$ are independent.
Furthermore, since $\wt p_{t/2}$ is supported on the $2^{-K} K^{\ep_0}$-neighborhood of 0 for each $t\leq 2^{-2K}$, it follows that $\psi_{K,n}^S$ is supported on the $2^{-K} K^{\ep_0}$-neighborhood of $S$ and that
\eqbn
\Psi_{K,n}  = \sum_{S\in\mcl S_K} \psi_{K,n}^S  \quad\text{on $\BB S$}.
\eqen

Let $\Psi_{K,n}^S := \Psi_{K,n} -  \psi_{K,n}^S + \wh\psi_{K,n}^S$, where $\wh\psi_{K,n}^S$ is a copy of $\psi_{K,n}^S$ sampled independently from everything else. 
Also let
\eqb \label{eqn-square-field-def}
\Psi_{0,n}^S := \Psi_{K,n}^S + \Psi_{0,K}. 
\eqe
Then $\Psi_{0,n}^S \eqD \Psi_{0,n}$. 

Write $\wt D_{0,n}^S$ for the LFPP distance defined with $\Psi_{0,n}^S$ in place of $\Psi_{0,n}$. Define $\wt L_{0,n}^S$ in the same manner as $\wt L_{0,n}$ but with $\Psi_{0,n}^S$ in place of $\Psi_{0,n}$, i.e., 
\eqb \label{eqn-square-field-crossing-def}
\wt L_{0,n}^S := \wt D_{0,n}^S\left( \bdy_{\op L} \BB S , \bdy_{\op R} \BB S ; \BB S \right). 
\eqe
Let us now record what we get from the Efron-Stein inequality. 
 
\begin{lem} \label{lem-efron-stein}
For each $n\in\BB N$,  
\allb \label{eqn-efron-stein}
\op{Var}\left[    \BB E\left[ \log \wt L_{0,n} \,|\, \Psi_{K,n}  \right]     \right] 
\leq   \sum_{S\in\mcl S_K} \BB E\left[ \sum_{S\in\mcl S_K} \left( \BB E\left[ \log \wt L_{0,n}^S \,|\, \Psi_{K,n}^S \right]   - \BB E\left[ \log \wt L_{0,n}  \,|\, \Psi_{K,n} \right]  \right)_+^2 \right] .
\alle
\end{lem}
\begin{proof}
Consider the measurable functional $F$ from continuous functions on $\BB S$ to $\BB R$ defined by
\eqb
F(\psi) := \BB E\left[ \log \wt L_{0,n} \,|\, \Psi_{K,n} =\psi \right] .
\eqe
Since $\Psi_{0,K}$ and $\Psi_{K,n}$ are independent, to define $F(\psi)$, we can first sample $\Psi_{0,K}$ from its marginal law then consider the LFPP metric on $\BB S$ with $\Psi_{0,K} + \psi$ in place of $\Psi_{0,n}$. 
From this description, we see that  
\eqb
\BB E\left[ \log \wt L_{0,n} \,|\, \Psi_{K,n} \right] = F(\Psi_{K,n}) \quad\text{and} \quad
\BB E\left[ \log \wt L_{0,n}^S \,|\, \Psi_{K,n}^S \right] = F(\Psi_{K,n}^S) , \quad\forall S \in \mcl S_K .
\eqe 
The random function $\Psi_{K,n}$ is the sum of the independent random functions $\psi_{K,n}^S$ of~\eqref{eqn-block-fn-def}. 
Consequently, we can apply the Efron-Stein inequality~\cite{efron-stein} to get~\eqref{eqn-efron-stein}. 
\end{proof}

\subsection{Outline of the proof}
\label{sec-ef-outline}

The rest of this section is devoted to bounding the right side of~\eqref{eqn-efron-stein}. 
To do this, we will fix $C > 1$ and $K \in\BB N$ and assume the inductive hypothesis
\eqb \label{eqn-ind-hyp}
\Lambda_{n-K}(\frk p) \leq e^{C  \sqrt K}   .
\eqe
We will show that~\eqref{eqn-ind-hyp} implies that the right side of~\eqref{eqn-efron-stein} is bounded above by $2^{-\alpha K + o_K(K)}$ for an exponent $\alpha >0$ depending on $\xi$ (see Proposition~\ref{prop-ef-sum}). By combining this estimate with Lemma~\ref{lem-cond-var} and~\eqref{eqn-var-decomp}, we will get that~\eqref{eqn-ind-hyp} implies that $\op{Var}\left[ \log \wt L_{0,n} \right] \preceq K$, with the implicit constant depending only on $\xi$. By Lemma~\ref{lem-induct-quantile}, this will show that~\eqref{eqn-ind-hyp} implies $\Lambda_n(\frk p) \leq e^{C\sqrt K}$ provided $C$ and $K$ are chosen to be sufficiently large, depending only on $\xi$. 

Before getting into the details, in the rest of this subsection we give an outline of how~\eqref{eqn-ind-hyp} leads to an upper bound for the right side of~\eqref{eqn-efron-stein}. 
To lighten notation, let
\eqbn
Z_S := \left( \BB E\left[ \log \wt L_{0,n}^S \,|\, \Psi_{K,n}^S \right]   - \BB E\left[ \log \wt L_{0,n}  \,|\, \Psi_{K,n} \right]\right)_+  
\eqen
be the quantity appearing inside the expectation on the right side of~\eqref{eqn-efron-stein}. 
By the Cauchy-Schwarz inequality,
\eqb \label{eqn-ef-split0}
\BB E\left[ \sum_{S\in\mcl S_K} Z_S^2 \right] 
\leq \BB E\left[\left(\sum_{S\in\mcl S_K} Z_S \right)^2 \right]^{1/2}   \BB E\left[  \left(\max_{S\in\mcl S_K} Z_S \right)^2 \right]^{1/2} .
\eqe
We will show that~\eqref{eqn-ind-hyp} implies upper bounds for each of the two factors on the right side of~\eqref{eqn-ef-split0}.
\medskip

\noindent\textbf{Section~\ref{sec-cond-exp-around}: bound for $Z_S$.}
In Section~\ref{sec-cond-exp-around}, we upper-bound $Z_S$ in terms of a quantity which is easier to work with. 
For each $2^{-K} \times 2^{-K}$ square $S\in\mcl S_K$, we let $A_S$ be an annulus surrounding $S$ with aspect ratio 2 and radius approximately $ 2^{-K} K^{\ep_0}$ (which we recall is an upper bound for the range of dependence for the function $\Psi_{K,n}$; see the discussion just after~\eqref{eqn-phi-def-trunc}).
See~\eqref{eqn-S-annuli} for a precise definition.

If $P$ is a path between the left and right sides of $\BB S$ which gets within Euclidean distance $2^{-K} K^{\ep_0}$ of $S$, then $P$ must cross $A_S$.
By replacing a segment of $P$ by a segment of a path around $A_S$, we get a new path between the left and right sides of $\BB S$ which stays away from $S$ and whose $\wt D_{0,n}$-length is at most the $\wt D_{0,n}$-length of $P$ plus $\wt D_{0,n}(\text{around $A_S$})$.
This leads to the a.s.\ bound
\eqb \label{eqn-cond-exp-around0} 
Z_S \leq \BB E\left[ \frac{1}{\wt L_{0,n}} \BB 1_{F_S} \wt D_{0,n}(\text{around $A_S$}) \,\big|\, \Psi_{K,n} \right] , 
\eqe
where $F_S$ is the event that the $\wt D_{0,n}$-geodesic between the left and right sides of $\BB S$ gets within distance $2^{-K} K^{\ep_0}$ of $ S$. See Lemma~\ref{lem-cond-exp-around} for a precise statement. 
\medskip

\noindent\textbf{Section~\ref{sec-ef-sum}: proof conditional on estimates for the right side of~\eqref{eqn-ef-split0}.}
We explain how to conclude the proof of Proposition~\ref{prop-quantile-ratio} assuming upper bounds for each of the two factors on the right side of~\eqref{eqn-ef-split0} (given in Propositions~\ref{prop-ef-length} and~\ref{prop-ef-max}). 
The rest of the section is devoted to the proofs of these two propositions.
\medskip
 
\noindent\textbf{Section~\ref{sec-across-around}: bounds for distances around and across annuli.}
With a view toward bounding the right side of~\eqref{eqn-cond-exp-around0}, we use our inductive hypothesis~\eqref{eqn-ind-hyp} and the a priori estimates from Proposition~\ref{prop-perc-estimate} to show that with high probability the following is true. For each $S \in\mcl S_K$, the $\wt D_{0,n}$-distance across $A_S$ and the $\wt D_{0,n}$-distance around $A_S$ are each comparable to $2^{-K} \lambda_{n-K} e^{\xi \Phi_{0,K}(v_S)}$, where $v_S$ is the center of $S$. See Lemma~\ref{lem-induct-full} for a precise statement. Combined with~\eqref{eqn-cond-exp-around0}, this shows that with high probability,
\eqb \label{eqn-induct-full0}
Z_S \leq 2^{-K + o_K(K)} \lambda_{n-K} \BB E\left[ \frac{1}{\wt L_{0,n}} \BB 1_{F_S}   e^{\xi \Phi_{0,K}(v_S)}  \,\big|\, \Psi_{K,n} \right] ,\quad\forall S\in\mcl S_K .
\eqe 
\medskip

\noindent\textbf{Section~\ref{sec-ef-length}: bound for $\sum_{S\in\mcl S_K} Z_S$}.
Our bound for the first factor on the right side of~\eqref{eqn-ef-split0} is stated in Proposition~\ref{prop-ef-length} and proven in Section~\ref{sec-ef-length}.
We first note that a $\wt D_{0,n}$-geodesic between the left and right boundaries of $\BB S$ must cross each annulus $A_S$ for $S\in\mcl S_K$ for which $F_S$ occurs.
After accounting for the overlap of the $S_K$'s, this implies that with high probability
\eqb  \label{eqn-around-sum0}
\sum_{S\in\mcl S_K} \BB 1_{F_S} \wt D_{0,n}(\text{across $A_S$}) \leq 2^{o_K(K)} \wt L_{0,n}  
\eqe 
(Lemma~\ref{lem-across-sum}). 
By the annulus estimates described just above,~\eqref{eqn-around-sum0} also holds with ``around" in place of ``across" (Lemma~\ref{lem-around-sum}). Using~\eqref{eqn-cond-exp-around0}, we therefore get that with high probability,
\eqb \label{eqn-ef-sum0}
  \sum_{S\in\mcl S_K} Z_S  \leq 2^{o_K(K)} .
\eqe 
\medskip

\noindent\textbf{Sections~\ref{sec-subadd} through \ref{sec-ef-max}: bound for $\max_{S\in\mcl S_K} Z_S$.} 
Our upper bound for the second factor on the right side of~\eqref{eqn-ef-split0} is more involved than the upper bound for the first factor, and is carried out in Sections~\ref{sec-subadd} through~\ref{sec-ef-max}.
This upper bound is based on the combination of two estimates. The first estimate (Lemma~\ref{lem-subadd}) says that with high probability,
\eqb \label{eqn-subadd0}
\wt L_{0,n} =  2^{-(1-\xi Q) K + o_K(K)}  \lambda_{n-K} 
\eqe
and is proven by looking at the times when a $D_{0,K}$ or $D_{0,n}$-geodesic between the left and right boundaries of $\BB S$ crosses the annuli $A_S$ (this is similar to the subadditivity argument of Lemma~\ref{lem-mu-exponent}). 

The second estimate is based on the following observation. 
The random variable $\Phi_{0,K}(v_S)$ is a centered Gaussian random variable of variance $K \log 2$ and is independent from $\Psi_{K,n}$.
On the event $F_S$ that the $\wt D_{0,n}$-geodesic between the left and right sides of $\BB S$ gets close to $ S$, the contribution to the $\wt D_{0,n}$-length of this geodesic coming from its crossing of $A_S$ must be at most $\wt L_{0,n}$. By~\eqref{eqn-induct-full0} and the lower bound for $\wt D_{0,n}(\text{across $A_S$})$ discussed above, this implies that on $F_S$, it holds with high probability that 
\eqb
2^{ - K + o_K(K)} \lambda_{n-K}    e^{\xi \Phi_{0,K}(v_S)}
\leq \wt D_{0,n}\left(\text{across $A_S$}\right) 
 \leq  \wt L_{0,n} 
 \leq 2^{-(1-\xi Q) K + o_K(K)}  \lambda_{n-K} .
\eqe
Re-arranging this bound shows that on $F_S$, typically $\Phi_{0,K}(v_S) \leq (Q + o_K(1)) K \log 2$.  

Therefore, with high probability $e^{\xi \Phi_{0,K}(v_S)} \BB 1_{F_S}$ is bounded above by $e^{\xi X} \BB 1_{X \leq (Q\log 2  +o_K(1)) K}$ where $X$ is a centered Gaussian random variable of variance $K \log 2$ which is independent from $\Psi_{K,n}$. 
By a straightforward calculation for the standard Gaussian distribution (Lemma~\ref{lem-gaussian-trunc}) this shows that $\BB E\left[ e^{\xi \Phi_{0,K}(v_S)} \BB 1_{F_S} \,|\, \Psi_{K,n} \right] \leq 2^{-[\xi (\xi \wedge Q) - (\xi \wedge Q)^2/2] K + o_K(K)}$. Combined with~\eqref{eqn-induct-full0} and~\eqref{eqn-subadd0}, this shows that with high probability
\allb \label{eqn-ef-max0}
\max_{S\in\mcl S_K} Z_S 
&\leq 2^{-K + o_K(K)} \lambda_{n-K} \max_{S\in\mcl S_K} \BB E\left[ \frac{1}{\wt L_{0,n}} \BB 1_{F_S}   e^{\xi \Phi_{0,K}(v_S)}  \,\big|\, \Psi_{K,n} \right] \notag\\
&\leq 2^{-K + o_K(K)} \lambda_{n-K} \times 2^{ (1-\xi Q) K + o_K(K)}  \lambda_{n-K}^{-1} \times 2^{-[\xi (\xi \wedge Q) - (\xi \wedge Q)^2/2] K + o_K(K)} \notag\\
& = 2^{-\alpha K + o_K(K)} 
\alle
where $\alpha $ is an explicit, positive, $\xi$-dependent constant. 
\medskip

\noindent\textbf{Conclusion.} Plugging~\eqref{eqn-ef-sum0} and~\eqref{eqn-ef-max0} into~\eqref{eqn-ef-split0} and then into~\eqref{eqn-efron-stein} shows that $\op{Var}\left[ \BB E\left[ \log \wt L_{0,n} \,|\, \Psi_{K,n} \right] \right] \leq 2^{-\alpha K + o_K(K)}$. 
Combined with~\eqref{eqn-var-decomp} and Lemma~\ref{lem-cond-var}, this gives $\op{Var}[\log\wt L_{0,n}] \preceq K$, which by Lemma~\ref{lem-induct-quantile} implies that $\Lambda_n(\frk p) \leq e^{C\sqrt K}$ provided $C$ is chosen to be large enough (depending only on $\xi$).  
This completes the induction, hence the proof of Proposition~\ref{prop-quantile-ratio}.

\begin{remark} \label{remark-annulus}
The estimates in Sections~\ref{sec-ef-length} through~\ref{sec-ef-max} are all consequences of the estimates for distances around and across annuli from Section~\ref{sec-across-around}, H\"older's inequality, and the fact that the ``scale $K$" field $\Phi_{0,K}$ is a centered Gaussian process independent from $\Psi_{K,n}$ with $\op{Var}\Phi_{0,K}(z) = K \log 2$ for each $z \in \BB C$. The use of this description of the conditional law of $\Phi_{0,K}$, rather than just estimates for the maximum of $\Phi_{0,K}$, is the key idea which allows us to establish the tightness of LFPP for all $\xi > 0$, not just for $\xi < \xi_{\op{crit}}$. 
More precisely, the description of the conditional law of $\Phi_{0,K}$ given $\Psi_{K,n}$ is used in the proofs of~\eqref{eqn-cond-exp-moment-D} of Lemma~\ref{lem-cond-exp-moment} and Step 2 in Section~\ref{sec-ef-max}, which are perhaps the most interesting parts of the argument in Sections~\ref{sec-ef-length} through~\ref{sec-ef-max}. 
\end{remark}

\subsection{Upper bound for Efron-Stein differences in terms of annulus crossings}
\label{sec-cond-exp-around}

Instead of estimating the differences of conditional expectations appearing in~\eqref{eqn-efron-stein} directly, we will instead estimate certain functionals of $\wt D_{0,n}$. 
To describe these functionals, let $m_K \in \BB N$ be chosen so that
\eqbn
2^{-m_K} \leq K^{\ep_0} \leq 2^{-m_K+1} . 
\eqen
For $S\in \mcl S_K$, let
\eqb \label{eqn-S-center}
v_S := \left(\text{center of $S$} \right) 
\eqe
and define the annulus 
\eqb \label{eqn-S-annuli}
A_S := \ol{ B_{2^{-(K-m_K) + 1}}(v_S)  \setminus B_{2^{- (K-m_K) }}(v_S) }   .
\eqe 
The reason for the definition of $A_S$ is that its aspect ratio is of constant order and the field $\Psi_{K,n} - \Psi_{K,n}^S$ vanishes outside of $B_{2^{-(K-m_K)}}(v_S)$ (see the discussion just after~\eqref{eqn-block-fn-def}). 

\begin{lem} \label{lem-cond-exp-around} 
Let $\wt P_{0,n} : [0, \wt L_{0,n}] \rta \BB S $ be a path in $\BB S$ between the left and right boundaries of $\BB S$ of minimal $\wt D_{0,n}$-length. For $S\in\mcl S_K$, let
\eqb \label{eqn-square-hit-event}
F_S := \left\{ \wt P_{0,n} \cap B_{2^{-m_K}}(v_S) \not=\emptyset \right\} .
\eqe
Almost surely,
\eqb \label{eqn-cond-exp-around}
\BB E\left[\log  \wt L_{0,n}^S \,|\, \Psi_{K,n}^S \right]   - \BB E\left[ \log \wt L_{0,n}  \,|\, \Psi_{K,n} \right]  
\leq  \BB E\left[ \frac{1}{\wt L_{0,n}} \BB 1_{F_S} \wt D_{0,n}\left( \text{around $A_S$} \right)  \,|\,   \Psi_{K,n}  \right] ,\quad \forall S\in\mcl S_K .
\eqe
\end{lem}

Once Lemma~\ref{lem-cond-exp-around} is established, we will never work with the left side of~\eqref{eqn-cond-exp-around} directly. Instead, we will prove bounds for the right side of~\eqref{eqn-cond-exp-around}.

\begin{figure}[t!]
 \begin{center}
\includegraphics[scale=1]{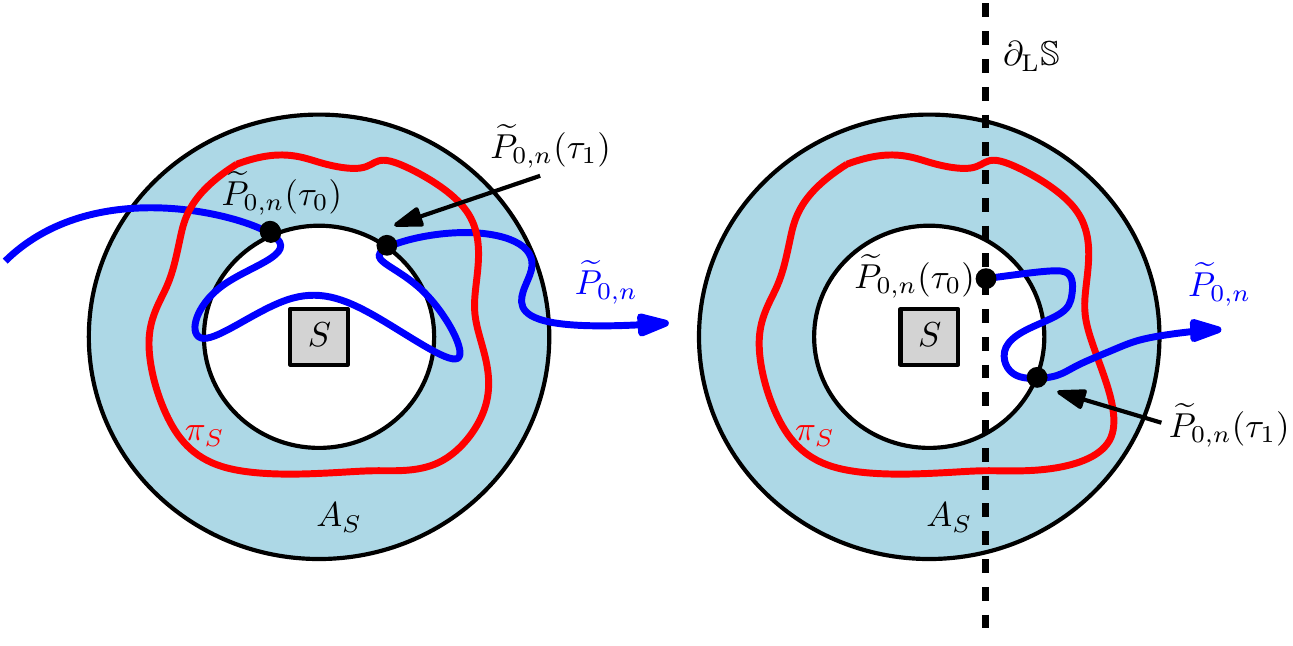}
\vspace{-0.01\textheight}
\caption{Illustration of the proof of Lemma~\ref{lem-cond-exp-around} in the case when $B_{2^{-(K-m_K)}}(v_S)$ does not intersect $\bdy \BB S$ (\textbf{left}) and the case when $B_{2^{-(K-m_K)}}(v_S)$ intersects $\bdy_{\op L} \BB S$ (\textbf{right}). In each case, the union of $\wt P_{0,n} |_{[0,\tau_0]}$, $\wt P_{0,n}|_{[\tau_1,\wt L_{0,n}]}$, and the path $\pi_S$ around $A_S$ contains a path in $\BB S$ between the left and right boundaries of $\BB S$ which does not intersect $B_{2^{-(K-m_K)}}(v_S)$. Note that in the figure on the right, $\tau_0 = 0$ and the starting point of the new path is not the same as the starting point of $\wt P_{0,n}$. 
}\label{fig-cond-exp-around}
\end{center}
\vspace{-1em}
\end{figure}

\begin{proof}[Proof of Lemma~\ref{lem-cond-exp-around}]
We first note that by the definition of $\Psi_{K,n}$, the fields $\Psi_{K,n}^S$ and $\Psi_{K,n}$ agree on $\BB C\setminus B_{2^{-(K-m_K)}}(v_S)$. 
Since $\Psi_{0,n}^S - \Psi_{K,n}^S = \Psi_{0,n}  - \Psi_{K,n}  =  \Psi_{0,K}$, it follows that $\Psi_{0,n}^S$ and $\Psi_{0,n}$ agree on $\BB C\setminus B_{2^{-(K-m_K)}}(v_S)$. 
We will now use this fact to compare $\wt L_{0,n}^S$ and $\wt L_{0,n}$.  

If $\wt P_{0,n} $ does not enter $B_{2^{-(K-m_K)}}(v_S)$ (i.e., $F_S$ does not occur), then the $\wt D_{0,n}^S$-length of $\wt P_{0,n}$ is the same as its $\wt D_{0,n}$-length, so $\wt L_{0,n}^S \leq \wt L_{0,n}$. 

If $\wt P_{0,n} $ enters $B_{2^{-(K-m_K)}}(v_S)$ (i.e., $F_S$ occurs), let $\tau_0$ be the first entrance time and let $\tau_1$ be the last exit time of $\wt P_{0,n} $ from $B_{2^{-(K-m_K)}}(v_S)$. 
Note that it is possible that $\tau_0 = 1$ or that $\tau_1 = \wt L_{0,n}$ if $v_S$ is within Euclidean distance $2^{-(K-m_K)}$ of the left or right boundary of $\BB S$. 
Also let $\pi_S$ be a path in $A_S$ of minimal $\wt D_{0,n} $-length, i.e., the $\wt D_{0,n}$-length of $\pi_S$ is $\wt D_{0,n}\left( \text{around $A_S$} \right)$. 
Then the union of $\wt P_{0,n} |_{[0,\tau_0]}$, $\pi_S$, and $\wt P_{0,n}  |_{[\tau_1, \wt L_{0,n} ]}$ contains a simple path in $\BB S$ between the left and right boundaries of $\BB S$ (see Figure~\ref{fig-cond-exp-around}). 
This path does not enter $B_{2^{-(K-m_K)}}(v_S)$, so its $\wt D_{0,n}$-length is the same as its $\wt D_{0,n}^S$-length. 
Therefore, on $F_S$, a.s.\ 
\eqbn
\wt L_{0,n}^S  
\leq \tau_0 + (\wt L_{0,n} - \tau_1) + \wt D_{0,n}\left( \text{around $A_S$} \right) 
\leq \wt L_{0,n} + \wt D_{0,n}\left( \text{around $A_S$} \right) .
\eqen

Combining the preceding two paragraphs gives that, a.s., 
\eqbn
\wt L_{0,n}^S  \leq \wt L_{0,n} + \BB 1_{F_S} \wt D_{0,n}\left( \text{around $A_S$} \right) . 
\eqen
By the mean value theorem, a.s.\ 
\allb \label{eqn-log-inequality}
\log \wt L_{0,n}^S - \log \wt L_{0,n}  
&\leq \log\left( \wt L_{0,n} + \BB 1_{F_S} \wt D_{0,n}\left( \text{around $A_S$} \right) \right) - \log \wt L_{0,n}  \notag\\
&\leq \wt L_{0,n}^{-1} \BB 1_{F_S} \wt D_{0,n}\left( \text{around $A_S$} \right)  .
\alle
The random variable $\log \wt L_{0,n} + \wt L_{0,n}^{-1} \BB 1_{F_S} \wt D_{0,n}\left( \text{around $A_S$} \right)$ is $\sigma( \Psi_{0,n})$-measurable, hence it is conditionally independent from $\Psi_{K,n}^S$ 
given $\Psi_{K,n}$. On the other hand, $\log \wt L_{0,n}^S$ is a measurable function of $\Psi_{0,n}^S$ so is conditionally independent from $\Psi_{K,n}$ given $\Psi_{K,n}^S$. Therefore,
\allb
\BB E\left[ \log \wt L_{0,n}^S \,|\, \Psi_{K,n}^S  \right]  
&= \BB E\left[ \log \wt L_{0,n}^S \,|\, \Psi_{K,n}^S , \Psi_{K,n}  \right]   \notag\\
&\leq \BB E\left[ \log  \wt L_{0,n} + \wt L_{0,n}^{-1} \BB 1_{F_S} \wt D_{0,n}\left( \text{around $A_S$} \right)  \,|\, \Psi_{K,n}^S , \Psi_{K,n}  \right] \quad \text{(by~\eqref{eqn-log-inequality})} \notag\\
&= \BB E\left[ \log  \wt L_{0,n} + \wt L_{0,n}^{-1} \BB 1_{F_S} \wt D_{0,n}\left( \text{around $A_S$} \right)  \,|\,   \Psi_{K,n}  \right]  .
\alle
Subtracting $ \BB E\left[ \log \wt L_{0,n}  \,|\, \Psi_{K,n} \right]$ now gives~\eqref{eqn-cond-exp-around}.
\end{proof}

\subsection{Proof of Proposition~\ref{prop-quantile-ratio} assuming moment estimates}
\label{sec-ef-sum}

As explained in Section~\ref{sec-ef-outline}, to bound $\BB E\left[ \op{Var}\left[ \log\wt L_{0,n} \,|\, \Psi_{K,n} \right] \right]$, and thereby prove Proposition~\ref{prop-quantile-ratio}, we will induct on $n$, taking~\eqref{eqn-ind-hyp} as our inductive hypothesis.  
The main estimate which we prove using~\eqref{eqn-ind-hyp} is the following proposition. For the statement, we introduce the $\xi$-dependent exponent 
\eqb \label{eqn-hit-exponent} 
\alpha(\xi) 
:= \begin{cases}
\xi Q - \frac{\xi^2}{2} ,\quad &\xi \leq Q \\
\frac{Q^2}{2}  ,\quad &\xi  > Q .
\end{cases}
\eqe
Note that $\alpha(\xi)  > 0$ since $Q > 0$ (Proposition~\ref{prop-Q-wn}).

\begin{prop} \label{prop-ef-sum}
Assume the inductive hypothesis~\eqref{eqn-ind-hyp}. Let $\alpha(\xi) $ for $\xi > 0$ be as in~\eqref{eqn-hit-exponent}. For each fixed $\zeta >0$, 
\eqb \label{eqn-ef-sum}
\BB E\left[ \sum_{S\in\mcl S_K} \left(  \BB E\left[ \frac{1}{\wt L_{0,n}} \BB 1_{F_S} \wt D_{0,n}\left( \text{around $A_S$} \right)  \,|\,   \Psi_{K,n}  \right]  \right)^2 \right] 
\preceq  2^{ - (\alpha(\xi) - \zeta)  K} ,
\eqe
with the implicit constant depending only on $\zeta , \xi , C$. 
\end{prop}

Note that due to Lemma~\ref{lem-cond-exp-around}, Proposition~\ref{prop-ef-sum} implies an upper bound for the right side of~\eqref{eqn-efron-stein} from Lemma~\ref{lem-efron-stein}. 
As we will explain just below, Proposition~\ref{prop-ef-sum} is an easy consequence of the following two propositions, whose proofs will occupy most of the rest of this section.

\begin{prop} \label{prop-ef-length}
Assume the inductive hypothesis~\eqref{eqn-ind-hyp}. We have
\eqb \label{eqn-ef-length}
\BB E\left[ \left( \sum_{S\in\mcl S_K} \BB E\left[ \BB 1_{F_S}  \frac{1}{\wt L_{0,n}} \wt D_{0,n}\left( \text{across $A_S$} \right) \,|\, \Psi_{n,K} \right] \right)^2 \right]^{1/2}  
\preceq  K \exp(K^{2/3})  ,
\eqe
with the implicit constant depending only on $ \xi , C$. 
\end{prop} 

\begin{prop} \label{prop-ef-max}
Assume the inductive hypothesis~\eqref{eqn-ind-hyp}.  
Let $\alpha(\xi) $ for $\xi > 0$ be as in~\eqref{eqn-hit-exponent}. For each fixed $\zeta >0$, 
\eqb \label{eqn-ef-max}
\BB E\left[ \left( \max_{S\in\mcl S_K} \BB E\left[ \BB 1_{F_S}  \frac{1}{\wt L_{0,n}} \wt D_{0,n}\left( \text{across $A_S$} \right) \,|\, \Psi_{n,K} \right] \right)^2 \right]^{1/2}
\preceq 2^{-(\alpha(\xi) - \zeta) K}
\eqe
with the implicit constant depending only on $\zeta , \xi , C$. 
\end{prop}

\begin{proof}[Proof of Proposition~\ref{prop-ef-sum}, assuming Propositions~\ref{prop-ef-length} and~\ref{prop-ef-max}]
To lighten notation, let
\eqb
X_S := \BB E\left[ \BB 1_{F_S}  \frac{1}{\wt L_{0,n}} \wt D_{0,n}\left( \text{across $A_S$} \right) \,\big|\, \Psi_{n,K} \right] .
\eqe
By the Cauchy-Schwarz inequality,
\eqb \label{eqn-ef-sum-split}
\BB E\left[ \sum_{S\in\mcl S_K} X_S^2 \right] 
\leq \BB E\left[   \left( \max_{S\in\mcl S_K} X_S \right) \left(   \sum_{S\in\mcl S_K}  X_S \right) \right] 
\leq \BB E\left[ \left( \max_{S\in\mcl S_K} X_S \right)^2 \right]^{1/2} \BB E\left[  \left(   \sum_{S\in\mcl S_K}  X_S \right)^2 \right]^{1/2} .
\eqe
We use Proposition~\ref{prop-ef-max} to bound the first factor on the right in~\eqref{eqn-ef-sum-split} and Proposition~\ref{prop-ef-length} to bound the second factor on the right in~\eqref{eqn-ef-sum}. This gives
\eqbn
\BB E\left[ \sum_{S\in\mcl S_K} X_S^2 \right] 
\preceq 2^{ - (\alpha(\xi) - \zeta)  K}  K \exp(K^{2/3})  .
\eqen
Slightly shrinking $\zeta$ to absorb the factor of $K\exp(K^{2/3})$ now gives~\eqref{eqn-ef-sum}. 
\end{proof}

Before proving Propositions~\ref{prop-ef-length} and~\ref{prop-ef-max}, we explain how to deduce Proposition~\ref{prop-quantile-ratio} from Proposition~\ref{prop-ef-sum}.

\begin{lem} \label{lem-var-cond}
Assume the inductive hypothesis~\eqref{eqn-ind-hyp}. 
Also let $\alpha(\xi)$ be as in~\eqref{eqn-hit-exponent} and fix $\zeta>0$. 
We have
\eqb \label{eqn-var-cond}
\op{Var}\left[    \BB E\left[ \log \wt L_{0,n} \,|\, \Psi_{K,n}  \right]     \right] 
\preceq  2^{ -(\alpha(\xi) - \zeta)  K} ,
\eqe
with the implicit constant depending only on $\zeta , \xi , C$. 
\end{lem}
\begin{proof}[Proof of Lemma~\ref{lem-var-cond}, assuming Proposition~\ref{prop-ef-sum}]
By plugging the estimate of Lemma~\ref{lem-cond-exp-around} into the estimate of Lemma~\ref{lem-efron-stein}, we get 
\allb  
\op{Var}\left[    \BB E\left[ \log \wt L_{0,n} \,|\, \Psi_{K,n}  \right]     \right] 
\leq   \BB E\left[ \sum_{S\in\mcl S_K} \left(  \BB E\left[ \frac{1}{\wt L_{0,n}} \BB 1_{F_S} \wt D_{0,n}\left( \text{around $A_S$} \right)  \,|\,   \Psi_{K,n}  \right]  \right)^2 \right] 
\alle
which we can then bound by means of Proposition~\ref{prop-ef-sum}. 
\end{proof}

\begin{proof}[Proof of Proposition~\ref{prop-quantile-ratio}, assuming Proposition~\ref{prop-ef-sum}]
Fix $\zeta> 0$, $K\in\BB N$,  and $C>0$ to be chosen later in a manner depending only on $\xi$. We proceed by induction on $n\in\BB N$ to show that
\eqb \label{eqn-ind-hyp-final} 
\Lambda_n(\frk p) \leq e^{C\sqrt K}  
\eqe
for all $n\in\BB N$. By Lemma~\ref{lem-a-priori-conc}, if $C > 0$ is chosen to be large enough (depending on $\xi$) then~\eqref{eqn-ind-hyp-final} holds for $n \in [1,K]_{\BB Z}$. This gives the base case.

For the inductive step, assume that $n\geq K+1$ and $\Lambda_{n-K}(\frk p) \leq e^{C\sqrt K}$. 
Let $M$ be the implicit constant in Lemma~\ref{lem-cond-var} (which depends only on $\xi$) and let $N_C$ be the implicit constant of Lemma~\ref{lem-var-cond} (which depends on $\zeta,\xi,C$). 
By plugging the estimates of Lemmas~\ref{lem-cond-var} and~\ref{lem-var-cond} into~\eqref{eqn-var-decomp}, we see that our inductive hypothesis implies that
\allb \label{eqn-use-var-decomp}
\op{Var}\left[ \log \wt L_{0,n} \right] 
\leq M K + N_C  2^{ - (\alpha(\xi) - \zeta) K} .
\alle
Since $\alpha(\xi) > 0$, we can choose $\zeta \in (0,\alpha(\xi))$. 
Henceforth fix such a $\zeta$.

There exists $K_C \in \BB N$ (depending on $\zeta,\xi,C$) such that if $K\geq K_C$, then $N_C 2^{ -(\alpha(\xi) - \zeta) K} \leq K$. 
For $K\geq K_C$, \eqref{eqn-use-var-decomp} implies that $\op{Var}\left[ \log \wt L_{0,n} \right]  \leq (M+1) K$. By Lemma~\ref{lem-induct-quantile}, this implies that
\eqb \label{eqn-induct-quantile'}
\frac{ \ell_n(1-\frk p)}{ \ell_n(\frk p)} \leq   \exp\left( c \sqrt{ (M+1) K}\right) 
\eqe
for a constant $c>0$ depending only on $\xi$. 
Since $\Lambda_{n'-K}(\frk p) \leq \Lambda_{n-K}(\frk p)$ for any $n'\leq n$, we infer that if $\Lambda_{n-K}(\frk p) \leq e^{C\sqrt K}$ then~\eqref{eqn-induct-quantile'} holds with $n$ replaced by any $n'\leq n$ and hence $\Lambda_n(\frk p) \leq  \exp\left( c \sqrt{ (M+1) K}\right)$. 

Therefore, if the constant $C$ from~\eqref{eqn-ind-hyp} is chosen to be at least $c \sqrt{M+1}$ (note that this last quantity depends only on $\xi$) then so long as $K\geq K_C$ we have that $\Lambda_{n-K}(\frk p) \leq e^{C\sqrt K}$ implies $\Lambda_n(\frk p) \leq e^{C\sqrt K} $.  
This completes the induction, so we obtain~\eqref{eqn-ind-hyp-final} for every $n\in\BB N$. 

Since $C$ and $\zeta$ have each been chosen in a manner depending only on $\xi$ and $K$ has been chosen in a manner depending only on $\zeta,\xi,C$ (hence only on $\xi$), the constant $e^{C\sqrt K}$ depends only on $\xi$. Thus~\eqref{eqn-quantile-ratio} holds. 
\end{proof}

\subsection{Bounds for distances around and across annuli}
\label{sec-across-around}

A key tool in our proofs of Propositions~\ref{prop-ef-length} and~\ref{prop-ef-max} are bounds for the $D_{0,n}$ and $\wt D_{0,n}$ distances across and around the annuli $A_S$ from~\eqref{eqn-S-annuli}, as given by the following lemma. 
For the statement, we recall that $v_S$ is the center of $S$.

\begin{lem} \label{lem-induct-full}
Assume the inductive hypothesis~\eqref{eqn-ind-hyp}.
For $T > 2^{K^{1/2 + 3\ep_0} }$, 
\eqb \label{eqn-induct-full-across}
\BB P\left[ \min_{S\in\mcl S_K} e^{-\xi \Phi_{0,K}(v_S)} D_{0,n} \left( \text{across $A_S$} \right)  < T^{-1} 2^{- K}   \lambda_{n-K} \right] 
\leq c_0  5^K \exp\left( - c_1 \frac{(\log T)^2 }{ K^{3\ep_0} } \right)   
\eqe 
and
\eqb \label{eqn-induct-full-around}
\BB P\left[ \max_{S\in\mcl S_K}  e^{-\xi \Phi_{0,K}(v_S)}  D_{0,n} \left( \text{around $A_S$} \right)  > T 2^{-K} \lambda_{n-K}  \right] 
\leq c_0 5^K  \exp\left( - c_1 \frac{(\log T)^2 }{  K^{3\ep_0}  \log\log T    } \right) ,
\eqe
where $c_0,c_1 >0$ are constants depending only on $\xi$ and $C$. Moreover, the same estimates also hold with $  D_{0,n}$ replaced by $\wt D_{0,n}$ and/or with  $A_S$ replaced by the smaller annulus
\eqb \label{eqn-S-annuli'}
A_S' := \ol{ B_{2^{-K}}(S) \setminus S } . 
\eqe
\end{lem}

To prove Lemma~\ref{lem-induct-full}, we first need the following trivial extension of Proposition~\ref{prop-perc-estimate} where we do not require that the sets under consideration are of constant-order size.

\begin{lem} \label{lem-perc-estimate-scale}
Let $\xi > 0$ and let $\frk p$ be as in Proposition~\ref{prop-perc-estimate}. 
Let $U\subset \BB C$ be an open set and let $K_1,K_2\subset U$ be disjoint compact connected sets which are not singletons. 
There are constants $c_0,c_1 >0$ depending only on $U,K_1,K_2,\xi$ such that for $n\in\BB N$, each $R > 1$, and each $T > 3$, 
\eqb \label{eqn-perc-lower-scale}
\BB P\left[ D_{0,n}(R K_1, R K_2 ; R U)  < T^{-1} \ell_n(\frk p) \right] \leq c_0 R e^{-c_1 (\log T)^2} 
\eqe
and
\eqb \label{eqn-perc-upper-scale}
\BB P\left[ D_{0,n}(R K_1, R K_2 ; R U)  > T R \Lambda_n(\frk p) \ell_n(\frk p) \right] \leq c_0 R e^{-c_1 (\log T  )^2 / \log\log T} .
\eqe  
\end{lem}
\begin{proof}
To prove~\eqref{eqn-perc-lower-scale}, choose (in a manner depending only on $U,K_1,K_2$) a smooth, bounded path $\Pi$ in $U$ which disconnects $K_1$ from $K_2$. 
There are constants $ 0  < c < C$ depending only on $U,K_1,K_2$ such that for each $R>1$, we can cover $\Pi$ by $m \leq C R$ Euclidean balls $B_{c}(z_1),\dots,B_{c}(z_m)$ such that the balls $B_{2c}(z_i)$ for $i \in [1,m]_{\BB Z}$ are disjoint from $K_1$ and $K_2$. 
By Proposition~\ref{prop-perc-estimate}, the translation invariance of the law of $\Phi_{0,n}$, and a union bound over $i\in [1,m]_{\BB Z}$, there are constants $c_0,c_1$ as in the lemma statement such that 
\eqbn \label{eqn-perc-lower-scale0}
\BB P\left[ D_{0,n}\left( \text{across $B_{2c}(z_i) \setminus B_{c}(z_i)$}\right) \geq   T^{-1} \ell_n(\frk p) ,\: \forall  i \in [1,m]_{\BB Z} \right] 
\geq 1 - c_0 R e^{-c_1 (\log T)^2} .
\eqen
Every path in $U$ from $K_1$ to $K_2$ must cross between the inner and outer boundaries of one of the annuli $B_{2c}(z_i) \setminus B_c(z_i)$. This gives~\eqref{eqn-perc-lower-scale}.  
 
To prove~\eqref{eqn-perc-upper-scale}, choose (in a manner depending only on $U,K_1,K_2$) a smooth path $\Pi$ in $U$ from $K_1$ to $K_2$. 
Since $K_1$ and $K_2$ are connected and not singletons, there are constants $ 0  < c < C$ depending only on $U,K_1,K_2$ such that for each $R>1$, we can cover $\Pi$ by $m \leq C R$ Euclidean balls $B_{c}(z_1),\dots,B_{c}(z_m)$ such that the balls $B_{2c}(z_i)$ for $i \in [1,m]_{\BB Z}$ are contained in $U$; any path separating the inner and outer boundaries of $B_{2c}(z_1) \setminus B_c(z_1)$ must intersect $K_1$; and any path separating the inner and outer boundaries of $B_{2c}(z_m) \setminus B_c(z_m)$ must intersect $K_2$. By Proposition~\ref{prop-perc-estimate} (applied with $C^{-1} T$ in place of $T$), the translation invariance of the law of $\Phi_{0,n}$, and a union bound over $i\in [1,m]_{\BB Z}$, there are constants $c_0,c_1$ as in the lemma statement such that 
\eqb \label{eqn-perc-upper-scale0}
\BB P\left[ D_{0,n}\left( \text{around $B_{2c}(z_i) \setminus B_{c}(z_i)$}\right) \leq C^{-1} T \Lambda_n(\frk p) \ell_n(\frk p) ,\: \forall  i \in [1,m]_{\BB Z} \right] 
\geq 1 - c_0 R e^{-c_1 (\log T )^2 / \log\log T} .
\eqe
If $\pi_i$ is a path around the annulus $B_{2c}(z_i) \setminus B_{c}(z_i)$ for each $i \in [1,m]_{\BB Z}$, then the union of the $\pi_i$'s contains a path in $U$ from $K_1$ to $K_2$. Since $m\leq C R$, if the event in~\eqref{eqn-perc-lower-scale0} holds then we can find such a path with $D_{0,n}$-length at most $ T  R \Lambda_n(\frk p) \ell_n(\frk p)$.    This gives~\eqref{eqn-perc-upper-scale}.   
\end{proof}

We can now establish bounds for the $D_{K,n}$-distances across and around $A_S$ and $A_S'$. 

\begin{lem} \label{lem-induct-across-around}
Assume the inductive hypothesis~\eqref{eqn-ind-hyp}.
For each $T > 2^{\sqrt K}$, 
\eqb \label{eqn-induct-across}
\BB P\left[ \min_{S\in\mcl S_K}  D_{K,n} \left( \text{across $A_S$} \right)  < T^{-1} 2^{-K } \lambda_{n-K} \right] 
\leq c_0  5^K \exp\left( - c_1 (\log T)^2  \right)   
\eqe 
and
\eqb \label{eqn-induct-around}
\BB P\left[ \max_{S\in\mcl S_K} D_{K,n} \left( \text{around $A_S$} \right)  > T 2^{-K  } \lambda_{n-K} \right] 
\leq c_0   5^K \exp\left( - c_1 \frac{(\log T)^2}{ \log \log T } \right) ,
\eqe
where $c_0,c_1 >0$ are constants depending only on $\xi$ and $C$. Moreover, the bounds~\eqref{eqn-induct-across} and~\eqref{eqn-induct-around} also hold with $A_S'$ from~\eqref{eqn-S-annuli'} in place of $A_S$. 
\end{lem}
\begin{proof}
Throughout the proof $c_0,c_1$ denote constants which depend only on $\xi$ and $C$ and which are allowed to change from line to line. 
Basically, the idea of the proof is to re-scale space by a factor of $2^K$, then apply Lemma~\ref{lem-perc-estimate-scale} to each of the annuli $2^K A_S'$, which have size of order $2^{m_K}$. We will then take a union bound over all $S\in\mcl S_K$ to conclude.
\medskip

\noindent
\textit{Step 1: re-scaling space.}
By the scale and translation invariance properties of $\Phi_{K,n}$ (see~\eqref{eqn-phi-scale}), 
\eqb \label{eqn-induct-field-scale}
\Phi_{K,n}( \cdot)  \eqD   \Phi_{0,n-K}  \left( 2^K \cdot  + v_S\right). 
\eqe 
Therefore,
\eqb \label{eqn-induct-dist-scale}
D_{K,n}\left( \text{across $A_S$} \right)  \eqD 2^{-K} D_{0 , n-K}\left( \text{across $B_{2^{m_K+1}}(0) \setminus B_{2^{m_K}}(0)$} \right)   
\eqe 
and the same holds with ``around" instead of ``across". 
\medskip

\noindent
\textit{Step 2: proof of~\eqref{eqn-induct-across}.} 
By~\eqref{eqn-perc-lower-scale} of Lemma~\ref{lem-perc-estimate-scale} (applied with $R = 2^{m_K}$), 
\eqb \label{eqn-induct-across0}
\BB P\left[ D_{0 , n-K}\left( \text{across $B_{2^{m_K+1}}(0) \setminus B_{2^{m_K}}(0)$} \right) < T^{-1} \ell_{n-K}(\frk p) \right] 
\leq c_0 2^{m_K} \exp\left( -c_1 (\log T)^2 \right)  .
\eqe
We now apply~\eqref{eqn-induct-across0} to estimate the right side of~\eqref{eqn-induct-dist-scale} for each $S\in \mcl S_K$, then take a union bound over all $S\in \mcl S_K$. This gives
\eqb \label{eqn-induct-across1}
\BB P\left[ \min_{S\in\mcl S_K} D_{K,n} \left( \text{across $A_S$} \right)  <  T^{-1-\xi} 2^{-K}   \ell_{n-K}(\frk p) \right] 
\leq c_0 2^{2K + m_K} \exp\left( - c_1  (\log T)^2  \right)  .
\eqe
To simplify the estimate~\eqref{eqn-induct-across1}, we first use~\eqref{eqn-ind-hyp} to get that $\ell_{n-K}(\frk p) \geq e^{-C\sqrt K} \ell_{n-K}(1/2)  = e^{-C\sqrt K} \lambda_{n-K}$. We also note that since $m_K \leq \log_2 K$, we have $2^{2K + m_K} \leq 5^K$. Therefore,~\eqref{eqn-induct-across1} implies that
\eqb \label{eqn-induct-across2}
\BB P\left[ \min_{S\in\mcl S_K} D_{K,n} \left( \text{across $A_S$} \right)  <  T^{-1-\xi} 2^{-K - C\sqrt K} \lambda_{n-K} \right] 
\leq c_0 5^K \exp\left( - c_1  (\log T)^2  \right)  .
\eqe  
We now apply~\eqref{eqn-induct-across2} with $T$ replaced by $(T 2^{-C\sqrt K})^{1/(1+\xi)}$, which is bounded above and below by $\xi,C$-dependent constants times $\xi,C$-dependent powers of $T$ since $T > 2^{\sqrt K}$. This gives~\eqref{eqn-induct-across} after possibly adjusting $c_0$ and $c_1$. 
\medskip

\noindent
\textit{Step 3: proof of~\eqref{eqn-induct-around}.} 
The proof of~\eqref{eqn-induct-around} is similar to the proof of~\eqref{eqn-induct-across}. By~\eqref{eqn-perc-upper-scale} of Lemma~\ref{lem-perc-estimate-scale},  
\allb \label{eqn-induct-around0}
&\BB P\left[ D_{0 , n-K}\left( \text{across $B_{2^{m_K+1}}(0) \setminus B_{2^{m_K}}(0)$} \right) > T 2^{m_K} \Lambda_{n-K}(\frk p) \ell_{n-K}(\frk p) \right] \notag\\ 
&\qquad\qquad\qquad \leq c_0 2^{m_K} \exp\left( - c_1 \frac{(\log T)^2 }{\log \log T} \right) .
\alle
By applying~\eqref{eqn-induct-around0} to estimate the right side of \eqref{eqn-induct-dist-scale}, then taking a union bound over all $S\in\mcl S_K$, we obtain
\allb \label{eqn-induct-around1}
&\BB P\left[ \max_{S\in\mcl S_K} D_{K,n} \left( \text{around $A_S$} \right)  >  T^{1+\xi} 2^{-(K-m_K)} \Lambda_{n-K}(\frk p) \ell_{n-K}(\frk p)  \right] \notag\\
&\qquad \qquad \qquad \leq  c_0   2^{2K + m_K} \exp\left( - c_1 \frac{(\log T)^2}{ \log\log T } \right)  .
\alle
To simplify this estimate, we apply~\eqref{eqn-ind-hyp} to get $\Lambda_{n-K}(\frk p) \ell_{n-K}(\frk p) \leq e^{  C\sqrt K} \lambda_{n-K}$, use that $2^{m_K} \leq K \leq e^{\sqrt  K}$ to absorb the factor of $2^{m_K}$ into a factor of $e^{\sqrt K}$, and apply the bound $2^{2K + m_K} \leq 5^K$ as above.
This gives
\allb \label{eqn-induct-around2}
&\BB P\left[ \max_{S\in\mcl S_K} D_{K,n} \left( \text{around $A_S$} \right)  >  T^{1+\xi} 2^{-K + ( C+1)\sqrt K} \Lambda_{n-K}(\frk p) \ell_{n-K}(\frk p)  \right] \notag\\
&\qquad\qquad\qquad \leq c_0   5^K \exp\left( - c_1 \frac{(\log T)^2}{ \log \log T } \right)  .
\alle
We apply this last estimate with $T$ replaced by by $(T 2^{-( C+1)\sqrt K})^{1/(1+\xi)}$ to get~\eqref{eqn-induct-around}.  

The proof of~\eqref{eqn-induct-across} and~\eqref{eqn-induct-around} with $A_S'$ in place of $A_S$ is essentially identical, except that we are working with an annulus of size of order $2^{-K}$ instead of $2^{-(K-m_K)}$ so we do not have to worry about extraneous factors of $2^{m_K}$.  
\end{proof}

\begin{proof}[Proof of Lemma~\ref{lem-induct-full}] 
Basically, the idea of the proof is that adding the coarse field $\Phi_{0,K}$ to $\Phi_{K,n}$ in order to get $\Phi_{0,n}$ scales distances in each annulus $A_S$ by approximately $e^{\xi \Phi_{0,K}(v_S)}$. Since $\Phi_{0,K}$ is not constant on $A_S$, we need some basic modulus of continuity estimates to compare the maximal and minimal values of $\Phi_{0,K}$ on $A_S$ to $\Phi_{0,K}(v_S)$, which we now explain.

By Lemma~\ref{lem-phi-cont},
\eqb \label{eqn-use-phi-grad}
\BB P\left[ 2^{-K} \sup_{z\in \BB S} |\nabla \Phi_{0,K}(z)| > \log T \right] \leq c_0 4^K e^{-c_1 (\log T)^2 } .
\eqe
For each $S\in\mcl S_K$, each point of $A_S \cap S$ is joined to $v_S$ by a line segment in $\BB S$ of length at most $2^{-(K-m_K)+1}$. By the mean value theorem and~\eqref{eqn-use-phi-grad}, 
\eqb \label{eqn-annulus-osc0}
\BB P\left[ \max_{S\in \mcl S_K} \sup_{z\in A_S} |\Phi_{0,K}(z) - \Phi_{0,K}(v_S)| > 2^{m_K+1} \log T \right] \leq c_0 4^K e^{-c_1 (\log T)^2} .
\eqe
Since $2^{m_K} \leq 2 K^{\ep_0}$, we obtain from~\eqref{eqn-annulus-osc0} that
\eqb \label{eqn-annulus-osc}
\BB P\left[ \max_{S\in \mcl S_K} \sup_{z\in A_S} |\Phi_{0,K}(z) - \Phi_{0,K}(v_S)| \leq  4 K^{\ep_0} \log T \right] \leq c_0 4^K e^{-c_1 (\log T)^2 } .
\eqe

By combining~\eqref{eqn-induct-across} and~\eqref{eqn-annulus-osc}, we therefore get that for $T>2^{\sqrt K}$, it holds with probability at least $1 - c_0 5^K \exp\left( - c_1 (\log T)^2   \right)$ that for each $S\in\mcl S_K$, 
\allb \label{eqn-induct-full-across0}
  D_{0,n} \left( \text{across $A_S$} \right) 
&\geq \exp\left( \xi \min_{z\in A_S} \Phi_{0,K} \right) D_{K,n}\left(\text{across $A_S$} \right) \quad \text{(since $\Phi_{0,n} = \Phi_{0,K} + \Phi_{K,n}$)} \notag\\
&\geq \exp\left( \xi \Phi_{0,K}(v_S) -   4 \xi K^{\ep_0} \log T \right) D_{K,n}\left(\text{across $A_S$} \right) \quad \text{(by~\eqref{eqn-annulus-osc})} \notag\\
&\geq T^{-1-4 \xi K^{\ep_0}} 2^{- K }  e^{\xi \Phi_{0,K}(v_S)} \lambda_{n-K} \quad \text{(by~\eqref{eqn-induct-across})} . 
\alle
If $T > 2^{K^{1/2 + 3\ep_0}}$, we can apply this last bound with $T^{1/(1+4\xi K^{\ep_0})}$ in place of $T$ to get
\eqbn 
\BB P\left[ \min_{S\in\mcl S_K} e^{-\xi \Phi_{0,K}(v_S)} D_{0,n} \left( \text{across $A_S$} \right)  < T^{-1} 2^{- K}   \lambda_{n-K} \right] 
\leq c_0 5^K \exp\left( - c_1 \frac{ (\log T)^2 }{ (1 + 4 \xi K^{\ep_0})^2 \log K} \right)   .
\eqen 
Since $(1+4\xi K^{\ep_0})^2 \log K$ is bounded above by a constant times $K^{3\ep_0}$, this implies~\eqref{eqn-induct-full-across}.

We similarly obtain~\eqref{eqn-induct-full-around} using~\eqref{eqn-induct-around} instead of~\eqref{eqn-induct-across}. The estimates with $A_S'$ in place of $A_S$ follow from the same argument, using the analogs of~\eqref{eqn-induct-across} and~\eqref{eqn-induct-around} with $A_S'$ in place of $A_S$. 
The estimates with $\wt D_{0,n}$ in place of $D_{0,n}$ follow from the estimates for $D_{0,n}$ together with Lemma~\ref{lem-field-compare}. 
\end{proof}

\subsection{Proof of Proposition~\ref{prop-ef-length}}
\label{sec-ef-length}

In this section we will prove Proposition~\ref{prop-ef-length}, which is the easier of the two unproven propositions in Section~\ref{sec-ef-sum}.
The idea of the proof is to show that  $\sum_{S\in\mcl S_K} \BB 1_{F_S}  \wt D_{0,n}\left( \text{around $A_S$} \right)  $ is very unlikely to be much larger than $\wt L_{0,n}$. This will be done in two steps. First, we show the analogous statement with ``across $A_S$" in place of ``around $A_S$" using the fact that the $\wt D_{0,n}$-geodesic $\wt P_{0,n}$ must cross each annulus $A_S$ for which $F_S$ occurs (Lemma~\ref{lem-across-sum}). 
Then, we show that $ \wt D_{0,n}\left( \text{around $A_S$} \right)  /   \wt D_{0,n}\left( \text{across $A_S$} \right) $ is small with high probability using Lemma~\ref{lem-induct-full} (Lemma~\ref{lem-around-sum}).

\begin{lem} \label{lem-across-sum}
Let $\wt P_{0,n}$ be the path as in Lemma~\ref{lem-cond-exp-around} and define the event $F_S$ for $S\in\mcl S_K$ as in~\eqref{eqn-square-hit-event}. 
Almost surely,
\eqb \label{eqn-across-sum}
\sum_{S\in\mcl S_K} \BB 1_{F_S} \wt D_{0,n}\left(\text{across $A_S$} \right)  \leq  K   \wt L_{0,n} .
\eqe
\end{lem}
\begin{proof}
If $F_S$ occurs, then $\wt P_{0,n}$ enters the region surrounded by the annulus $A_S$, so $\wt P_{0,n}$ must cross between the inner and outer boundaries of $A_S$ at least once (it must cross between the inner and outer boundaries twice if $A_S$ does not surround the starting or ending point of $\wt P_{0,n}$). 
For each $S\in\mcl S_K$ for which $F_S$ occurs, let $  [a_S, b_S] \subset [0,\wt L_{0,n}]$ be a time interval such that $\wt P_{0,n}( [a_S, b_S]) \subset \ol{A_S}$ and $\wt P_{0,n}(a_S)$ and $\wt P_{0,n}(b_S)$ lie on opposite boundary circles of $A_S$. 
If $F_S$ does not occur, instead set $a_S = b_S  =0$. 
Then
\eqb \label{eqn-interval-across}
b_S - a_S \geq \BB 1_{F_S} \wt D_{0,n}\left(\text{across $A_S$} \right) .
\eqe
We want to prove~\eqref{eqn-across-sum} by summing~\eqref{eqn-interval-across} over all $S\in\mcl S_K$. However, we need to deal with the potential for overlap between the intervals $ [a_S, b_S]$. 

For $S,S' \in \mcl S_K$, the intervals $ [a_S, b_S]$ and $ [a_{S'} , b_{S'}]$ can intersect only if $A_S\cap A_{S'}  \not=\emptyset $, which can happen only if the Euclidean distance between $S$ and $S'$ is at most $2^{-(K-m_K)+1}$. Since $\mcl S_K$ consists of dyadic squares of since length $2^{-K}$, it follows that for each $S\in\mcl S_K$ there are at most a constant times $2^{2m_K}$ squares $S'\in \mcl S_K$ for which $ [a_S, b_S] \cap  [a_{S'}, b_{S'}]  \not=\emptyset$. 
Therefore, 
\eqbn
\wt L_{0,n} \geq \sum_{S\in\mcl S_K} 2^{-2m_K} (b_S - a_S) \geq \sum_{S\in\mcl S_K} 2^{-2m_K} \BB 1_{F_S} \wt D_{0,n}\left(\text{across $A_S$} \right) .
\eqen
Recalling that $m_K \leq \log_2 K^{\ep_0} \leq \frac12 \log_2 K$, we now obtain~\eqref{eqn-across-sum}. 
\end{proof}

Now we upgrade the statement of Lemma~\ref{lem-across-sum} from a bound for distances across annuli to a bound for distances around annuli using Lemma~\ref{lem-induct-full}.

\begin{lem} \label{lem-around-sum}
Assume the inductive hypothesis~\eqref{eqn-ind-hyp}. 
Let $\wt P_{0,n}$ be the path as in Lemma~\ref{lem-cond-exp-around} and define the event $F_S$ for $S\in\mcl S_K$ as in~\eqref{eqn-square-hit-event}. 
For each $T> 2^{K^{1/2 + 3\ep_0}}$, 
\eqb \label{eqn-around-sum}
\BB P\left[ \sum_{S\in\mcl S_K} \BB 1_{F_S} \wt D_{0,n}\left(\text{around $A_S$} \right)   >    T K  \wt L_{0,n}  \right] 
\leq c_0 5^K  \exp\left( - c_1 \frac{(\log T)^2 }{K^{3\ep_0}  \log\log T   } \right) 
\eqe
where $c_0,c_1 > 0$ are constants depending only on $\xi$ and $C$. 
\end{lem}
\begin{proof}
By combining the two estimates from Lemma~\ref{lem-induct-full} (each applied with $T^{1/2}$ in place of $T$), we get
\eqb \label{eqn-induct-full-ratio}
\BB P\left[ \max_{S\in\mcl S_K}  \frac{ \wt D_{0,n} \left( \text{around $A_S$} \right) }{ \wt D_{0,n} \left(\text{across $A_S$} \right)}  > T  \right] 
\leq c_0 5^K  \exp\left( - c_1 \frac{(\log T)^2 }{K^{3\ep_0}  \log\log T   } \right)   .
\eqe
Combining~\eqref{eqn-induct-full-ratio} with Lemma~\ref{lem-across-sum} gives~\eqref{eqn-around-sum}.
\end{proof}

By integrating the bound from Lemma~\ref{lem-around-sum}, we can convert from a probability estimate to a moment estimate.

\begin{lem} \label{lem-around-sum-moment}
Assume the inductive hypothesis~\eqref{eqn-ind-hyp}. 
Let $\wt P_{0,n}$ be the path as in Lemma~\ref{lem-cond-exp-around} and define the event $F_S$ for $S\in\mcl S_K$ as in~\eqref{eqn-square-hit-event}. 
For each $T>1$, 
\eqb \label{eqn-around-moment}
\BB E\left[ \left(  \frac{1}{\wt L_{0,n}}   \sum_{S\in\mcl S_K}  \BB 1_{F_S} \wt D_{0,n}\left(\text{around $A_S$} \right)   \right)^2  \right] 
\preceq K^2 \exp\left( 2 K^{2/3} \right)  
\eqe
with the implicit constant depending only on $\xi$ and $C$. 
\end{lem}
\begin{proof}
By Lemma~\ref{lem-around-sum} (for $T \geq \exp(K^{2/3})$) and a trivial bound (for $K \in [0 , \exp(K^{2/3}) )$) we have
\eqbn
\BB P\left[  \frac{1}{K \wt L_{0,n}} \sum_{S\in\mcl S_K}  \BB 1_{F_S} \wt D_{0,n}\left(\text{around $A_S$} \right)   > T \right] 
\leq 
\begin{cases}
c_0 5^K  \exp\left( - c_1 \frac{(\log T)^2 }{ K^{3\ep_0}  \log\log T     } \right) \quad & T \geq \exp(K^{2/3} ) \notag\\
1 \quad &T < \exp(K^{2/3}) .
\end{cases}
\eqen
Integrating this estimate over $T\in [0,\infty)$ gives
\allb
&\BB E\left[ \left(   \frac{1}{K \wt L_{0,n} } \sum_{S\in\mcl S_K}  \BB 1_{F_S} \wt D_{0,n}\left(\text{around $A_S$} \right)   \right)^2  \right] \notag \\ 
&\qquad = 2 \int_0^\infty T \BB P\left[  \frac{1}{K \wt L_{0,n}} \sum_{S\in\mcl S_K}  \BB 1_{F_S} \wt D_{0,n}\left(\text{around $A_S$} \right)   > T \right] \,dT \notag \\
&\qquad \preceq \exp\left( 2 K^{2/3} \right)  + 5^K \int_{\exp(K^{2/3})}^\infty   \exp\left( - c_1 \frac{(\log T)^2 }{ K^{3\ep_0}  \log\log T     } \right) \,dT \notag \\
&\qquad \preceq \exp\left( 2 K^{2/3} \right) + 5^K \exp\left( - c_1 K^{4/3 -3\ep_0} / \log\log K \right) \notag\\
&\qquad \preceq \exp\left( 2 K^{2/3} \right) .
\alle
\end{proof}

\begin{proof}[Proof of Proposition~\ref{prop-ef-length}] 
By Jensen's inequality (to move the power of 2 inside the conditional expectation) followed by Lemma~\ref{lem-around-sum-moment},
\allb \label{eqn-ef-term2}
&\BB E\left[  \left(   \sum_{S\in\mcl S_K}  \BB E\left[ \frac{1}{\wt L_{0,n}} \BB 1_{F_S} \wt D_{0,n}\left( \text{around $A_S$} \right)  \,\big|\,   \Psi_{K,n}  \right]  \right)^2 \right]  \notag\\
&\qquad \leq \BB E\left[ \BB E\left[ \left(  \sum_{S\in\mcl S_K} \frac{1}{\wt L_{0,n}} \BB 1_{F_S} \wt D_{0,n}\left( \text{around $A_S$} \right) \right)^2   \,\big|\,   \Psi_{K,n}  \right] \right] \notag\\
&\qquad \leq \BB E\left[ \left(  \frac{1}{\wt L_{0,n}}   \sum_{S\in\mcl S_K}  \BB 1_{F_S} \wt D_{0,n}\left(\text{around $A_S$} \right)   \right)^2  \right] \notag\\ 
&\qquad \preceq K^2 \exp\left( 2 K^{2/3} \right) .
\alle
Taking the $1/2$ power of both sides of~\eqref{eqn-ef-term2} now gives~\eqref{eqn-ef-length}. 
\end{proof}

\subsection{Bound for left-right crossing distance}
\label{sec-subadd}

In this section we use Lemma~\ref{lem-induct-full} to prove a bound for the left-right crossing distance $\wt L_{0,n}$, assuming the inductive hypothesis.

\begin{lem} \label{lem-subadd}
Assume the inductive hypothesis~\eqref{eqn-ind-hyp}. 
Also fix $\zeta>0$. 
For each $T > 2^{K^{1/2 + 3\ep_0}}$, 
\eqb \label{eqn-subadd-lower}
\BB P\left[  \wt L_{0,n} <   T^{-1}  2^{-(1-\xi Q + \zeta) K} \lambda_{n-K}   \right] 
\leq c_0 5^K  \exp\left( - c_1 \frac{(\log T)^2 }{ K^{3\ep_0}  \log\log T     } \right)
\eqe
and
\eqb \label{eqn-subadd-upper}
\BB P\left[  \wt L_{0,n}  > T 2^{-(1-\xi Q - \zeta) K} \lambda_{n-K}   \right] 
\leq c_0 5^K  \exp\left( - c_1 \frac{(\log T)^2 }{ K^{3\ep_0}  \log\log T     } \right)
\eqe
for constants $c_0,c_1 >0$ depending only on $\zeta,\xi,C$.  
\end{lem}

For the proof of Lemma~\ref{lem-subadd} we will need an a priori estimate for $L_{0,n}$ which hold without assuming the inductive hypothesis~\eqref{eqn-ind-hyp}.

\begin{lem} \label{lem-left-right-conc}
Fix a small parameter $\zeta > 0$. For each $n\in\BB N$ and each $T>1$, 
\eqb \label{eqn-left-right-conc}
\BB P\left[ T^{-1} 2^{-(1-\xi Q+ \zeta) n} \leq L_{0,n} \leq T 2^{-(1-\xi Q - \zeta) n} \right] 
\geq 1 - c_0 \exp\left( - c_1  \frac{(\log T)^2}{\log\log T} \right) 
\eqe 
for constants $c_0,c_1 > 0$ depending only on $\zeta,\xi$. 
\end{lem}
\begin{proof}
Lemma~\ref{lem-a-priori-conc} implies that there is a constant $c >0$ depending only on $\xi$ such that for each $n\in\BB N$, we have $\ell_n(\frk p) \geq e^{-c\sqrt n} \lambda_n$ and $\Lambda_n(\frk p) \ell_n(\frk p) \leq e^{ c \sqrt n} \lambda_n$. Furthermore, Proposition~\ref{prop-Q-wn} implies that $ 2^{-(1-\xi Q + \zeta) n} \preceq \lambda_n \preceq 2^{-(1-\xi Q - \zeta ) n} $ with implicit constants depending only on $\zeta,\xi$. 
Plugging these estimates into Proposition~\ref{prop-perc-estimate} and absorbing the factor of $e^{c\sqrt n}$ into a small power of $2^n$ gives~\eqref{eqn-left-right-conc}.
\end{proof}

\begin{proof}[Proof of Lemma~\ref{lem-subadd}]
We first prove~\eqref{eqn-subadd-lower}. The proof is similar to the proof of Lemma~\ref{lem-mu-exponent}. 
Let $\wt P_{0,n} : [0, \wt L_{0,n} ] \rta \BB S$ be a path in $\BB S$ between the left and right boundaries of $\BB S$ with minimal $\wt D_{0,n}$-length. 
We will approximate $\wt P_{0,n}$ by a path of squares in $\mcl S_K$, then use this path of squares together with Lemma~\ref{lem-induct-full} to build a path between the left and right boundaries of $\BB S$ whose $\wt D_{0,K}$-length is bounded above. We will then use our a priori estimates for $\wt D_{0,K}$ to deduce~\eqref{eqn-subadd-lower}. 

Let $t_0 = 0$ and let $S_0 \in \mcl S_K$ be chosen so that $\wt P_{0,n}(0) \in S_0$. 
Inductively, if $j\in\BB N$ and $S_{j-1}$ and $t_{j-1}$ have been defined, let $t_j$ be the first time after $t_{j-1}$ at which $\wt P_{0,n}$ hits a square $S\in \mcl S_K$ which does not share a corner or side with $S_{j-1}$; and let $S_j$ be this square. If no such time $t_j$ exists, we instead set $t_j = \wt L_{0,n}$ and $S_j =  S_{j-1}$. Let $J$ be the smallest integer for which $t_j =  \wt L_{0,n}$. 

For each $j\in [0,J-1]_{\BB Z}$, the path $\wt P_{0,n}$ crosses between the inner and outer boundaries of the annulus $A_{S_j}' = B_{2^{-K}}(S_j) \setminus S_j$ during the time interval $[t_j , t_{j+1}]$. Consequently,
\eqbn
\wt L_{0,n}
= t_J 
\geq \sum_{j=0}^{J-1} \wt D_{0,n}\left( \text{across $A_{S_{j }}'$} \right) .
\eqen
By combining this bound with the variant of~\eqref{eqn-induct-full-across} of Lemma~\ref{lem-induct-full} for $A_{S_j}'$, we get that for each $T> 2^{K^{1/2 + 3\ep_0}}$, 
\eqb \label{eqn-subadd-across-lower}
\BB P\left[ \wt L_{0,n} \geq    T^{-1/2} 2^{-K} \lambda_{n-K} \sum_{j=0}^{J-1} e^{\xi \Phi_{0,K}(v_{S_j}) } \right] 
\geq 1 - c_0 5^K  \exp\left( - c_1 \frac{(\log T)^2 }{ K^{3\ep_0}      } \right) .
\eqe

For $S \in \mcl S_K$, let $A_S''$ be the annular region consisting of the union of the 16 squares $S' \in \mcl S_K$ which do not share a corner or a side with $S$, but which share a corner or a side with a square which shares a corner or a side with $S$. Then for each $j\in [1,J-1]_{\BB Z}$, we have $S_j \subset A_{S_{j-1}}''$.
Consequently, if $\pi_j$ is a path in $A_{S_j}'$ which disconnects its inner and outer boundaries of $A_{S_j}'$ for each $j=0 ,\dots,J-1$, then the union of the paths $\pi_j$ contains a path in $\BB S$ between the left and right boundaries of $\BB S$. 
Therefore,
\eqb \label{eqn-subadd-sum-lower}
\wt L_{0,K} 
\leq \sum_{j=0}^{J-1} \wt D_{0,K} \left( \text{around $A_{S_{j}}''$} \right)  .
\eqe

The bound~\eqref{eqn-induct-full-around} of Lemma~\ref{lem-induct-full} holds with $A_S''$ in place of $A_S$, with the same proof. 
By combining~\eqref{eqn-subadd-sum-lower} with this bound in the case when $n = K$ (note that the inductive hypothesis~\eqref{eqn-ind-hyp} holds vacuously in this case), we get that for $T>  2^{K^{1/2 + 3\ep_0}}$, 
\eqb \label{eqn-subadd-around-lower}
\BB P\left[  \wt L_{0,K} \leq   T^{1/2} 2^{-K} \sum_{j=0}^{J-1} e^{\xi \Phi_{0,K} (v_{S_j})} \right] 
\geq 1 -  c_0 5^K  \exp\left( - c_1 \frac{(\log T)^2 }{ K^{3\ep_0}  \log\log T     } \right)  .
\eqe 
 
By~\eqref{eqn-subadd-across-lower} and~\eqref{eqn-subadd-around-lower},   
\eqb \label{eqn-subadd-ratio-lower}
\BB P\left[ \wt L_{0,n} \geq   T   \lambda_{n-K} \wt L_{0,K} \right] 
\geq 1 -  c_0 5^K  \exp\left( - c_1 \frac{(\log T)^2 }{ K^{3\ep_0}  \log\log T     } \right) . 
\eqe
By Lemma~\ref{lem-left-right-conc} with $K$ in place of $n$ along with Lemma~\ref{lem-field-compare} (to transfer from $L_{0,K}$ to $\wt L_{0,K}$), 
\eqb \label{eqn-subadd-a-priori}
\BB P\left[ \wt L_{0,K} \geq T^{-1} 2^{-(1-\xi Q + \zeta) K} \right]
\geq 1  - c_0 \exp\left( - c_1  \frac{(\log T)^2}{\log\log T} \right) .
\eqe
Combining~\eqref{eqn-subadd-a-priori} with~\eqref{eqn-subadd-ratio-lower} and replacing $T$ by $T^{1/2}$ gives~\eqref{eqn-subadd-lower}.

The proof of the bound~\eqref{eqn-subadd-upper} is essentially identical, with the roles of $\wt D_{0,n}$ and $\wt D_{0,K}$ interchanged.
\end{proof}

As an immediate consequence of Lemma~\ref{lem-subadd}, we get the following bounds for the median of $L_{0,n}$ which improve on Proposition~\ref{prop-Q}.  

\begin{lem} \label{lem-exponent-relation}
Assume the inductive hypothesis~\eqref{eqn-ind-hyp}. For $\zeta >0$, 
\eqbn
2^{-(1-\xi Q + \zeta) K} \lambda_{n-K} 
\preceq \lambda_n 
\preceq 2^{-(1-\xi Q - \zeta) K} \lambda_{n-K} 
\eqen
with the implicit constants depending only on $\zeta,\xi,C$. 
\end{lem}
\begin{proof}
This follows from Lemma~\ref{lem-subadd} combined with Lemma~\ref{lem-field-compare} (to transfer from $\wt L_{0,n}$ to $L_{0,n}$). 
\end{proof}

\subsection{A variant of Proposition~\ref{prop-ef-max} without the indicator function}
\label{sec-ef-max-no-F}

In this subsection we prove a variant of Proposition~\ref{prop-ef-max} in which we do not include the factor of $\BB 1_{F_S}$ inside the conditional expectation. 
This factor will be added in Section~\ref{sec-ef-max}, which will lead to a better exponent in the upper bound. 

\begin{lem} \label{lem-ef-max-no-F}
Assume the inductive hypothesis~\eqref{eqn-ind-hyp}.
Also let $\zeta>0$ and $p > 0$. 
We have
\eqb \label{eqn-ef-max-no-F}
\BB E\left[ \left(  \max_{S\in\mcl S_K}  \BB E\left[ \frac{1}{\wt L_{0,n}^p} \wt D_{0,n}\left( \text{around $A_S $} \right)^p \,\big|\,\Psi_{K,n} \right] \right)^2 \right]^{1/2}
\preceq   2^{(p^2 \xi^2/2 - p \xi Q + \zeta )K }  
\eqe
with the implicit constant depending only on $\zeta,p,\xi,C$. 
\end{lem}

Note that for $p = 1$ (which is the relevant case in the setting of Proposition~\ref{prop-ef-max}), the exponent on the right in~\eqref{eqn-ef-max-no-F} is $  \xi^2/2 -   \xi Q$, which is negative if and only if $\xi < 2Q$. As a consequence of this, Lemma~\ref{lem-ef-max-no-F} can be used in place of Proposition~\ref{prop-ef-max} to prove Proposition~\ref{prop-quantile-ratio} in the case when $\xi < 2Q$. However, in the case when $\xi  > 2Q > 0$, Lemma~\ref{lem-ef-max-no-F} is not sufficient for our purposes and we instead need the stronger estimate of Proposition~\ref{prop-ef-max} (which is proven using Lemma~\ref{lem-ef-max-no-F}). 

To prove Lemma~\ref{lem-ef-max-no-F}, we will separately prove a lower bound for $\wt L_{0,n}$ via Lemma~\ref{lem-subadd} and an upper bound for $\max_{S\in\mcl S_K} \wt D_{0,n}\left( \text{around $A_S$}\right)$ via Lemma~\ref{lem-induct-full}, then combine these bounds via H\"older's inequality. 
We start with tail estimates, which are provided by the following lemma.

\begin{lem} \label{lem-cond-exp-moment}
Assume the inductive hypothesis~\eqref{eqn-ind-hyp}. 
Also fix $\zeta > 0$ and $p > 0$. 
For $T \geq 2^{K^{1/2 + 3\ep_0}}$,  
\allb \label{eqn-cond-exp-moment-L}
\BB P\left[    \BB E\left[  \wt L_{0,n}^{-p} \,|\,\Psi_{K,n} \right]  > T 2^{ p (1-\xi Q + \zeta)  K} \lambda_{n-K}^{-p}    \right]
\leq c_0 5^K  \exp\left( - c_1 \frac{(\log T)^2 }{ K^{3\ep_0}  \log\log T     } \right)  
\alle
and
\allb \label{eqn-cond-exp-moment-D}
&\BB P\left[  \max_{S\in\mcl S_K}  \BB E\left[  \wt D_{0,n}\left( \text{around $A_S$} \right)^p \,|\,\Psi_{K,n} \right]  > T 2^{-(p  - p^2 \xi^2 /2  -\zeta )K } \lambda_{n-K}^p \right] \notag\\
&\qquad\qquad\qquad\leq c_0 5^K  \exp\left( - c_1 \frac{(\log T)^2 }{ K^{3\ep_0}  \log\log T     } \right)   
\alle 
where the constants $c_0,c_1 > 0$ depend only on $\zeta$, $p$, $\xi$, and $C$. 
\end{lem}
\begin{proof}
\noindent\textit{Step 1: defining a good event.}
We will first build a ``good" event $G_T$ on which $ \wt D_{0,n}\left( \text{around $A_S$} \right)$ can be bounded above and $\wt L_{0,n}$ can be bounded below with high conditional probability given $\Psi_{K,n}$. 
For $j \in \BB N$, let
\allb \label{eqn-cond-exp-event}
E_j &:= \left\{ \max_{S\in\mcl S_K}  e^{-\xi \Phi_{0,K}(v_S)}  \wt D_{0,n} \left( \text{around $A_S$} \right)   \in    2^{-K} \lambda_{n-K} [2^{ j-1} , 2^{j }] \right\} \notag\\
E_j' &:= \left\{ \wt L_{0,n} \in   2^{-(1-\xi Q + \zeta) K} \lambda_{n-K} [2^{-j} , 2^{-j+1}] \right\}. 
\alle 
By Lemmas~\ref{lem-induct-full} and~\ref{lem-subadd}, each applied with $2^{j-1}$ in place of $T$, we get that for $j >  K^{1/2+3\ep_0}$,  
\eqb \label{eqn-cond-exp-prob}
\BB E\left[ \BB P\left[ E_j  | \Psi_{K,n} \right] \right]
 = \BB P\left[ E_j \right] 
 \leq c_0 5^K  \exp\left( - c_1 \frac{j^2 }{ K^{3\ep_0}   \log j     } \right)  .
\eqe
By Lemma~\ref{lem-subadd}, the bound~\eqref{eqn-cond-exp-prob} also holds with $E_j'$ in place of $E_j$. 
By Markov's inequality, 
\eqb \label{eqn-cond-exp-markov}
\BB P\left[  \BB P\left[ E_j  | \Psi_{K,n} \right] >     \exp\left( - \frac{c_1}{2} \frac{j^2 }{  K^{3\ep_0}  \log\log j     } \right)   \right] 
\leq  c_0 5^K  \exp\left( - \frac{c_1}{2} \frac{j^2 }{  K^{3\ep_0}  \log j     } \right)  ;
\eqe
and the same is true with $E_j'$ in place of $E_j$. 

For $T > 2^{K^{1/2 + 3\ep_0}}$, define
\eqb \label{eqn-cond-exp-good-def}
G_T := \left\{ \max\left\{ \BB P\left[ E_j  | \Psi_{K,n} \right] , \BB P\left[E_j' | \Psi_{K,n} \right] \right\} \leq    \exp\left( - \frac{c_1}{2} \frac{j^2 }{ K^{3\ep_0}  \log j  } \right) ,\:\forall j \geq \log_2 T \right\} .
\eqe
Since $\log_2 T \geq  K^{1/2 + 3\ep_0}   $, the probabilities on the right side of~\eqref{eqn-cond-exp-markov} for $j \geq \log_2 T$ are each bounded above by a constant depending only on $\xi,C$ and these probabilities decay superexponentially fast in $j$. By a union bound, we therefore have
\allb \label{eqn-cond-exp-good-prob}
\BB P[G_T] 
&\geq 1 - \sum_{j= \lceil \log_2 T\rceil }^\infty 5^K  \exp\left( - \frac{c_1}{2} \frac{j^2 }{  K^{3\ep_0} \log j     } \right) \notag \\ 
&\geq 1 - c_0 5^K  \exp\left( - c_1 \frac{(\log T)^2 }{  K^{3\ep_0}  \log\log T     } \right) ,
\alle
where here we have replaced $c_0$ and $c_1$ by possibly different constants in the last line. 

Henceforth assume that $G_T$ occurs. We will bound the conditional expectations of $\wt L_{0,n}^{-p}$ and $\wt D_{0,n}(\text{around $A_S$})^p$ given $\Psi_{K,n}$. 
\medskip

\noindent\textit{Step 2: bound for $ \wt L_{0,n}^{-p} $.} 
We use the definition~\eqref{eqn-cond-exp-event} of $E_j'$ and the definition~\eqref{eqn-cond-exp-good-def} of $G_T$ to get that 
\allb \label{eqn-cond-exp-split'}
 \BB E\left[ \frac{1}{\wt L_{0,n}^p} \,\big|\,   \Psi_{K,n} \right]
&=   \BB E\left[ \frac{1}{\wt L_{0,n}^p} \BB 1\left\{ \bigcap_{j=\lceil \log_2 T \rceil}^\infty (E_j')^c \right\} \,\big|\,   \Psi_{K,n} \right] + \sum_{j=\lceil \log_2 T \rceil}^\infty \BB E\left[ \frac{1}{\wt L_{0,n}^p} \BB 1_{E_j'} \,\big|\,   \Psi_{K,n} \right] \notag\\
&\preceq 2^{ p (1-\xi Q + \zeta) K} \lambda_{n-K}^{-p} \left( T^p  + \sum_{j=\lceil \log_2 T \rceil}^\infty 2^{j p} \BB P\left[ E_j' \,|\, \Psi_{K,n} \right] \right) \notag\\
&\preceq 2^{ p (1-\xi Q + \zeta) K} \lambda_{n-K}^{-p} \left( T^p  + \sum_{j=\lceil \log_2 T \rceil}^\infty 2^{j p } \exp\left( - \frac{c_1}{2} \frac{j^2 }{ K^{3\ep_0}  \log j } \right) \right) ,
\alle
with the implicit constants depending only on $\zeta,\xi,C$. 
Since $\log_2 T \geq K^{1/2 + 3\ep_0}$, the sum on the last line of~\eqref{eqn-cond-exp-split'} is of order $o(T) < O(T^p)$. 
We therefore get that on $G_T$, 
\eqb \label{eqn-cond-exp-moment-L0}
 \BB E\left[ \frac{1}{\wt L_{0,n}^p} \,\big|\,   \Psi_{K,n} \right]
  \preceq  T^p   2^{ p (1-\xi Q + \zeta) K} \lambda_{n-K}^{-p}  .
\eqe  
By replacing $T$ by a constant times $T^{1/p}$ (and reducing the value of $c_1$ to compensate) we now obtain~\eqref{eqn-cond-exp-moment-L} from~\eqref{eqn-cond-exp-moment-L0} together with~\eqref{eqn-cond-exp-good-prob}. 
\medskip

\noindent\textit{Step 3: bound for $\wt D_{0,n}\left( \text{around $A_S$} \right)^p$.}
We now prove~\eqref{eqn-cond-exp-moment-D} via a similar, but slightly more complicated, argument to the one leading to~\eqref{eqn-cond-exp-moment-L}. 
The extra complication comes from the need to estimate $\BB E\left[ e^{p\xi \Phi_{0,K}(v_S)} \BB 1_{E_j} \,|\, \Psi_{K,n}\right]$ for each $S\in\mcl S_K$ instead of just $\BB P[E_j \,|\, \Psi_{K,n}]$. 

By the definition~\eqref{eqn-cond-exp-event} of $E_j$ it holds for each $S\in\mcl S_K$ that 
\allb \label{eqn-cond-exp-split}
&\BB E\left[   \wt D_{0,n}\left( \text{around $A_S$} \right)^p \,|\,   \Psi_{K,n}  \right] \notag\\
&\qquad \leq 2^{- p K} \lambda_{n-K}^{p}  \left(   T^p \BB E\left[ e^{ p \xi \Phi_{0,K}(v_S)} \,|\, \Psi_{K,n} \right] 
 + \sum_{j=\lceil \log_2 T \rceil }^\infty 2^{p j} \BB E\left[ \BB 1_{E_j}  e^{ p \xi \Phi_{0,K}(v_S)} \,|\, \Psi_{K,n} \right] \right)  .
\alle
The random variable $\Phi_{0,K}(v_S)$ is independent from $\Psi_{K,n}$ and is centered Gaussian with variance $K \log 2$, so for each $s \in \BB R$ we have
\eqb \label{eqn-exponential-moment}
\BB E\left[ e^{ s \Phi_{0,K}(v_S)} \,|\, \Psi_{K,n} \right] = 2^{s^2 K/2} . 
\eqe
This already gives us a bound for the first term in the parentheses on the right in~\eqref{eqn-cond-exp-split}. 

To bound the sum in~\eqref{eqn-cond-exp-split}, we apply H\"older's inequality with exponents $q := (1+\zeta)/\zeta$ and $1+\zeta$, to get that for each $j\geq \log_2 T$, 
\eqbn
\BB E\left[ \BB 1_{E_j}  e^{  p \xi  \Phi_{0,K}(v_S)} \,|\, \Psi_{K,n} \right] 
\leq \BB P\left[ E_j \,|\, \Psi_{K,n}\right]^{1/q} \BB E\left[ e^{ (1+\zeta) p \xi   \Phi_{0,K}(v_S)} \,|\, \Psi_{K,n} \right]^{1/(1+\zeta)} .
\eqen
By combining this estimate with~\eqref{eqn-exponential-moment} and the definition~\eqref{eqn-cond-exp-good-def} of $G_T$, we get
\eqb \label{eqn-cond-exp-term} 
\BB E\left[ \BB 1_{E_j}  e^{ \xi (1+\zeta) \Phi_{0,K}(v_S)} \,|\, \Psi_{K,n} \right] 
\leq   \exp\left( - \frac{c_1}{2q} \frac{j^2 }{K^{3\ep_0} \log j } \right)  2^{(1+\zeta)^2 p^2 \xi^2   K / 2}   .
\eqe
We now use~\eqref{eqn-cond-exp-term} to bound each term in the sum on the right side of~\eqref{eqn-cond-exp-split} and~\eqref{eqn-exponential-moment} to bound the first term on in the parentheses on the right side of~\eqref{eqn-cond-exp-split}. We obtain that on $G_T$, 
\allb \label{eqn-cond-exp-sum}
&2^{ p K  } \lambda_{n-K}^{-p } \BB E\left[   \wt D_{0,n}\left( \text{around $A_S$} \right)^p \,|\,   \Psi_{K,n}  \right] \notag\\
&\qquad\qquad \leq T^p 2^{p^2 \xi^2  K/2} + 2^{(1+\zeta)^2 p^2 \xi^2 K/2} \sum_{j=\lceil \log_2 T \rceil}^\infty \exp\left( - \frac{c_1}{ 2q} \frac{j^2 }{K^{3\ep_0} \log j} \right) \notag\\ 
&\qquad\qquad\preceq T^p 2^{p^2 \xi^2  K/2} + 2^{(1+\zeta)^2 p^2 \xi^2 K/2} \notag \\
&\qquad\qquad\preceq T^p 2^{(1+\zeta)^2 p^2 \xi^2 K/2}  .
\alle 
Note that in the second inequality, we use the fact that $\log_2  T \geq K^{1/2 + 3\ep_0}$ to ensure that the sum is of at most constant order. 
By replacing $T$ with $T^{1/p}$ (and reducing the value of $c_1$ to compensate) we obtain from~\eqref{eqn-cond-exp-sum} and~\eqref{eqn-cond-exp-good-prob} that~\eqref{eqn-cond-exp-moment-D} holds with $2^{-(p - (1+\zeta)^2 p^2 \xi^2/2)K}$ instead of $2^{-(p -   p^2 \xi^2/2 - \zeta)K}$. We can now apply this variant of~\eqref{eqn-cond-exp-moment-D} with an appropriate choice of $\wt\zeta  = \wt\zeta(p,\xi) < \zeta$ in place of $\zeta$ to obtain~\eqref{eqn-cond-exp-moment-D}. 
\end{proof}

We now use H\"older's inequality to combine the bounds from Lemma~\ref{lem-cond-exp-moment}. 

\begin{lem} \label{lem-cond-exp-diff}
Assume the inductive hypothesis~\eqref{eqn-ind-hyp}. 
Also fix $p > 0$ and $\zeta > 0$. 
For $T \geq 2^{K^{1/2 + 3\ep_0}}$,  
\allb \label{eqn-cond-exp-diff} 
&\BB P\left[ \max_{S\in\mcl S_K} \BB E\left[ \frac{1}{\wt L_{0,n}^p}  \wt D_{0,n}\left( \text{around $A_S$} \right)^p \,\big|\,   \Psi_{K,n}  \right]> T 2^{(p^2 \xi^2/2 - p \xi Q  + \zeta ) K}  \right] \notag \\ 
&\qquad \leq  c_0 5^K  \exp\left( - c_1 \frac{(\log T)^2}{K^{3\ep_0} \log \log T} \right)       
\alle
where the constants $c_0,c_1 > 0$ depend only on $\zeta, p $, $\xi$, and $C$. 
\end{lem}
\begin{proof}   
Let $q  = (1+\zeta)/\zeta$, so that $1/q + 1/(1+\zeta ) = 1$. By H\"older's inequality with exponents $q$ and $1+\zeta$, for each $S\in\mcl S_K$, 
\eqb \label{eqn-cond-exp-holder}
\BB E\left[ \frac{1}{\wt L_{0,n}^p} \wt D_{0,n}\left( \text{around $A_S$} \right)^p \,|\,   \Psi_{K,n}  \right]
\leq \BB E\left[ \frac{1}{\wt L_{0,n}^{p q} } \,|\,   \Psi_{K,n} \right]^{1/q} 
\BB E\left[ \wt D_{0,n}\left( \text{around $A_S$} \right)^{   (1+\zeta) q } \,|\,   \Psi_{K,n} \right]^{1/(1+\zeta )} .
\eqe
By Lemma~\ref{lem-cond-exp-moment}, it holds with probability at least $1-c_0 5^K  \exp\left( - c_1 \frac{(\log T)^2}{K^{3\ep_0} \log \log T} \right)$ that
\eqbn
 \BB E\left[  \wt L_{0,n}^{-p q} \,|\,\Psi_{K,n} \right]^{1/q}  \leq T^{1/q} 2^{ p (1-\xi Q + \zeta)  K} \lambda_{n-K}^{-p}  
\eqen
and
\eqbn
 \max_{S\in\mcl S_K}  \BB E\left[  \wt D_{0,n}\left( \text{around $A_S$} \right)^{(1+\zeta) p} \,|\,\Psi_{K,n} \right]^{1/(1+\zeta)} \leq T^{1/(1+\zeta)} 2^{-(p  -    \xi^2 p^2 /2  -o_\zeta(1) )K } \lambda_{n-K}^p  ,
\eqen 
where $o_\zeta(1)$ denotes a deterministic quantity which tends to zero as $\zeta\rta 0$ and depends only on $p , \xi$. 
Plugging these last two estimates into~\eqref{eqn-cond-exp-holder} and canceling the factors of $\lambda_{n-K}^{-p}$ and $\lambda_{n-K}^p$ shows that with probability at least $1-c_0 5^K  \exp\left( - c_1 \frac{(\log T)^2}{K^{3\ep_0} \log \log T} \right)$, 
\eqbn
\max_{S\in\mcl S_K} \BB E\left[ \frac{1}{\wt L_{0,n}^p} \wt D_{0,n}\left( \text{around $A_S$} \right)^p  \,|\,   \Psi_{K,n}  \right]
\leq T 2^{(p^2 \xi^2/2 - p \xi Q + o_\zeta(1)) K} .
\eqen
This gives~\eqref{eqn-cond-exp-diff} with $o_\zeta(1)$ instead of $\zeta$. We then apply this variant of~\eqref{eqn-cond-exp-diff} with an appropriate choice of $\wt\zeta  = \wt\zeta(p,\xi) < \zeta$ in place of $\zeta$ to obtain~\eqref{eqn-cond-exp-diff}. 
\end{proof}

To conclude the proof, it remains to transfer from a probability estimate to a moment estimate. 

\begin{proof}[Proof of Lemma~\ref{lem-ef-max-no-F}]
To lighten notation let
\eqbn
Y_S:=  \BB E\left[ \frac{1}{\wt L_{0,n}^p}  \wt D_{0,n}\left( \text{around $A_S$} \right)^p \,|\,   \Psi_{K,n}  \right] .
\eqen
By Lemma~\ref{lem-cond-exp-diff}, for each $T > \exp(K^{2/3})$, 
\eqb \label{eqn-use-cond-exp-diff}
\BB P\left[ \max_{S\in\mcl S_K} Y_S > T 2^{(p^2 \xi^2/2 - p \xi Q + \zeta) K}   \right] \leq  c_0 5^K  \exp\left( - c_1 \frac{(\log T)^2}{K^{3\ep_0} \log \log T} \right) .
\eqe
We use the bound~\eqref{eqn-use-cond-exp-diff} for $T \geq \exp(K^{2/3})$ and the trivial bound $\BB P[\cdots] \leq 1$ for $T\leq \exp(K^{2/3})$ to get 
\allb \label{eqn-ef-term1}
\BB E\left[ \left(   2^{-(p^2 \xi^2/2 - p \xi Q + \zeta) K}     \max_{S\in\mcl S_K} Y_S \right)^2 \right] 
&= 2 \int_0^\infty T \BB P\left[    2^{-(p \xi^2/2 - p\xi Q + \zeta) K}  \max_{S\in\mcl S_K} Y_S   > T   \right]  \, dT \notag\\
&\preceq \exp\left( 2 K^{2/3} \right) + 5^K \int_{\exp(K^{2/3})}^\infty  \exp\left( - c_1 \frac{(\log T)^2}{K^{3\ep_0} \log \log T} \right)  \,dT \notag\\
&\preceq \exp\left( 2 K^{2/3} \right) + 5^K \exp\left( - c_1 K^{4/3 - 3\ep_0} / \log\log K\right)  \notag \\ 
&\preceq \exp\left( 2 K^{2/3} \right)   .
\alle 
Hence
\eqbn
\BB E\left[ \left( \max_{S\in\mcl S_K} Y_S \right)^2 \right]^{1/2}  
\preceq \exp\left( K^{2/3} \right)  2^{ (p^2 \xi^2/2 -  p \xi Q + \zeta) K} .
\eqen
We now slightly increase $\zeta$ to absorb the factor of $ \exp(K^{2/3})$. This gives~\eqref{eqn-ef-max-no-F}.
\end{proof}

\subsection{Proof of Proposition~\ref{prop-ef-max}}
\label{sec-ef-max}

The idea of the proof is as follows. Just below, we will define a high-probability regularity event $E$. We will then use the bound
\allb \label{eqn-hit-split}
&\BB E\left[ \left( \max_{S\in\mcl S_K} \BB E\left[ \BB 1_{F_S}  \frac{1}{\wt L_{0,n}} \wt D_{0,n}\left( \text{across $A_S$} \right) \,|\, \Psi_{n,K} \right] \right)^2 \right] \notag\\
&\qquad \leq 2 \BB E\left[ \left( \max_{S\in\mcl S_K} \BB E\left[ \BB 1_{F_S \cap E}  \frac{1}{\wt L_{0,n}} \wt D_{0,n}\left( \text{across $A_S$} \right) \,|\, \Psi_{n,K} \right] \right)^2 \right] \notag\\
&\qquad\qquad + 2 \BB E\left[ \left( \max_{S\in\mcl S_K} \BB E\left[ \BB 1_{E^c}  \frac{1}{\wt L_{0,n}} \wt D_{0,n}\left( \text{across $A_S$} \right) \,|\, \Psi_{n,K} \right] \right)^2 \right] .
\alle
We will show (using a Gaussian estimate) that the first term on the right in~\eqref{eqn-hit-split} is bounded above by a constant times $2^{-(\alpha(\xi) - \zeta) K}$. 
As for the second term, we will use the Cauchy-Schwarz inequality together with the fact that $\BB P[E^c]$ is small and Lemma~\ref{lem-ef-max-no-F} to show that this term is smaller than any negative power of $2^K$. 
\medskip

\noindent\textit{Step 1: high-probability regularity event.}
Let $E$ be the event that the following is true.
\begin{enumerate}
\item $\wt L_{0,n}  \leq   2^{-(1-\xi Q - \zeta) K + K^{2/3} } \lambda_{n-K}$. \label{item-hit-L-upper}
\item $\wt L_{0,n} \geq 2^{-(1-\xi Q + \zeta) K - K^{2/3}} \lambda_{n-K}$. \label{item-hit-L-lower}
\item $\min_{S\in\mcl S_K}  e^{-\xi \Phi_{0,K}(v_S)}  \wt D_{0,n} \left( \text{across $A_S$} \right)  \geq   2^{-K - K^{2/3} } \lambda_{n-K}$.  \label{item-hit-across}
\item $\max_{S\in\mcl S_K}  e^{-\xi \Phi_{0,K}(v_S)}    \wt D_{0,n}\left( \text{around $A_S$} \right)     \leq   2^{-K + K^{2/3} } \lambda_{n-K}$. \label{item-hit-around}
\end{enumerate}
Lemmas~\ref{lem-subadd} and \ref{lem-induct-full} (applied with $T = 2^{K^{2/3}}$) give us bounds for the probabilities of each the four conditions in the definition of $E$. Combined, these bounds yield
\eqb \label{eqn-hit-event-prob}
\BB P\left[ E^c \right] 
=\BB E\left[ \BB P\left[ E_T^c \,|\, \Psi_{K,n} \right] \right] 
\leq  c_0 5^K  \exp\left( - c_1 \frac{K^{4/3} }{K^{3\ep_0} \log K} \right)
\leq c_0 \exp\left( - c_1   K^{4/3 - 3\ep_0} / \log K  \right) 
\eqe
where in the last inequality we decreased $c_1$ to absorb the factor of $5^K$. 
\medskip

\noindent\textit{Step 2: bound for the conditional expectation on $E$.}
The key observation for this step is that if $S\in\mcl S_K$ such that $F_S$ occurs, then the $\wt D_{0,n}$-geodesic $\wt P_{0,n}$ must cross between the inner and outer boundaries of the annulus $A_S$ at least once. Consequently, 
\eqb \label{eqn-cond-exp-total}
F_S \quad \Rightarrow \quad \left\{\wt D_{0,n}\left( \text{across $A_S$} \right) \leq \wt L_{0,n} \right\} .
\eqe 
We will now investigate what this bound gives us on the event $E$. 
By~\eqref{eqn-cond-exp-total} together with conditions~\ref{item-hit-L-upper} and~\ref{item-hit-across} in the definition of $E$, for each $S\in\mcl S_K$, 
\eqbn \label{eqn-hit-indicator}
F_S \cap E \subset H_S:= \left\{ e^{ \xi \Phi_{0,K}(v_S)} \leq 2^{(\xi Q + \zeta) K + 2 K^{2/3}}  \right\} .
\eqen
By conditions~\ref{item-hit-L-lower} and~\ref{item-hit-around} in the definition of $E$, we also have
\eqbn \label{eqn-hit-variable}
\BB 1_E \frac{1}{\wt L_{0,n}} \wt D_{0,n}\left( \text{across $A_S$} \right)
\leq 2^{  - (\xi Q - \zeta) K  + 2 K^{2/3}} e^{\xi \Phi_{0,K}(v_S)} .
\eqen
Combining the above two inequalities gives
\eqb \label{eqn-hit-expectation}
\BB E\left[ \BB 1_{F_S \cap E} \frac{1}{\wt L_{0,n}} \wt D_{0,n}\left( \text{across $A_S$} \right) \,\big|\, \Psi_{n,K} \right] 
\leq 2^{  - (\xi Q - \zeta) K  + 2 K^{2/3}}  \BB E\left[  e^{\xi \Phi_{0,K}(v_S)} \BB 1_{H_S} \,\big|\, \Psi_{n,K} \right] .
\eqe

The random variable $\Phi_{0,K}(v_S)$ is centered Gaussian with variance $ K \log 2$ and is independent from $\Psi_{n,K}$. 
By the definition of $H_S$ in~\eqref{eqn-hit-indicator} and a basic estimate for exponential moments of a Gaussian random variable (Lemma~\ref{lem-gaussian-trunc} applied with $R = K\log 2$ and with $\beta$ slightly larger than $Q$), we get
\eqbn
\BB E\left[  e^{\xi \Phi_{0,K}(v_S)} \BB 1_{H_S} \,\big|\, \Psi_{n,K} \right]
=\BB E\left[  e^{\xi \Phi_{0,K}(v_S)} \BB 1_{H_S}  \right]
= 2^{\left[\xi (\xi \wedge Q) - (\xi \wedge Q)^2/2 + o_\zeta(1) \right] K + O_K(K^{2/3}) }  
\eqen
where the $o_\zeta(1)$ tends to zero as $\zeta\rta 0$ at a deterministic rate which depends only on $\xi$; and the implicit constant in the $O_K(K^{2/3})$ is deterministic and depends only on $\xi,\zeta$. 
By combining this last estimate with~\eqref{eqn-hit-expectation}, we arrive at
\eqbn \label{eqn-hit-good0}
\BB E\left[ \BB 1_{F_S \cap E} \frac{1}{\wt L_{0,n}} \wt D_{0,n}\left( \text{across $A_S$} \right) \,\big|\, \Psi_{n,K} \right] 
\leq 2^{ - (\alpha(\xi)  + o_\zeta(1)) K  + 2 K^{2/3}} ,\quad\forall S\in\mcl S_K 
\eqen
where $\alpha(\xi)$ is as in~\eqref{eqn-hit-exponent}. Consequently,
\eqb \label{eqn-hit-good}
\BB E\left[ \left( \max_{S\in\mcl S_K} \BB E\left[ \BB 1_{F_S \cap E} \frac{1}{\wt L_{0,n}} \wt D_{0,n}\left( \text{across $A_S$} \right) \,\big|\, \Psi_{n,K} \right] \right)^2 \right] 
\leq 2^{ -2 (\alpha(\xi)  + o_\zeta(1)) K  + 4 K^{2/3}} .
\eqe
\medskip

\noindent\textit{Step 3: bound for the conditional expectation on $E^c$.}
By the Cauchy-Schwarz inequality,
\eqbn
\BB E\left[ \BB 1_{E^c} \frac{1}{\wt L_{0,n}} \wt D_{0,n}\left( \text{across $A_S$} \right) \,\big|\, \Psi_{n,K} \right]
\leq \BB P\left[ E^c \,\big|\, \Psi_{K,n}\right]^{1/2} \BB E\left[  \frac{1}{\wt L_{0,n}^2} \wt D_{0,n}\left( \text{across $A_S$} \right)^2 \,\big|\, \Psi_{n,K} \right]^{1/2} .
\eqen
Therefore,
\allb \label{eqn-hit-bad-split}
&\BB E\left[ \left(\max_{S\in\mcl S_K} \BB E\left[ \BB 1_{E^c} \frac{1}{\wt L_{0,n}} \wt D_{0,n}\left( \text{across $A_S$} \right) \,\big|\, \Psi_{n,K} \right] \right)^2 \right] \notag\\
&\qquad \leq \BB E\left[   \BB P\left[ E^c \,\big|\, \Psi_{K,n}\right]    \max_{S\in \mcl S_K} \BB E\left[  \frac{1}{\wt L_{0,n}^2} \wt D_{0,n}\left( \text{across $A_S$} \right)^2 \,\big|\, \Psi_{n,K} \right] \right]   \notag \\
&\qquad \leq \BB E\left[   \BB P\left[ E^c \,\big|\, \Psi_{K,n}\right]^2 \right]^{1/2} \BB E\left[ \left(     \max_{S\in \mcl S_K} \BB E\left[  \frac{1}{\wt L_{0,n}^2} \wt D_{0,n}\left( \text{across $A_S$} \right)^2 \,\big|\, \Psi_{n,K} \right] \right)^2  \right]^{1/2} \notag\\
&\qquad\qquad\qquad\qquad\qquad\qquad\qquad \text{(by Cauchy-Schwarz)}  .
\alle
Since $\BB P[E^c \,\big|\, \Psi_{K,n} ] \leq 1$ and by~\eqref{eqn-hit-event-prob}, 
\allb \label{eqn-hit-bad-term1}
\BB E\left[   \BB P\left[ E^c \,\big|\, \Psi_{K,n}\right]^2 \right]^{1/2}
 \leq  \BB E\left[   \BB P\left[ E^c \,\big|\, \Psi_{K,n}\right]  \right]^{1/2}
 = \BB P\left[ E^c \right]^{1/2}  
 \leq c_0 \exp\left( - c_1 K^{4/3 - 3\ep_0} \right)    .
\alle 
By Lemma~\ref{lem-ef-max-no-F}, 
\eqb \label{eqn-hit-bad-term2} 
 \BB E\left[ \left(     \max_{S\in \mcl S_K} \BB E\left[  \frac{1}{\wt L_{0,n}^2} \wt D_{0,n}\left( \text{across $A_S$} \right)^2 \,\big|\, \Psi_{n,K} \right] \right)^2  \right]^{1/2}
\preceq  2^{ (2 \xi Q   - 2 \xi^2   + \zeta )K } .
\eqe
By plugging~\eqref{eqn-hit-bad-term1} and~\eqref{eqn-hit-bad-term2} into~\eqref{eqn-hit-bad-split}, then absorbing a power of $2^K$ into a power of $e^{K^{4/3 - 3\ep_0}}$, we arrive at
\eqb \label{eqn-hit-bad}
\BB E\left[ \left(\max_{S\in\mcl S_K} \BB E\left[ \BB 1_{E^c} \frac{1}{\wt L_{0,n}} \wt D_{0,n}\left( \text{across $A_S$} \right) \,\big|\, \Psi_{n,K} \right] \right)^2 \right]
\leq c_0 \exp\left( - c_1 K^{4/3 - 3\ep_0} \right) .
\eqe
\medskip

\noindent\textit{Step 4: Conclusion.}
By plugging~\eqref{eqn-hit-bad} and~\eqref{eqn-hit-good} into~\eqref{eqn-hit-split} and adjusting $\zeta$, we obtain~\eqref{eqn-ef-max}. \qed

\section{Tightness of point-to-point distances}
\label{sec-pt-tight}

This section has two main purposes.
\begin{itemize}
\item We establish tightness for the $\lambda_n^{-1} D_{0,n}$-distances between points (not just between non-trivial connected sets, which is the setting covered by Proposition~\ref{prop-perc-estimate'}).
\item We transfer our tightness results for $\lambda_n^{-1} D_{0,n}$ to tightness results for the variant of LFPP used in Theorem~\ref{thm-lfpp-tight}, namely $\frk a_\ep^{-1} D_h^\ep$, which is defined using convolutions of the GFF with the heat kernel rather than the white noise decomposition (see~\eqref{eqn-gff-lfpp}). 
\end{itemize}
The following proposition is the main result of this section, and will be proven in Section~\ref{sec-gff-compare}.

\begin{prop} \label{prop-gff-pt-tight} 
Let $\xi > 0$. 
Let $U\subset\BB C$ be a connected open set and let $K_1,K_2\subset U$ be disjoint compact connected sets (allowed to be singletons). 
The random variables $\frk a_\ep^{-1} D_h^\ep(K_1,K_2 ; U)$ and $\left( \frk a_\ep^{-1} D_h^\ep(K_1,K_2;U) \right)^{-1}$ for $\ep \in (0,1)$ are tight. 
Moreover, there are constants $c_0,c_1 > 0$ depending only on $\xi$ such that if $n = \log_2\ep^{-1}$, then\footnote{Note that $n$ is not required to be an integer, but the definition of $\lambda_n$ in Section~\ref{sec-wn-decomp} still makes sense for non-integer values of $n$.}
\eqb \label{eqn-scale-compare}
c_0 \leq \frk a_\ep^{-1} \lambda_n \leq c_1 .
\eqe  
\end{prop}

Our proof of Proposition~\ref{prop-gff-pt-tight} combined with Lemma~\ref{lem-exponent-relation} will also yield bounds for the scaling constants $\frk a_\ep$ appearing in Theorem~\ref{thm-lfpp-tight}, which will eventually be used to obtain Assertion~\ref{item-constant} of Theorem~\ref{thm-lfpp-tight}.

\begin{prop} \label{prop-gff-scale-constant}
Let $\xi > 0$. 
For each $\zeta > 0$, we have
\eqb \label{eqn-gff-scale-constant1}
  r^{\xi Q + \zeta} \preceq  \frac{r \frk a_{\ep/r}}{\frk a_\ep}  \preceq r^{\xi Q - \zeta} ,\quad\forall r  \in (0,1) ,\quad\forall \ep \in (0,r)
\eqe
and
\eqb \label{eqn-gff-scale-constant2}
 r^{\xi Q - \zeta}  \preceq   \frac{r \frk a_{\ep/r}}{\frk a_\ep}   \preceq  r^{\xi Q + \zeta} ,\quad\forall r  >1 , \quad \forall \ep \in (0,1)   
\eqe
with the implicit constants depending only on $\zeta,\xi$. 
\end{prop}

The rest of this section is structured as follows. In Section~\ref{sec-wn-pt-tight}, we prove the tightness of $D_{0,n}$-distances between points by applying the estimate of Proposition~\ref{prop-perc-estimate'} at dyadic scales, then summing over scales. In Section~\ref{sec-non-dyadic}, we explain why our estimates for $D_{0,n}$ continue to hold when $n$ is not required to be an integer (this is important since we do not restrict to dyadic values of $\ep$ in Proposition~\ref{prop-gff-pt-tight}). 
In Section~\ref{sec-gff-compare}, we prove a comparison lemma for $D_{0,n}$ and $D_h^\ep$ (Lemma~\ref{lem-gff-compare}) and use it to extract Propositions~\ref{prop-Q}, \ref{prop-gff-pt-tight}, and~\ref{prop-gff-scale-constant} from the analogous results for $D_{0,n}$. In Section~\ref{sec-gff-estimates}, we record some basic estimates for $D_h^\ep$ which are consequences of our previously known estimates for $D_{0,n}$. 

\subsection{Pointwise tightness for white-noise LFPP}
\label{sec-wn-pt-tight}

Proposition~\ref{prop-perc-estimate'} shows that $\lambda_n^{-1} D_{0,n}$-distances between infinite connected sets are tight.
In this subsection, we will show that also $\lambda_n^{-1} D_{0,n}$-distances between points are tight.

\begin{prop} \label{prop-pt-tight} 
Let $\xi > 0$. 
Let $z,w\in\BB C$ be distinct and let $U\subset\BB C$ be a connected open set containing $z$ and $w$. The random variables $\lambda_n^{-1} D_{0,n}(z,w ; U)$ and $\left( \lambda_n^{-1} D_{0,n}(z,w ; U) \right)^{-1}$ for $n\in\BB N$ are tight. 
\end{prop}

The tightness of $\left( \lambda_n^{-1} D_{0,n}(z,w ; U) \right)^{-1}$ follows directly from Proposition~\ref{prop-perc-estimate'}, so the main difficulty in the proof of Proposition~\ref{prop-pt-tight} is showing the tightness of $\lambda_n^{-1} D_{0,n}(z,w ; U)$. 
The idea of the proof is to first apply Proposition~\ref{prop-perc-estimate'} at dyadic scales to build paths around and across dyadic annuli surrounding each point whose $\lambda_n^{-1} D_{0,n}$-lengths are bounded above. We will then string together such paths to build paths between points whose $D_{0,n}$-lengths are bounded above. 
We first need a variant of Proposition~\ref{prop-perc-estimate'} which can be applied at multiple scales.

\begin{lem} \label{lem-perc-field}
Fix $\zeta >0$. 
Let $U\subset \BB C$ be open and let $K_1,K_2\subset U$ be disjoint compact connected sets which are not singletons.
There are constants $c_0,c_1> 0$ depending only on $\zeta, U,K_1,K_2,\xi$ such that the following is true.  
For each $n , k \in\BB N_0$ with $k \leq n$ and each $T > 3$,  
\eqb \label{eqn-perc-field-lower} 
\BB P\left[ \lambda_n^{-1} D_{0,n}\left( 2^{-k} K_1 , 2^{-k}  K_2 ;  2^{-k} U \right) < T^{-1} 2^{-(\xi Q + \zeta) k} e^{\xi \Phi_{0,k}(0)} \right] 
\leq c_0 e^{-c_1 (\log T)^2} 
\eqe
and
\eqb \label{eqn-perc-field-upper} 
\BB P\left[ \lambda_n^{-1} D_{0,n}\left(2^{-k}  K_1 , 2^{-k} K_2 ; 2^{-k} U \right) > T 2^{-(\xi Q - \zeta) k} e^{\xi \Phi_{0,k}(0)} \right] 
\leq c_0 e^{-c_1 (\log T)^2 / \log\log T}  .
\eqe
\end{lem}
\begin{proof}
The idea of the proof is to re-scale space then apply Proposition~\ref{prop-perc-estimate'}. 
By the scale invariance property~\eqref{eqn-phi-scale} of $\Phi_{0,n}$, 
\eqbn
 \Phi_{k,n}\left( \cdot \right) \eqD \Phi_{0,n-k}\left(2^k \cdot\right) .
\eqen
Therefore,
\eqbn
D_{k,n}\left( 2^{-k} K_1 , 2^{-k} K_2 ; 2^{-k} U\right) 
\eqD 2^{-k} D_{0,n-k}\left( K_1 , K_2 ; U\right) .
\eqen
By Proposition~\ref{prop-perc-estimate'} applied with $n-k$ in place of $n$, we thus obtain
\eqb \label{eqn-perc-field-lower0}
\BB P\left[ D_{k,n}\left( 2^{-k} K_1 , 2^{-k} K_2 ; 2^{-k} U\right)  < T^{-1/2} 2^{-k} \lambda_{n-k} \right] 
\leq c_0 e^{-c_1 (\log T)^2}
\eqe
and
\eqb \label{eqn-perc-field-upper0}
\BB P\left[ D_{k,n}\left( 2^{-k} K_1 , 2^{-k} K_2 ; 2^{-k} U\right)   > T^{1/2}  2^{-k} \lambda_{n-k} \right] 
\leq c_0 e^{-c_1 (\log T)^2  /(\log\log T) } .
\eqe

We now temporarily impose the assumption that $U$ is bounded (we will remove this assumption at the end of the proof). 
We want to transfer from estimates for $D_{k,n}$ to estimates for $D_{0,n}$ using the fact that $\Phi_{0,n} = \Phi_{k,n} + \Phi_{0,k}$. To this end, we first use Lemma~\ref{lem-phi-cont} (with $k$ in place of $n$) and the mean value theorem to get
\eqb \label{eqn-perc-field-cont}
\BB P\left[ \sup_{z \in 2^{-k} U} |\Phi_{0,k}(z) - \Phi_{0,k}(0)| > \frac{1}{2\xi} \log T\right] \leq c_0 e^{-c_1(\log T)^2} .
\eqe
Since $\Phi_{0,n} = \Phi_{k,n} + \Phi_{0,k}$, 
\allb \label{eqn-perc-field-compare}
&\exp\left(\xi \min_{z\in 2^{-k} U} \Phi_{0,k}(z) \right)  D_{k,n}\left( 2^{-k} K_1 , 2^{-k} K_2 ; 2^{-k} U\right) \notag\\
&\qquad \leq D_{0,n} \left( 2^{-k} K_1 , 2^{-k} K_2 ; 2^{-k} U\right)\notag\\
&\qquad\qquad \leq \exp\left(\xi \max_{z\in 2^{-k} U} \Phi_{0,k}(z) \right)  D_{k,n}\left( 2^{-k} K_1 , 2^{-k} K_2 ; 2^{-k} U\right) .
\alle
By~\eqref{eqn-perc-field-cont} and~\eqref{eqn-perc-field-compare}, it holds with probability at least $1-c_0 e^{-c_1(\log T)^2}$ that
\allb \label{eqn-perc-field-compare'}
&T^{-1/2} e^{\xi \Phi_{0,k}(v_S)} D_{k,n}\left( 2^{-k} K_1 , 2^{-k} K_2 ; 2^{-k} U\right) \notag\\
&\qquad\leq  D_{0,n} \left( 2^{-k} K_1 , 2^{-k} K_2 ; 2^{-k} U\right)\notag\\
&\qquad\qquad \leq  T^{ 1/2} e^{\xi \Phi_{0,k}(v_S)} D_{k,n}\left( 2^{-k} K_1 , 2^{-k} K_2 ; 2^{-k} U\right) .
\alle

Combining~\eqref{eqn-perc-field-compare'} with~\eqref{eqn-perc-field-lower0} and~\eqref{eqn-perc-field-upper0} shows that
\allb \label{eqn-perc-field-before}
\BB P\left[  D_{0,n}\left( 2^{-k} K_1 , 2^{-k}  K_2 ;  2^{-k} U \right)  \geq  T^{-1} 2^{-k} \lambda_{n-k}      \right]
&\geq 1 - c_0 e^{-c_1 (\log T)^2 } \quad \text{and} \notag\\
\BB P\left[  D_{0,n}\left( 2^{-k} K_1 , 2^{-k}  K_2 ;  2^{-k} U \right)  \leq  T 2^{-k} \lambda_{n-k}      \right]
&\geq 1 - c_0 e^{-c_1 (\log T)^2/\log\log T}  
\alle
By Lemma~\ref{lem-exponent-relation}, $2^{(1-\xi Q - \zeta) k} \lambda_n \preceq   \lambda_{n-k} \preceq 2^{(1-\xi Q + \zeta) k} \lambda_n$ with the implicit constants depending only on $\xi,\zeta$. Combining this with~\eqref{eqn-perc-field-before} gives~\eqref{eqn-perc-field-lower} and~\eqref{eqn-perc-field-upper} in the case when $U$ is bounded.

To deduce~\eqref{eqn-perc-field-lower} in the case when $U$ is unbounded, choose a compact set $K'$ which disconnects $K_1$ from $K_2$ and a bounded open set $U'\subset \BB C$ which contains $K_1,K_2,K'$. Then apply~\eqref{eqn-perc-field-lower} with $U'$ in place of $U$ and $K'$ in place of $K_2$.
The estimate~\eqref{eqn-perc-field-upper} in the case when $U$ is unbounded is immediate from the bounded case since increasing $U$ causes $ D_{0,n}\left(2^{-k}  K_1 , 2^{-k} K_2 ; 2^{-k} U \right) $ to decrease.
\end{proof}

The following lemma is the main input in the proof of Proposition~\ref{prop-pt-tight}.

\begin{figure}[t!]
 \begin{center}
\includegraphics[scale=1]{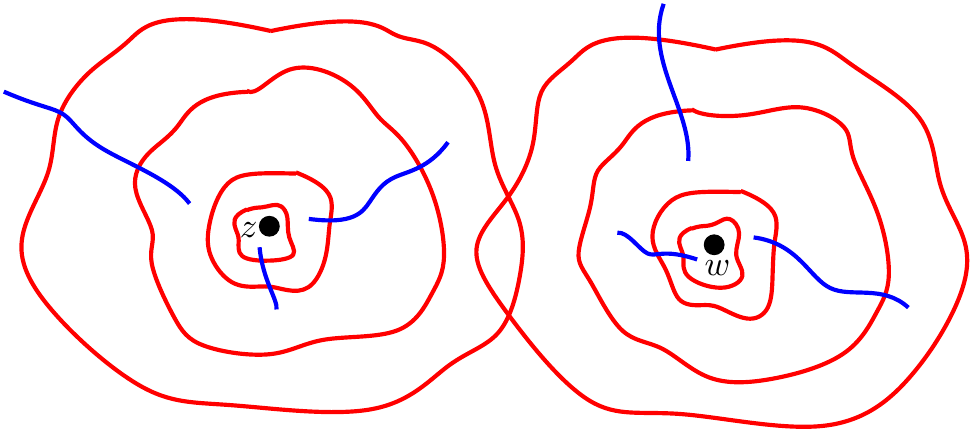}
\vspace{-0.01\textheight}
\caption{Illustration of the proof of Lemma~\ref{lem-pt-dist}. The red (resp.\ blue) curves are $D_{0,n}$-geodesics around (resp.\ across) Euclidean annuli surrounding $z$ (resp.\ $w$). We upper-bound the $D_{0,n}$-lengths of these geodesics using Proposition~\ref{prop-perc-estimate'}, then use that the union of these geodesics is connected to upper-bound $D_{0,n}(z,w)$. 
}\label{fig-pt-dist}
\end{center}
\vspace{-1em}
\end{figure}

\begin{lem} \label{lem-pt-dist}
Fix $\zeta>0$, a bounded open set $V\subset\BB C$, and an open set $U\subset \BB C$ with $\ol V\subset U$. There are constants $c_0,c_1> 0$ depending only on $\zeta, V,U , \xi$ such that for each $n\in\BB N$ and each $T>3$, it holds with probability at least $1- c_0 e^{-c_1 (\log T)^2  / \log\log T}$ that the following is true. 
For each $z,w \in V$ with $|z-w| \leq 2 \op{dist}(\ol V , \bdy U)$ (here $\op{dist}$ denotes Euclidean distance), 
\eqb \label{eqn-pt-dist}
%T^{-1}   \sum_{j=\lceil \log_2 |z-w| \rceil  + 1}^n 2^{-(\xi Q + \zeta) j} \left( e^{\xi \Phi_{0,j}(z)}  + e^{\xi \Phi_{0,j}(w)} \right)\leq 
\lambda_n^{-1} D_{0,n}(z,w ; U) 
\leq  T   \sum_{j=\lceil \log_2 |z-w| \rceil}^n 2^{-(\xi Q - \zeta) j} \left( e^{\xi \Phi_{0,j}(z)}  + e^{\xi \Phi_{0,j}(w)} \right).
\eqe 
\end{lem}
\begin{proof}%[Proof of Lemma~\ref{lem-pt-dist}] 
See Figure~\ref{fig-pt-dist} for an illustration of the proof. 
Throughout the proof, $c_0,c_1$ denote deterministic constants which depend only on $\zeta,V,U,\xi$ and which are allowed to change from line to line. We also require all implicit constants in $\preceq$ to be deterministic and depend only on $\zeta,V,U,\xi$. 
\medskip

\noindent\textit{Step 1: defining a high-probability regularity event.}
For $k \in [0,n]_{\BB Z}$, let $\mcl X_k$ be a collection of $O_k(4^k)$ points $x\in \BB C$ so that the balls $B_{2^{-k-2}}(x)$ for $x\in\mcl X_k$ cover $V$. 
For $x\in\mcl X_k$ and $T>3$, let $E_k(x,T)$ be the event that
\eqb \label{eqn-1pt-dist-across}
\lambda_n^{-1} D_{0,n}\left( \text{across $B_{2^{-k+4}}(x) \setminus B_{2^{-k-2}}(x)$} \right) \leq T 2^{-(\xi Q + \zeta) k} e^{\xi \Phi_{0,k}(x)} 
\eqe
and
\eqb \label{eqn-1pt-dist-around}
\lambda_n^{-1} D_{0,n}\left( \text{around $B_{2^{-k+1}}(x) \setminus B_{2^{-k}}(x)$} \right) \leq T 2^{-(\xi Q - \zeta) k} e^{\xi \Phi_{0,k}(x)}  .
\eqe
By Lemma~\ref{lem-perc-field} (applied with $\zeta/2$ in place of $\zeta$ and $2^{(\zeta/2) k} T$ in place of $T$) and the translation invariance of the law of $\Phi_{0,n}$,
\eqb
\BB P\left[ E_k(x,T)^c \right] \leq c_0 \exp\left( - c_1 \frac{(k + \log T)^2 }{\log(k + \log T)} \right) ,\quad\forall k\in\BB N,\quad\forall x \in \mcl X_k .
\eqe
Therefore,
\eqb \label{eqn-1pt-dist-event}
\BB P\left[ \left( \bigcap_{x\in \mcl X_k} E_k(x,T) \right)^c \right] 
\leq c_0  4^k \exp\left( - c_1 \frac{(k + \log T)^2 }{\log(k + \log T)} \right) .
\eqe

The events $E_k(x,T)$ are not quite sufficient for our purposes since they only allow us to bound distances in terms of $\Phi_{0,k}(x)$ for $x\in\mcl X_k$, but we want to bound distances in terms of $\Phi_{0,k}(z)$ for an \emph{arbitrary} $z\in U$. For this purpose we also need a continuity condition for $\Phi_{0,k}$. 
By Lemma~\ref{lem-phi-cont} (applied with $x = \frac{1}{\xi}( \log T + \zeta k)$), for each $k\in \BB N$ it holds with probability at least $1-c_0  \exp\left( - c_1 (\log T + \log k)^2 \right)$ that
\eqb \label{eqn-1pt-cont}
2^{-k} \sup_{z\in B_2(U)} |\nabla \Phi_{0,k}(z)| \leq \frac{1}{\xi}( \log T + \zeta k) .
\eqe

For $k\in\BB N$, let $E_k(T)$ be the union of $\bigcup_{x\in \mcl X_k} E_k(x,T)$ and the event in~\eqref{eqn-1pt-cont}. Then $\BB P\left[ E_k(T)^c \right] \leq c_0  4^k \exp\left( - c_1 \frac{(k + \log T)^2 }{\log(k + \log T)} \right)$. 
Therefore, if we set 
\eqb
E(T) := \bigcap_{k=0}^n E_k(T) 
\eqe
then by a union bound over $k \in [0,n]_{\BB Z}$, 
\allb \label{eqn-pt-dist-prob}
\BB P\left[ E(T)^c \right] 
&\leq c_0 \sum_{k=0}^n  4^k \exp\left( - c_1 \frac{(k + \log T)^2 }{\log(k + \log T)} \right) \notag\\
&\leq  c_0 \sum_{j=\lfloor \log T \rfloor}^{n + \lfloor \log T \rfloor} 4^{j - \log T} \exp\left( - c_1 \frac{j^2}{\log j} \right) \notag\\
&\leq  c_0  \exp\left( - c_1 \frac{(\log T)^2}{\log \log T}  \right) .
\alle
\medskip

\noindent\textit{Step 2: building a path using~\eqref{eqn-1pt-dist-across} and~\eqref{eqn-1pt-dist-around}.}
Henceforth assume that $E(T)$ occurs, so that~\eqref{eqn-1pt-dist-across} and~\eqref{eqn-1pt-dist-around} hold for every $k\in[0,n]_{\BB Z}$ and $x\in\mcl X_k$ and~\eqref{eqn-1pt-cont} holds for every $k\in [0,n]_{\BB Z}$. We will show that~\eqref{eqn-pt-dist} holds.

For $z\in U$ and $k\in [0,n]_{\BB Z}$, let $x_k^z \in \mcl X_k$ be chosen so that $z\in B_{2^{-k-2}}(x)$. 
Also let $P_k^z$ (resp.\ $\wt P_k^z$) be a path across (resp.\ around) $B_{2^{-k+4}}(x_k^z) \setminus B_{2^{-k-2}}(x_k^z)$ (resp.\  $B_{2^{-k+1}}(x_k^z ) \setminus B_{2^{-k }}(x_k^z) $) of minimal $D_{0,n}$-length. Note that the $D_{0,n}$-lengths of these paths are bounded by~\eqref{eqn-1pt-dist-across} and~\eqref{eqn-1pt-dist-around}.
For $k \in [1,n]_{\BB Z}$, the annuli 
\eqb
B_{2^{-k+1}}(x_k^z ) \setminus B_{2^{-k }}(x_k^z) \quad \text{and} \quad  B_{2^{-k+2}}(x_{k-1}^z ) \setminus B_{2^{-k+1 }}(x_{k-1}^z) 
\eqe
are each contained in $B_{2^{-k+4}}(x_k^z) \setminus B_{2^{-k-2}}(x_k^z)$. Consequently, the union of the paths $P_k^z$, $\wt P_k^z$, and $\wt P_{k-1}^z$ is connected. By iterating this, we get that for each $k \in [0,n]_{\BB Z}$, the union of the paths $P_j^z$ and $\wt P_j^z$ for $j \in [k,n]_{\BB Z}$ is connected. In particular, since $P_n^z \subset B_{2^{-n+2}}(z)$, we can use~\eqref{eqn-1pt-dist-across} and~\eqref{eqn-1pt-dist-around} to get
\eqb \label{eqn-1pt-dist-sum0}
D_{0,n}\left( B_{2^{-n+2}}(z)  , P_k^z \right) + \left( \text{$D_{0,n}$-length of $P_k^z$} \right)
\leq 2 T \sum_{j=k}^n   2^{-(\xi Q - \zeta) j} e^{\xi \Phi_{0,j}(x_j^z)} .
\eqe
\medskip

\noindent\textit{Step 3: comparing $\Phi_{0,k}(x_k^z)$ and $\Phi_{0,k}(z)$.}
Since $|x_k^z - z| \leq 2^{-j-2}$ for each $j\in [0,n]_{\BB Z}$, it follows from~\eqref{eqn-1pt-cont} that
\eqb \label{eqn-1pt-dist-diff}
|\Phi_{0,k}(x_k^z) - \Phi_{0,k}(z)| \leq \frac{1}{\xi}( \log T + \zeta k) . 
\eqe
Furthermore, by integrating $ e^{\xi \Phi_{0,n}(\cdot)}$ along a straight-line path from $z$ to $u$ and using~\eqref{eqn-1pt-cont}, we obtain
\eqb \label{eqn-1pt-dist-near}
\lambda_n^{-1} \sup_{u \in B_{2^{-n+2}}(z)} D_{0,n}(u,z) 
\leq T \lambda_n^{-1} 2^{-n+2} e^{\xi \Phi_{0,n}(z)} 
\preceq T 2^{-(\xi Q - \zeta) n} e^{\xi \Phi_{0,n}(z)}
\eqe
where in the last inequality we use that $\lambda_n \succeq 2^{-(1-\xi Q  + \zeta ) n}$ (Proposition~\ref{prop-Q-wn}). 
Plugging~\eqref{eqn-1pt-dist-diff} and~\eqref{eqn-1pt-dist-near} into~\eqref{eqn-1pt-dist-sum0} gives
\allb \label{eqn-1pt-dist-sum}
&\lambda_n^{-1} D_{0,n}\left( z , P_k^z \right) + \lambda_n^{-1} \left( \text{$D_{0,n}$-length of $P_k^z$} \right) \notag\\
&\qquad \leq \lambda_n^{-1}D_{0,n}\left( B_{2^{-n+2}}(z)  , P_k^z \right)  + \lambda_n^{-1} \left( \text{$D_{0,n}$-length of $P_k^z$} \right)  + \lambda_n^{-1} \sup_{u \in B_{2^{-n+2}}(z)} D_{0,n}(u,z)  \notag\\
&\qquad \preceq   T^2 \sum_{j=k}^n   2^{-(\xi Q - 2 \zeta) j} e^{\xi \Phi_{0,j}(z)} .
\alle
\medskip

\noindent\textit{Step 4: distance between two points.}
If $z,w\in U$ with $|z-w| \in [2^{-k} , 2^{-k+1}]$, then the paths $P_k^z$ and $P_k^w$ necessarily intersect. Moreover, if $\op{dist}(z,w) \leq 2\op{dist}(\ol V , \bdy U)$, then the paths $P_k^z$ and $P_k^w$ are contained in $U$. Therefore, by~\eqref{eqn-1pt-dist-sum} (applied to each of $z$ and $w$) together with the triangle inequality, we obtain that on $E(T)$, 
\eqb
\lambda_n^{-1} D_{0,n}\left( z,w ; U \right)
\preceq T^2 \sum_{j=k}^n   2^{-(\xi Q - 2 \zeta) j} \left(  e^{\xi \Phi_{0,j}(z)} + e^{\xi \Phi_{0,j}(w)} \right) .
\eqe
Replacing $T$ by $T^{1/2}$ and $\zeta$ by $\zeta/2$ (which results in an adjustment to $c_0,c_1$ in~\eqref{eqn-pt-dist-prob}) now gives~\eqref{eqn-pt-dist}. 
\end{proof}

\begin{proof}[Proof of Proposition~\ref{prop-pt-tight}]
To establish the tightness of $\lambda_n^{-1} D_{0,n}(z,w ; U)$, we note that the random variables $\Phi_{0,k}(z)$ and $\Phi_{0,k}(w)$ for $k\in\BB N$ are centered Gaussian with variance $k \log 2$. From this, it is easily seen that if $\zeta$ is chosen to be small enough that $\xi Q - \zeta > 0$, then the sum on the right side of~\eqref{eqn-pt-dist} converges a.s.\ as $n\rta\infty$ (with $T$, $z$, $w$ fixed).

If $|z-w| < 2 \op{dist}(\{z\} \cup \{w\} , \bdy U)$, then we can find a bounded connected open set $V$ with $\ol V\subset U$ such that $z,w \in V$ and $|z-w| \leq 2 \op{dist}(\ol V , \bdy U)$. Therefore, in this case the tightness of $\lambda_n^{-1} D_{0,n}(z,w ; U)$ follows from Lemma~\ref{lem-pt-dist}. 
If $|z-w| \geq 2 \op{dist}(\{z\} \cup \{w\} , \bdy U)$, then we can find finitely many points $z = z_0 , \dots, z_N = w$ in $U$ such that for each $j \in [1,N]_{\BB Z}$, we have $|z_j - z_{j-1} | \leq 2  \op{dist}(\{z_{j-1} \} \cup \{w_j\} , \bdy U)$. The tightness of $\lambda_n^{-1} D_{0,n}(z,w ; U)$ then follows from the triangle inequality and the tightness of $\lambda_n^{-1} D_{0,n}(z_{j-1} , z_j ; U)$ for each $j\in [1,N]_{\BB Z}$. 

To check tightness for $\left( \lambda_n^{-1} D_{0,n}(z,w) \right)^{-1}$, let $A$ be a Euclidean annulus which lies at positive distance from each of $z$ and $w$ and which disconnects $z$ from $w$. Then $D_{0,n}(z,w) \geq D_{0,n}\left(\text{across $A$}\right) $. 
Proposition~\ref{prop-perc-estimate'} implies that the random variables $\left(\lambda_n^{-1} D_{0,n}\left(\text{across $A$}\right) \right)^{-1}$ are tight, so also the random variables $\left( \lambda_n^{-1} D_{0,n}(z,w) \right)^{-1}$ are tight. 
\end{proof}

\subsection{Tightness for non-dyadic scales}
\label{sec-non-dyadic}

So far, all of our results have only been proven for dyadic scales, i.e., for $D_{0,n}$ with an integer value of $n$. However, our main theorems are stated for general values of $\ep$ which are not necessarily dyadic. So, we need to extend our results to the case of non-dyadic scales. 

We extend the definition~\eqref{eqn-phi-def} of $\Phi_{m,n}$ to the case when $m$ and $n$ are not necessarily integers. 
We similarly extend the definitions of $D_{m,n}$ and $\lambda_n$ from~\eqref{eqn-lfpp-def} and~\eqref{eqn-median-def}, respectively, to non-integer values of $m,n$. 
The purpose of this brief subsection is to show that our tightness results for $D_{0,n}$ continue to hold if $n$ is not required to be an integer.

\begin{prop} \label{prop-non-dyadic}
Propositions~\ref{prop-perc-estimate'} and~\ref{prop-pt-tight} continue to hold when $n$ is allowed to be any positive real number, not just an integer. 
\end{prop}

Throughout the rest of the paper, we will use Propositions~\ref{prop-perc-estimate'} and~\ref{prop-pt-tight} when $n$ is not necessarily an integer, without comment. 
The key input in the proof of Proposition~\ref{prop-non-dyadic} is the following lemma.

\begin{lem} \label{lem-non-dyadic-compare}
For each $s \in [-1,1]$ and each $n\in\BB N$, there is a coupling of two white noises $W , W'$ on $\BB C\times \BB R$ such that the following is true. Suppose that $\Phi_{0,n}$ is defined using the white noise $W$ and $\Phi_{0,n+s}'$ is defined in the same manner as $\Phi_{0,n+s}$ but with $W'$ in place of $W $. For each bounded open set $U\subset\BB C$, there are constants $c_0,c_1 >0$ depending only on $U$ such that for each $T>1$, 
\eqb \label{eqn-non-dyadic-compare}
\BB P\left[ \sup_{z\in U} |\Phi_{0,n+s}'(z) - \Phi_{0,n}(2^s z)|  > T \right] \leq c_0 e^{-c_1 T^2} .
\eqe
\end{lem}
\begin{proof}
We will treat the case when $s \geq 0$; the case when $s < 0$ is treated similarly but with the roles of $\Phi_{0,n}$ and $\Phi_{0,n+s}$ interchanged.
The relation~\eqref{eqn-phi-scale} continues to hold with non-integer values of $n,m,k$. In particular,
\eqbn
\Phi_{0,n}(2^s \cdot) \eqD \Phi_{s,n+s}(  \cdot) .
\eqen
Hence, we can find a coupling of two white noises $W \eqD W'$ such that if $\Phi_{0,n+s}'$, $\Phi_{0,s}'$, and $\Phi_{s,n+s}'$ are defined with $W'$ in place of $W$, then a.s.\ $\Phi_{0,n}(2^s \cdot) = \Phi_{s,n+s}'(\cdot)$. Recalling that $\Phi_{0,n+s}' = \Phi_{s,n+s}' + \Phi_{0,s}'$,  
\eqb  \label{eqn-non-dyadic-phi}
\sup_{z\in U} |\Phi_{0,n+s}'(z) - \Phi_{0,n}(2^s z)| = \sup_{z\in U} |\Phi_{0,s}'(z)| .
\eqe 
The function $\Phi_{0,s}'$ is continuous and centered Gaussian with pointwise variance $s\log 2$. 
By the Borell-TIS inequality~\cite{borell-tis1,borell-tis2} (see, e.g.,~\cite[Theorem 2.1.1]{adler-taylor-fields}), there are constants $c_0,c_1> 0$ depending only on $U$ such that for each $T>1$, 
\eqb \label{eqn-non-dyadic-btis}
\BB P\left[   \sup_{z\in U} |\Phi_{0,s}'(z)|  - \BB E\left[  \sup_{z\in U} |\Phi_{0,s}'(z)|  \right]  > T\right] \leq c_0 e^{-c_1 T^2} .
\eqe
By Fernique's criterion~\cite{fernique-criterion} (see~\cite[Theorem 4.1]{adler-gaussian} or~\cite[Lemma 2.3]{dzz-heat-kernel} for the version we use here) together with~\cite[Lemma 4]{dddf-lfpp}, $\BB E\left[  \sup_{z\in U} |\Phi_{0,s}'(z)|  \right] $ is bounded above by a constant depending only on $U$. 
By combining this with~\eqref{eqn-non-dyadic-btis} and recalling~\eqref{eqn-non-dyadic-phi}, we get~\eqref{eqn-non-dyadic-compare} (with possibly modified values of $c_0,c_1$).
\end{proof}

\begin{lem} \label{lem-non-dyadic-median}
For each $n\in\BB N$ and each $s\in[-1,1]$, we have $    \lambda_{n+s} \asymp \lambda_n$ with the implicit constant depending only on $\xi$. 
\end{lem}
\begin{proof}
For $p \in (0,1)$, let $\ell_n^s(p)$ be the $p$th quantile of $D_{0,n}(2^s \bdy_{\op L} \BB S , 2^s \bdy_{\op R} \BB S ; 2^s \BB S)$, i.e., $\ell_n^s(p)$ is defined in the same manner as $\ell_n(p)$ from~\eqref{eqn-quantile-def} but with $2^s\BB S$ instead of $\BB S$. 
Lemma~\ref{lem-non-dyadic-compare} implies that there is a $p\in (0,1)$ depending only on $\xi$ such that $\ell_n^s(p) \leq \lambda_{n+s} \leq \ell_n^s(1-p)$. 
By Proposition~\ref{prop-perc-estimate'}, $\ell_n^s(p)$ and $\ell_n^s(1-p)$ each differ from $\lambda_n$ by at most a $\xi$-dependent positive constant factor. 
Thus the lemma statement holds. 
\end{proof}

\begin{proof}[Proof of Proposition~\ref{prop-non-dyadic}]
This follows by combining Propositions~\ref{prop-perc-estimate'} and~\ref{prop-pt-tight}, respectively, with Lemmas~\ref{lem-non-dyadic-compare} and~\ref{lem-non-dyadic-median}. 
\end{proof}

\subsection{Comparison of $D_{0,n}$ and $D_h^\ep$} 
\label{sec-gff-compare}

Recall the definitions of $h_\ep^*$ from~\eqref{eqn-gff-convolve}, $D_h^\ep$ from~\eqref{eqn-gff-lfpp}, and $\frk a_\ep$ from~\eqref{eqn-gff-constant}. 
Before we transfer our above results for $D_{0,n}$ to results for $D_h^\ep$, we need some further preliminary facts and definitions concerning $D_h^\ep$. 

The metrics $D_h^\ep$ possess an important scale invariance property which is easy to check from the definition (see~\cite[Lemma 2.6]{lqg-metric-estimates}). 
To state this property, we let $h_r(z)$ for $r>0$ and $z\in\BB C$ be the average of $h$ over the circle $\bdy B_r(z)$ (see~\cite[Section 3.1]{shef-kpz} for more on the circle average process). Then our scale invariance property reads
\eqb \label{eqn-gff-scale}
\left( D_h^{ \ep/r} (z,w)  \right)_{z,w\in\BB C} \eqD \left( r^{-1} e^{-\xi h_r(0)} D_h^{ \ep} ( r  z , r  w) \right)_{z,w\in\BB C} ,\quad\forall  r , \ep  > 0.
\eqe

The random variable $h_\ep^*(z)$ does not depend locally on $h$ since the heat kernel $p_{\ep^2/2}$ is non-zero on all of $\BB C$. For this reason we will also need to work with a localized version of $h_\ep^*$, which we will denote by $\wh h_\ep^*$. This same function $\wh h_\ep^*$ is also used in~\cite[Section 2.1]{lqg-metric-estimates}. 
 
For $\ep > 0$, let $\chi_\ep  : \BB C\rta [0,1]$ be a deterministic, smooth, radially symmetric bump function which is identically equal to 1 on $B_{\ep^{1/2}/2}(0)$ and vanishes outside of $B_{\ep^{1/2}}(0)$ (in fact, the power $1/2$ could be replaced by any $p\in (0,1)$). 
We can choose $\chi_\ep$ in such a way that $\ep\mapsto \chi_\ep$ is a continuous mapping from $(0,\infty)$ to the space of continuous functions on $\BB C$, equipped with the uniform topology. 
We define
\eqb \label{eqn-localized-def}
\wh h_\ep^*(z) :=  \int_{\BB C}  h(w) \chi_\ep(z-w) p_{\ep^2/2} (z,w) \, dw ,
\eqe
with the integral interpreted in the sense of distributional pairing.
Since $\psi_\ep$ vanishes outside of $B_{\ep^{1/2}}(0)$, we have that $\wh h_\ep^*(z)$ is a.s.\ determined by $h|_{B_{\ep^{1/2}}(z)}$. 
It is easy to see that $\wh h_\ep^*$ a.s.\ admits a continuous modification (see Lemma~\ref{lem-localized-approx} below).
We henceforth assume that $\wh h_\ep^*$ is replaced by such a modification. 

As in~\eqref{eqn-gff-lfpp}, we define the localized LFPP metric
\eqb \label{eqn-localized-lfpp}
\wh D_h^\ep(z,w) := \inf_{P : z\rta w} \int_0^1 e^{\xi \wh h_\ep^*(P(t))} |P'(t)| \,dt ,
\eqe
where the infimum is over all piecewise continuously differentiable paths from $z$ to $w$. 
By the definition of $\wh h_\ep^*$, 
\eqb \label{eqn-localized-property}
\text{for any open $U\subset \BB C$, the internal metric $\wh D_h^\ep(\cdot,\cdot; U)$ is a.s.\ determined by $h|_{B_{\ep^{1/2}}(U)}$} .
\eqe 
The reason why we make~\eqref{eqn-localized-property} an indented equation is for ease of reference later.

The following lemma is essentially a restatement of~\cite[Lemma 2.1]{lqg-metric-estimates}. 
  
\begin{lem}[\cite{lqg-metric-estimates}] \label{lem-localized-approx} 
Almost surely, $(z,\ep)\mapsto \wh h_\ep^*(z)$ is continuous. 
Furthermore, for each bounded open set $U\subset \BB C$, a.s.\ 
\eqb \label{eqn-localized-approx}
\lim_{\ep\rta 0} \sup_{z\in \ol U} |h_\ep^*(z) - \wh h_\ep^*(z)|  = 0.
\eqe
In particular, a.s.\
\eqb \label{eqn-localized-lfpp-approx}
\lim_{\ep\rta 0}  \frac{\wh D_h^\ep(z,w;U)}{D_h(z,w;U)} = 1 , \quad \text{uniformly over all $z,w\in U$ with $z\not=w$}.
\eqe
The same is true if we replace $h$ by a zero-boundary GFF on an open subset $V$ of $\BB C$ and we require that $\ol U\subset V$.
\end{lem}

The last part of Lemma~\ref{lem-localized-approx} (concerning the zero-boundary GFF) is not explicitly stated in~\cite{lqg-metric-estimates}, but it follows from the same proof as in the case of a whole-plane GFF.

The following lemma is our main tool for transferring results between $D_{0,n}$ and $D_h^\ep$. 

\begin{lem} \label{lem-gff-compare} 
Let $U\subset\BB C$ be a bounded open set. 
Let $\ep \in (0,1)$ and let $n := \log_2 \ep^{-1}$ (note that $n$ is not required to be an integer). There is a coupling of $h$ with the white noise $W$ from~\eqref{eqn-phi-def} and constants $c_0,c_1 >0$ depending only on $U,\xi$ such that 
\eqb \label{eqn-gff-compare}
\BB P\left[ \sup_{z\in U} |\Phi_{0,n}(z) - h_\ep^*(z)| > x\right] \leq c_0 e^{-c_1 x^2}  + o_\ep(1) ,\quad\forall x > 0 
\eqe
where the $o_\ep(1)$ is deterministic and tends to zero as $\ep\rta 0$. 
\end{lem}
\begin{proof}
Throughout the proof, $c_0,c_1$ denote deterministic positive constants depending only on $U,\xi$ which may change from line to line.
Let $U'' \supset U' \supset U$ be bounded open sets with $\ol U  \subset U'$ and $\ol U '\subset U''$. Let $\rng h$ be a zero-boundary GFF on $U''$. Define $\rng h_\ep^*(z)$ for $z\in U''$ as in~\eqref{eqn-gff-convolve} but with $\rng h$ in place of $h$ (we take $\rng h = 0$ outside of $U''$). 
By~\cite[Proposition 29]{dddf-lfpp}, there is a coupling of $\rng h$ with $\Phi_{0,n}$ such that
\eqb \label{eqn-gff-compare0}
\BB P\left[ \sup_{z\in \ol U'} |\Phi_{0,n}(z) - \rng h_\ep^*(z)| > x\right] \leq c_0 e^{-c_1 x^2}  .
\eqe

To transfer from $\rng h$ to $h$, we use the Markov property of $h$ (see, e.g.,~\cite[Lemma 2.1]{gms-harmonic}) to couple $h$ and $\rng h$ in such a way that $h|_{U''} = \rng h + \frk h$, where $\frk h$ is a random centered Gaussian harmonic function on $U''$. Since $\frk h$ is continuous on $\ol U'$, the Borell-TIS inequality implies that
\eqb \label{eqn-gff-compare-harmonic}
\BB P\left[ \sup_{z \in \ol U'} |\frk h(z)| > x \right]  \leq c_0 e^{-c_1 x^2} .
\eqe

We want to use~\eqref{eqn-gff-compare-harmonic} to compare $h_\ep^*(z)$ and $\rng h_\ep^*(z)$, but we cannot do this directly since $h_\ep^*(z)$ is not determined by $h|_{\ol U'}$. So, we instead need to use localized versions of $h_\ep^*(z)$ and $\rng h_\ep^*(z)$.
Define $\wh h_\ep^*$ as in~\eqref{eqn-localized-def} and define $\wh{\rng h^*}_\ep$ as in~\eqref{eqn-localized-def} with $\rng h$ in place of $h$. Also note that by the mean value property of $\frk h$, if we define $\wh{\frk h}_\ep^*(z)$ as in~\eqref{eqn-localized-def} with $\frk h$ in place of $h$, then $\wh{\frk h}_\ep^*(z) = \frk h(z)$ whenever $B_{\ep^{1/2}}(z)\subset U''$. Since $\wh h_\ep^*(z)$ (resp.\ $\wh{\rng h}_\ep^*(z)$) is determined by the restriction of $h$ (resp.\ $\rng h$) to $B_{\ep^{1/2}}(z)$, it follows from~\eqref{eqn-gff-compare-harmonic} that whenever $\ep^{1/2}$ is smaller than the Euclidean distance from $\ol U$ to $\bdy U'$,  
\eqb \label{eqn-gff-compare-local}
\BB P\left[ \sup_{z \in \ol U } |   \wh h_\ep^*(z) - \wh{\rng h^*}_\ep(z)     | > x \right]  \leq c_0 e^{-c_1 x^2} .
\eqe
On the other hand, by Lemma~\ref{lem-localized-approx}, a.s.\ 
\eqb \label{eqn-gff-compare-approx}
\lim_{\ep\rta 0} \sup_{z \in \ol U } \max\{ | h_\ep^*(z) - \wh h_\ep^*(z) | , |\rng h_\ep^*(z) -   \wh{\rng h^*}_\ep(z)     |  = 0. 
\eqe
By combining~\eqref{eqn-gff-compare0}, \eqref{eqn-gff-compare-local}, and~\eqref{eqn-gff-compare-approx} we obtain~\eqref{eqn-gff-compare}.
\end{proof}

\begin{proof}[Proof of Proposition~\ref{prop-gff-pt-tight}]
First assume that $U$ is bounded. 
By Lemma~\ref{lem-gff-compare}, if we set $n = \log_2\ep^{-1}$ then the metrics $D_{0,n}$ and $D_h^\ep$ can be coupled together so that the corresponding internal metrics on $U$ are bi-Lipschitz equivalent, and moreover the laws of the Lipschitz constants in each direction can be bounded independently of $\ep$. 
From this bi-Lipschitz equivalence applied with $U$ equal to a neighborhood of the unit square $\BB S$, it follows that there exists $p\in (0,1/2)$ depending only on $\xi$ such that in the notation~\eqref{eqn-quantile-def}, 
\eqbn
\ell_n(p)  \leq \frk a_\ep \leq \ell_n(1-p) .
\eqen
Due to the tightness of the law of $\lambda_n^{-1} L_{0,n}$, we have $\ell_n(p) \succeq \lambda_n$ and $\ell_n(1-p) \preceq \lambda_n$, with the implicit constants depending only on $\xi$. Thus~\eqref{eqn-scale-compare} holds.  

The tightness of $\frk a_\ep^{-1} D_h^\ep(K_1,K_2 ; U)$ and its reciporical for bounded $U$ now follows from the analogous tightness statement for $\lambda_n^{-1} D_{0,n} (K_1,K_2 ; U)$ (Propositions~\ref{prop-perc-estimate'} and~\ref{prop-pt-tight}) together with Proposition~\ref{prop-non-dyadic} (to allow for non-integer values of $n$).  

We now treat the case when $U$ is unbounded. 
Let $V\subset U$ be a bounded connected open set which contains $K_1$ and $K_2$. 
Then $D_h^\ep(K_1,K_2 ; U) \leq D_h^\ep(K_1,K_2 ; V)$ so the case when $U$ is bounded already implies that $\frk a_\ep^{-1} D_h^\ep(K_1,K_2 ; U)$ is tight. 
To get the tightness of $\left(\frk a_\ep^{-1} D_h^\ep(K_1,K_2 ; U) \right)^{-1}$, let $\Pi$ be Jordan curve in $V$ which disconnects $K_1$ from $K_2 \cup \bdy V$. 
Then
\eqb
D_h^\ep\left(K_1,K_2 ; U \right) \geq D_h^\ep\left( K_1 , \Pi ; U\right) = D_h^\ep\left( K_1 , \Pi ; V\right) 
\eqe
so the tightness of $\left(\frk a_\ep^{-1} D_h^\ep(K_1,K_2 ; U) \right)^{-1}$ follows from the tightness of $\left(\frk a_\ep^{-1} D_h^\ep(K_1,\Pi ; V) \right)^{-1}$. 
\end{proof}

\begin{proof}[Proof of Proposition~\ref{prop-Q}]
This is immediate from Proposition~\ref{prop-Q-wn} and~\eqref{eqn-scale-compare} (plus Lemma~\ref{lem-non-dyadic-median} to deal with non-integer values of $n$). 
\end{proof}

\begin{proof}[Proof of Proposition~\ref{prop-gff-scale-constant}]
By Lemma~\ref{lem-exponent-relation} combined with Lemma~\ref{lem-non-dyadic-median} (to allow for non-integer $n$ and $K$), it holds for each $n > K > 0$ that, 
\eqb \label{eqn-use-exponent-relation}
2^{-(1-\xi Q + \zeta) K} \preceq \frac{\lambda_n}{\lambda_{n-K}} \preceq 2^{-(1-\xi Q -\zeta) K} ,
\eqe
with the implicit constants depending only on $\zeta,\xi$. 
To prove~\eqref{eqn-gff-scale-constant1}, we apply Lemma~\ref{lem-exponent-relation} with $n  > 0$ such that $\ep  = 2^{-n}$ and $K > 0$ such that $2^{-K} = r$, then use~\eqref{eqn-scale-compare} to compare $\frk a_\ep$ with $\lambda_n$ and $\frk a_{\ep/r}$ with $\lambda_{n-K}$. 
To prove~\eqref{eqn-gff-scale-constant2}, we apply Lemma~\ref{lem-exponent-relation} with $n > 0$ such that $\ep / r = 2^{-n}$ and $K > 0$ such that $2^{-K} =r$, then use~\eqref{eqn-scale-compare} to compare $\frk a_{\ep/r}$ with $\lambda_n$ and $\frk a_{\ep}$ with $\lambda_{n-K}$.
\end{proof}

\subsection{Additional estimates for $D_h^\ep$}
\label{sec-gff-estimates}

With Proposition~\ref{prop-gff-pt-tight} established, we can now prove some useful estimates for the LFPP metric $D_h^\ep$. 
We first have a variant of Lemma~\ref{lem-perc-field} for $D_h^\ep$. For the statement, we recall that $h_r(0)$ is the average of $h$ over the circle of radius $r$ centered at 0. 

\begin{lem} \label{lem-perc-gff}
Let $U\subset \BB C$ be a bounded open set and let $K_1,K_2\subset U$ be disjoint compact connected sets which are not singletons. 
There are constants $c_0,c_1 >0$ depending only on $U,K_1,K_2,\xi$ such that for each $r > 0$, each $\ep \in (0,r)$, and each $T > 3$, 
\eqb \label{eqn-perc-gff-lower}
\BB P\left[ D_h^\ep (r K_1, r K_2 ; r U)  < T^{-1} r  \frk a_{\ep / r} e^{\xi h_r(0)}  \right] \leq c_0 e^{-c_1 (\log T)^2}  + o_\ep(1) 
\eqe
and
\eqb \label{eqn-perc-gff-upper}
\BB P\left[ D_h^\ep(r K_1, r K_2 ;r U)  > T  r \frk a_{\ep / r} e^{\xi h_r(0)} \right] \leq c_0 e^{-c_1 (\log T)^2 / \log\log T} + o_\ep(1) ,
\eqe
where the rate of convergence of the $o_\ep(1)$ depends on $U,K_1,K_2,\xi , T , r$.
\end{lem}
\begin{proof}
By Proposition~\ref{prop-perc-estimate'} combined with Lemma~\ref{lem-gff-compare} (to compare $D_h^\ep$ and $D_{0,n}$) and~\eqref{eqn-scale-compare} of Proposition~\ref{prop-gff-pt-tight} (to compare $\frk a_\ep$ and $\lambda_n$), we obtain the lemma statement in the case when $r =1$. 
For a general choice of $r$, we use~\eqref{eqn-gff-scale} which gives
\eqbn
D_h^{\ep/r}(K_1,K_2 ;U) \eqD  r^{-1} e^{-\xi h_r(0)} D_h^{ \ep} ( r K_1 , rK_2 ; r U ) .
\eqen
So, the lemma for a general choice of $r$ follows from the case when $r=1$ (with $\ep/r$ in place of $r$).  
\end{proof}

Our next lemma gives us tells us that, in a certain precise sense, $D_h^\ep$-distance from a compact set to $\infty$ is infinite. 

\begin{lem} \label{lem-infty-dist}
For each fixed $r > 0$ and each $T>1$, 
\eqb \label{eqn-infty-dist}
 \lim_{R\rta\infty} \lim_{\ep\rta 0} \BB P\left[ \frk a_\ep^{-1} D_h^\ep\left( \bdy B_r(0) , \bdy B_R(0) \right) \geq T\right] = 1 .
\eqe
\end{lem}
\begin{proof}
Fix a small constant $\zeta > 0$. By Lemma~\ref{lem-perc-gff} (applied with $T=2^{\zeta k}$), for each $k\in\BB N$,
\eqb \label{eqn-infty-dist-perc}
\lim_{\ep\rta 0} \BB P\left[ \frk a_\ep^{-1} D_h^\ep\left( \text{across $B_{2^k}(0) \setminus B_{2^{ k-1}}(0)$} \right) \geq  2^{-\zeta  k}  \frac{2^k \frk a_{2^{-k} \ep } }{  \frk a_\ep } e^{\xi h_{2^k}(0)}  \right]
= 1 .
\eqe
By Proposition~\ref{prop-gff-scale-constant}, $2^k \frk a_{2^{-k}\ep} / \frk a_\ep = 2^{\xi Q k  + o_k(k)} $ as $k\rta\infty$, uniformly over all $\ep \in (0,1)$. 
Since $h_{2^k}(0)$ is centered Gaussian with variance $k \log 2$, it holds with probability tending to 1 as $k\rta\infty$ that $h_{2^k}(0) \leq \frac{\zeta}{\xi} k \log 2$. 
By combining these last two estimates with~\eqref{eqn-infty-dist-perc} and shrinking the value of $\zeta$, we get
\eqb \label{eqn-infty-dist0}
\lim_{k\rta\infty} \lim_{\ep\rta 0} \BB P\left[\frk a_\ep^{-1} D_h^\ep\left( \text{across $B_{2^k}(0) \setminus B_{2^{ k-1}}(0)$} \right) \geq  2^{(\xi Q - \zeta) k} \right] = 1 .  
\eqe
If $\zeta \in (0,\xi Q)$, then $\lim_{k\rta\infty} 2^{(\xi Q - \zeta) k} = \infty$. Furthermore, if $2^k \in [2r , R]$, then $D_h^\ep\left( \bdy B_r(0) , \bdy B_R(0) \right) \geq  D_h^\ep\left( \text{across $B_{2^k}(0) \setminus B_{2^{ k-1}}(0)$} \right)$. Hence~\eqref{eqn-infty-dist} follows from~\eqref{eqn-infty-dist0}.
\end{proof}

\section{Tightness of LFPP}
\label{sec-lfpp-tight}

\subsection{Subsequential limits}
\label{sec-ssl}

In this subsection we will extract a subsequence of $\ep$-values tending to zero along which a number of functionals of $\frk a_\ep^{-1} D_h^\ep$ converge jointly in law. 
In Section~\ref{sec-dist-def}, we will use these convergence statements to produce a subsequential limiting metric $D_h$. 
The rest of the section is devoted to showing that $D_h$ satisfies the conditions of Theorems~\ref{thm-lfpp-tight} and Theorem~\ref{thm-ssl-properties}. 
 
We start with some definitions which will be convenient since we can only show the subsequential convergence of the joint law of countably many functionals of $\frk a_\ep^{-1} D_h^\ep$. 

\begin{defn} \label{def-rational}
A \emph{rational circle} is a circle of the form $O = \bdy B_r(u)$ for $r\in \BB Q\cap [0,\infty)$ and $u\in\BB Q^2$. 
In order to make statements more succinct, we view a point in $\BB Q^2$ as a rational circle with radius 0. 
A \emph{rational annulus} is the closure of the region $A$ between two non-intersecting rational circles with positive radii. 
We say that a rational circle with positive radius or a rational annulus \emph{surrounds} $z\in\BB C$ if it does not contain $z$ and it disconnects $z$ from $\infty$. 
\end{defn}

By countably many applications of Proposition~\ref{prop-gff-pt-tight}, for any sequence of $\ep$-values tending to zero we can find a subsequence $\mcl E$ and a coupling of the GFF $h$ with random variables $\rng D_h(O_1,O_2)$ for rational circles $O_1,O_2$, random variables $D_h^\ep(\text{around $A$})$ for rational annuli $A$ for which the following is true.
Let $h^\ep$ for $\ep\in\mcl E$ be a random variable with the same law as $h$, and assume that $D_h^\ep$ has been defined with $h^\ep$ in place of $h$. 
Then we have the following joint convergence in law as $\mcl E\ni \ep\rta 0$:
\eqb \label{eqn-ssl-circle}
\frk a_\ep^{-1} D_h^\ep\left( O_1 , O_2 \right) \rta \rng D_h \left( O_1 , O_2 \right) ,\quad \text{$\forall$ rational circles $O_1,O_2$} 
\eqe 
\eqb \label{eqn-ssl-around}
\frk a_\ep^{-1} D_h^\ep\left( \text{around $A$} \right) \rta \rng D_h \left( \text{around $A$} \right) ,\quad \text{$\forall$ rational annuli $A$}   
\eqe  
\eqb \label{eqn-ssl-h}
h_r^\ep \rta h_r \quad \text{w.r.t.\ the local uniform topology} \quad\text{and} \quad h^\ep \rta h \quad \text{w.r.t.\ the distributional topology} .
\eqe 
The reason why we write $\rng D_h$ instead of $D_h$ for the limiting random variables is that we do \emph{not} know a priori that these quantities are distances with respect to the metric $D_h$ which we define below.  
Nevertheless, by a slight abuse of notation, for a rational annulus $A$, we define
\eqb \label{eqn-ssl-across}
\rng D_h\left( \text{across $A$} \right) := \rng D_h\left( O_1 , O_2 \right) 
\eqe
where $O_1,O_2$ are the boundary circles of $A$. 

In order to use various scaling arguments and also to check assertion~\ref{item-constant} of Theorem~\ref{thm-lfpp-tight}, we will also need to extract subsequential limits of the ratios of scaling constants $\frk a_\ep$. 
By Proposition~\ref{prop-gff-scale-constant}, after possibly replacing $\mcl E$ by a further subsequence we can arrange that there are numbers $\frk c_r > 0$ for each $r\in \BB Q\cap (0,\infty)$ such that
\eqb \label{eqn-ssl-constant}
\lim_{\mcl E\ni \ep} \frac{r \frk a_{\ep/r}}{\frk a_\ep} = \frk c_r .
\eqe
Note that Proposition~\ref{prop-gff-scale-constant} implies that
\eqb \label{eqn-ssl-constant-lim}
\frk c_r = r^{\xi Q + o_r(1)} ,\quad\text{as $\BB Q\ni r \rta 0$ or $\BB Q\ni r \rta\infty$}. 
\eqe
Throughout the rest of this section we fix a sequence $\mcl E$ for which the convergence of joint laws~\eqref{eqn-ssl-circle}, \eqref{eqn-ssl-around}, and~\eqref{eqn-ssl-h} hold and also~\eqref{eqn-ssl-constant} holds. 
 
By the Skorokhod embedding theorem, we can find a coupling of $\{(h^\ep , D_h^\ep)\}_{\ep \in\mcl E}$ with the random variables $\rng D_h(O_1,O_2)$ and $\rng D_h(\text{around $A$})$ such that the convergence~\eqref{eqn-ssl-circle}, \eqref{eqn-ssl-around}, and~\eqref{eqn-ssl-h} occurs a.s. Note that with this choice of coupling, $\{D_h^\ep\}_{\ep\in\mcl E}$ are defined using $h^\ep \eqD h$ instead of $h$. 

We have the following elementary relations between the random variables defined above, which are the starting point of the proofs in this section. 

\begin{lem} \label{lem-ssl-across-around}
Almost surely, the following is true. 
\begin{enumerate}[(i)]
\item Let $O_0,\dots,O_k$ be disjoint rational circles such that $O_j$ disconnects $O_{j-1}$ from $O_{j+1}$ for each $j\in[1,k-1]_{\BB Z}$. Then $\rng D_h(O_0,O_k) \geq \sum_{j=1}^k \rng D_h(O_j , O_{j-1})$.  \label{item-ssl-across}
\item Let $O_1,O_1' , O_2 , O_2'  $ be rational circles and suppose that $A$ is a rational annulus whose interior is disjoint from each of $O_1,O_1',O_2,O_2'$. Suppose also that $O_1,O_1'$ lie in different connected components of $\ol{\BB C\setminus A}$ and $O_2,O_2'$ lie in different connected components of $\ol{\BB C\setminus A}$. Then $\rng D_h(O_1,O_2) \leq \rng D_h(O_1,O_1') + \rng D_h(O_2,O_2') + \rng D_h(\text{around $A$})$. \label{item-ssl-around}
\item If $O_1,O_2$ are disjoint rational annuli then $\rng D_h(O_1,O_2) >0$. \label{item-ssl-pos}
\end{enumerate}
\end{lem}
\begin{proof} 
To prove assertions~\eqref{item-ssl-across} and~\eqref{item-ssl-around}, note that the analogous statements with $\frk a_\ep^{-1} D_h^\ep$ in place of $\rng D_h^\ep$ are obvious from the fact that that $\frk a_\ep^{-1} D_h^\ep$ is a length metric (see Figure~\ref{fig-ssl}). Passing through to the limit gives these two assertions. 
Assertion~\eqref{item-ssl-pos} follows since the random variables $(\frk a_\ep^{-1} D_h^\ep(O_1,O_2) )^{-1}$ are tight (Proposition~\ref{prop-gff-pt-tight}).
\end{proof}

Roughly speaking, Lemma~\ref{lem-ssl-across-around} says that the random variables $\rng D_h(O_0,O_k)$ and $\rng D_h(\text{across $A$})$ behave like distances w.r.t.\ a length metric, even though we do not know that these random variables are actual distances w.r.t.\ a length metric.

\begin{figure}[t!]
 \begin{center}
\includegraphics[scale=1]{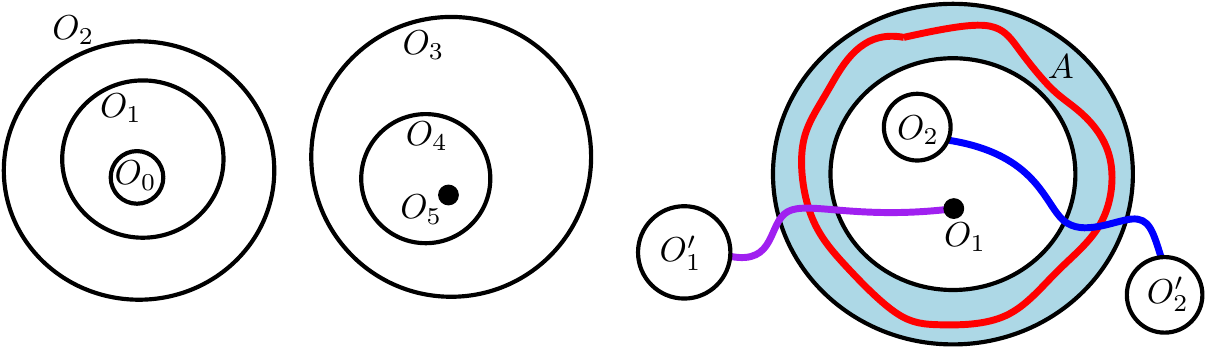}
\vspace{-0.01\textheight}
\caption{\textbf{Left:} Illustration of the proof of assertion~\eqref{item-ssl-across} of Lemma~\ref{lem-ssl-across-around} in the case when $k=5$. Note that here $O_5$ is a rational circle of radius 0 (i.e., a point). The lemma follows since any path from $O_0$ to $O_5$ must pass through $O_1,O_2,O_3,O_4$. 
\textbf{Right:} Illustration of the proof of assertion~\eqref{item-ssl-around} of Lemma~\ref{lem-ssl-across-around}. The paths of minimal $\frk a_\ep^{-1} D_h^\ep$-length from $O_1$ to $O_1'$ and from $O_2$ to $O_2'$ must each cross $A$. So, the union of these two paths with a path of minimal $\frk a_\ep^{-1} D_h^\ep$-length around $A$ is connected. This gives us an upper bound for $\frk a_\ep^{-1} D_h^\ep(O_1,O_2)$. Note that it does not matter whether $O_1$ and $O_2$ lie in the same connected component of $\ol{\BB C\setminus A}$ or not. Also note that one or more of $O_1,O_1',O_2,O_2'$ is allowed to be equal to the inner or outer boundary of $A$. 
}\label{fig-ssl}
\end{center}
\vspace{-1em}
\end{figure}

\subsection{Definition of the limiting metric}
\label{sec-dist-def}

We now define the subsequential limiting metric appearing in our main theorem statements. 
For $z,w\in\BB C$, let
\allb \label{eqn-dist-def}
D_h(z,w) :=
\begin{cases}
 \lim_{O_z \downarrow z , O_w \downarrow w} \rng D_h(O_z,O_w) ,\quad & z\not=w   \\
 0,\quad &z=w 
\end{cases}
\alle
where the limit is over any sequence of rational circles with positive radii $O_z$ surrounding $z$ and $O_w$ surrounding $w$ whose radii shrink to zero.
The following lemma tells us that~\eqref{eqn-dist-def} is well-defined.

\begin{lem} \label{lem-dist-lim}
Almost surely, the limit in~\eqref{eqn-dist-def} exists for all pairs of distinct points $z,w\in\BB C$ (it is allowed to be equal to $\infty$) and does not depend on the sequence of approximating circles. 
\end{lem}

Once Lemma~\ref{lem-dist-lim} is established, it is immediate from~\eqref{eqn-dist-def} that $D_h$ is symmetric. 
Furthermore, assertion~\eqref{item-ssl-pos} of Lemma~\ref{lem-ssl-across-around} (together with the monotonicity considerations described in the proof of Lemma~\ref{lem-dist-lim} just below) implies that $D_h$ is positive definite.
We will check that $D_h$ satisfies the triangle inequality in Lemma~\ref{prop-triangle-inequality} below, so that $D_h$ is a metric on $\BB C$, although it is allowed to take on infinite values. 
For the proof of Lemma~\ref{lem-dist-lim} we need the following definition.

\begin{defn} \label{def-nest}
We say that a sequence of rational circles $\{O^m\}_{m\in\BB N}$ with positive radii \emph{nests down to $z$} if $O^m$ surrounds $O^{m+1}$ for each $m\in\BB N$, $O^m$ surrounds $z$ for each $m\in\BB N$, and the radii of the $O^m$'s tends to zero as $m\rta\infty$. 
\end{defn}

\begin{proof}[Proof of Lemma~\ref{lem-dist-lim}]
The key observation is as follows. If $O_z,O_z'$ are rational circles surrounding $z$ and $O_w,O_w'$ are rational circles surrounding $w$ such that $O_z\cap O_w = \emptyset$, $O_z$ surrounds $O_z'$, and $O_w$ surrounds $O_w'$, then by assertion~\eqref{item-ssl-across} of Lemma~\ref{lem-ssl-across-around}, 
\eqb \label{eqn-dist-lim-mono}
\rng D_h(O_z,O_w) \leq \rng D_h(O_z' , O_w'). 
\eqe 
This gives us a monotonicity property for the limit in~\eqref{eqn-dist-def} which will allow us to check convergence. We remark that this monotonicity property is the main reason why we define $D_h(z,w)$ in terms of $\rng D_h$-``distances" between rational circles which shrink to $z$ and $w$ rather than in terms of $\rng D_h$-``distances" between rational points which converge to $z$ and $w$. 

We first check convergence for sequences of rational circles $\{O_z^m\}_{m\in\BB N}$ and $\{O_w^m\}_{m\in\BB N}$ which nest down to $z$ and $w$, respectively.
By~\eqref{eqn-dist-lim-mono}, $\rng D_h(O_z^m , O_w^m)$ is non-decreasing in $m$ provided $m$ is large enough that the radii of the circles is smaller than $|z-w|/2$. Therefore, the limit
\eqb
\beta := \lim_{m\rta\infty} \rng D_h(O_z^m , O_w^m) 
\eqe
exists (it is allowed to be equal to $\infty$). 

We will now check that this limit does not depend on the choice of $\{O_z^m\}_{m\in\BB N}$ and $\{O_w^m\}_{m\in\BB N}$. 
Let $\{\wt O_z^m\}_{m\in\BB N}$ and $\{\wt O_w^m\}_{m\in\BB N}$ be another pair of sequences of rational circles which nest down to $z$ and $w$, respectively and let $\wt\beta :=  \lim_{m\rta\infty} \rng D_h(O_z^m , O_w^m) $. We will show that $\beta \leq \wt \beta $ (the inequality in the other direction follows by symmetry). Indeed, since the radii of the $O_z^m$'s and $O_w^m$'s tends to zero, for each $m_0 \in\BB N$ there exists $m_1 \geq m_0$ such that $\wt O_z^{m_0}$ surrounds $O_z^m$ and $\wt O_w^{m_0}$ surrounds $O_w^m$ for each $m\geq m_1$. 
By~\eqref{eqn-dist-lim-mono}, it follows that
\eqb
\rng D_h(\wt O_z^{m_0} , \wt O_w^{m_0} ) \leq \rng D_h(O_z^m , O_w^m ), \quad\forall m \geq m_1 .
\eqe
Sending $m\rta \infty$ and then $m_0 \rta\infty$ shows that $\wt\beta \leq \beta$, as required.

If $\{O_z^m\}_{m\in\BB N}$ and $\{O_w^m\}_{m\in\BB N}$ are arbitrary sequences of rational circles surrounding $z$ and $w$, respectively, whose radii shrink to zero (which are not necessarily nested) then there is a subsequence $m_j \rta\infty$ along which $O_z^{m_j}$ and $O_w^{m_j}$ nest down to $z$ and $w$, respectively. By the preceding paragraph it follows that every subsequence of the numbers $\rng D_h(O_z^m , O_w^m)$ has a further subsequence which converges to $\beta$. This implies that $\lim_{m\rta\infty} \rng D_h(O_z^m , O_w^m) = \beta$. Hence the limit in~\eqref{eqn-dist-def} exists and equals $\beta$.
\end{proof}

We have the following trivial consequence of Lemma~\ref{lem-ssl-across-around} which allows us to bound $D_h$-distances instead of just $\rng D_h$-``distances". 

\begin{lem} \label{lem-lim-across-around}
Almost surely, the following is true. 
\begin{enumerate}[(i)] 
\item Let $z,w\in\BB C$ be distinct and let $O_0,\dots,O_k$ be disjoint rational circles such that $O_j$ disconnects $O_{j-1}$ from $O_{j+1}$ for each $j\in[1,k-1]_{\BB Z}$, $O_0$ disconnects $z$ from $O_1$, and $O_k$ disconnects $w$ from $O_{k-1}$. Then $ D_h(z,w) \geq \sum_{j=1}^k \rng D_h(O_j , O_{j-1})$.  \label{item-lim-across}
\item Let $z,w,z',w' \in\BB C $ and suppose that $A$ is a rational annulus which does not contain any of $z,w,z',w'$. Suppose also that $z,z'$ lie in different connected components of $ \BB C\setminus A $ and $w,w'$ lie in different connected components of $ \BB C\setminus A $. Then $ D_h(z,w) \leq   D_h(z,z') +   D_h(w,w') + \rng D_h(\text{around $A$})$. \label{item-lim-around}
\end{enumerate}
\end{lem}
\begin{proof}
This follows by applying Lemma~\ref{lem-ssl-across-around} with some of the rational circles equal to the circles $O_z^m$ and $O_w^m$ in the definition~\eqref{eqn-dist-def} of $D_h$, then taking a limit as $m\rta\infty$. 
\end{proof}

\subsection{Lower semicontinuity}
\label{sec-lower-semicont}

We now check assertion~\ref{item-tight} of Theorem~\ref{thm-lfpp-tight}. 

\begin{prop} \label{prop-lower-semicont}
In the coupling defined in Section~\ref{sec-ssl}, we have $\frk a_\ep^{-1} D_h^\ep \rta D_h$ a.s.\ as $\mcl E\ni \ep \rta 0$ w.r.t.\ the topology on lower semicontinuous functions defined in Section~\ref{sec-main-results}. In particular, $D_h$ is lower semicontinuous.
\end{prop}

To prove Proposition~\ref{prop-lower-semicont}, we  will check the two conditions for convergence of lower semicontinuous functions in terms of sequences of points. 
The following lemma corresponds to condition~\eqref{item-semicont-liminf}. 

\begin{lem} \label{lem-dist-seq}
Almost surely, the following is true. Let $z,w\in \BB C$ and let $\{z^\ep\}_{\ep \in \mcl E}$ and $\{w^\ep\}_{\ep\in\mcl E}$ be such that $z^\ep\rta z$ and $w^\ep\rta w$. Then
\eqb \label{eqn-dist-seq}
D_h(z,w) \leq \liminf_{\mcl E \ni \ep\rta 0} \frk a_\ep^{-1} D_h^\ep(z^\ep,w^\ep) . 
\eqe
\end{lem}
\begin{proof}
For $m\in\BB N$, let $O^m_z$ be a rational circle surrounding $ B_{2^{-m-1}}(z)$ with radius $2^{-m}$. Similarly define $O^m_w$. 
For small enough $\ep \in\mcl E $ we have $z^\ep \in B_{2^{-m-1}}(z)$ and $w^\ep \in B_{2^{-m-1}}(w)$, in which case $O^m_z$ (resp.\ $O^m_w$) surrounds $z^\ep$ (resp.\ $w^\ep$). Therefore, 
\eqb
\frk a_\ep^{-1} D_h^\ep(z^\ep,w^\ep) \geq \frk a_\ep^{-1} D_h^\ep\left( O_z^m , O_w^m\right) .
\eqe
Taking the liminf of both sides gives
\eqb \label{eqn-dist-seq0}
\liminf_{\mcl E \ni \ep\rta 0} \frk a_\ep^{-1} D_h^\ep(z^\ep,w^\ep) \geq \rng D_h\left( O_z^m , O_w^m\right) , \quad\forall m\in\BB N .
\eqe
By~\eqref{eqn-dist-def}, the right side of~\eqref{eqn-dist-seq} converges to $D_h(z,w)$ as $m\rta\infty$. Hence~\eqref{eqn-dist-seq0} implies~\eqref{eqn-dist-seq}.
\end{proof}

The next lemma corresponds to the other condition~\eqref{item-semicont-lim} needed for convergence.

\begin{lem} \label{lem-dist-sup}
Almost surely, the following is true. For each $z,w\in\BB C$, there exists sequences $\{z^\ep\}_{\ep \in \mcl E}$ and $\{w^\ep\}_{\ep\in\mcl E}$ such that $z^\ep\rta z$, $w^\ep\rta w$, and
\eqb \label{eqn-dist-sup}
D_h(z,w) =  \lim_{\mcl E \ni \ep\rta 0} \frk a_\ep^{-1} D_h^\ep(z^\ep,w^\ep) . 
\eqe
\end{lem}
\begin{proof}
Let $\{O_z^m\}_{m\in\BB N}$ and $\{O_w^m\}_{m\in\BB N}$ be sequences of rational circles nesting down to $z$ and $w$, respectively. 
For $\ep \in\mcl E$ and $m\in\BB N$, let $z^{\ep,m} \in O_z^m$ and $w^{\ep,m} \in O_w^m$ be such that 
\eqbn
\frk a_\ep^{-1}D_h^\ep(z^{\ep,m} , w^{\ep,m}) = \frk a_\ep^{-1}D_h^\ep(O_z^m , O_w^m ).
\eqen
Then for each fixed $m\in\BB N$, we have $\frk a_\ep^{-1}D_h^\ep(z^{\ep,m} , w^{\ep,m})  \rta \rng D_h(O_z^m , O_w^m)$ along $\mcl E$, so by~\eqref{eqn-dist-def} we have $\frk a_\ep^{-1}D_h^\ep(z^{\ep,m} , w^{\ep,m}) \rta D_h(z,w)$ as $\ep\rta 0$ and then $m\rta\infty$. 

We will now pass to a suitable ``diagonal subsequence" of $z^{\ep,m}$ in order to get a sequence satisfying~\eqref{eqn-dist-sup}. 
For $m\in\BB N$, choose $\ep_m \in\mcl E$ such that 
\eqb  \label{eqn-dist-seq-compare}
| \frk a_\ep^{-1}D_h^\ep(z^{\ep,m} , w^{\ep,m})  -  \rng D_h(O_z^m , O_w^m) | \leq 2^{-m} ,\quad \forall \ep \in\mcl E \quad \text{with} \quad \ep \leq \ep_m .
\eqe  
We can take $\ep_m$ to be strictly decreasing in $m$, so that $\ep_m \rta 0$ as $m\rta\infty$. 
For each $\ep \in [\ep_m , \ep_{m+1}) \cap\mcl E$, we define $z^\ep := z^{\ep,m}$, $w^\ep := w^{\ep,m}$, and $m(\ep) := m$. Since $m(\ep) \rta \infty$ as $\ep_m \rta 0$, it follows that the radii of $O_z^{m(\ep)}$ and $O_w^{m(\ep)}$ tend to zero as $\ep\rta 0$. Therefore, $z^\ep \rta z$, $w^\ep \rta w$, and the definition~\eqref{eqn-dist-def} of $D_h(z,w)$ implies that $\lim_{\ep\rta 0} \rng D_h(O_z^{m(\ep)} , O_w^{m(\ep)} ) = D_h(z,w)$. By~\eqref{eqn-dist-seq-compare}, 
\eqbn
| \frk a_\ep^{-1}D_h^\ep(z^\ep , w^\ep)  -  \rng D_h(O_z^{m(\ep)} , O_w^{m(\ep)} ) | \leq 2^{-m(\ep)} \rta 0 \quad \text{as $\mcl E\ni\ep\rta 0$}. 
\eqen
Therefore,~\eqref{eqn-dist-sup} holds.
\end{proof}

\begin{proof}[Proof of Proposition~\ref{prop-lower-semicont}]
Combine Lemmas~\ref{lem-dist-seq} and~\ref{lem-dist-sup}. 
\end{proof}

We note that Lemma~\ref{lem-dist-seq} immediately implies the following. 

\begin{lem} \label{lem-finite-ae}
For any fixed $z,w\in\BB C$, a.s.\ $D_h(z,w) < \infty$.
\end{lem}
\begin{proof}
By Proposition~\ref{prop-gff-pt-tight}, the random variables $\frk a_\ep^{-1} D_h^\ep(z,w)$ are tight. 
Consequently, it is a.s.\ the case that there is a random $C >0$ and a random subsequence $\mcl E' \subset \mcl E$ such that $\frk a_\ep^{-1} D_h^\ep(z,w) \leq C$ for each $\ep\in\mcl E'$. Hence a.s.\ $\liminf_{\ep\rta0 } \frk a_\ep^{-1} D_h^\ep(z,w) < \infty$. 
The lemma statement now follows from Lemma~\ref{lem-dist-seq}. 
\end{proof}

\subsection{Triangle inequality}
\label{sec-triangle}

Our next goal is to check the triangle inequality for $D_h$, and thereby establish that $D_h$ is a metric. 

\begin{prop} \label{prop-triangle-inequality}
Almost surely, the function $D_h$ satisfies the triangle inequality, i.e., for each $x,y,z\in \BB C$ we have
\eqb \label{eqn-tri}
D_h(x,z) \leq D_h(x,y) + D_h(y,z) .
\eqe
Hence $D_h$ is a metric on $\BB C$. 
\end{prop}

The proof of Proposition~\ref{prop-triangle-inequality} is more involved than one might initially expect, for the following reason. Suppose $\{O_x^m\}_{m\in\BB N}$, $\{O_y^m\}_{m\in\BB N}$, and $\{O_z^m\}_{m\in\BB N}$ be sequences of rational circles nesting down to $x,y,z$, respectively  (Definition~\ref{def-nest}), so that $\rng D_h(O_x^m , O_y^m) \rta D_h(x,y)$ and similarly for $(y,z)$ and $(x,z)$. It is \emph{not} necessarily the case that
\eqb
\rng D_h(O_x^m , O_z^m) \leq \rng D_h(O_x^m , O_y^m) + \rng D_h(O_y^m , O_z^m ) .
\eqe
The heuristic reason for this is that the points on $O_y^m$ at ``minimal $\rng D_h$-distance" from each of $O_x^m$ and $O_z^m$ are not necessarily the same. To deal with this difficulty, we need a way to ``join up" a ``path" from $O_x^m$ to $O_y^m$ and a ``path" from $O_y^m$ to $O_z^m$ into a ``path" from $O_x^m$ to $O_z^m$ (the reason for all of the quotations is that we do not know that ``$\rng D_h$-distances" come from an actual length metric). We will do this by showing that we can choose $O_y^m$ in such a way that there is an annulus $A_y^m$ with $O_y^m$ as its inner boundary such that $\rng D_h(\text{around $A_y^m$})$ is small, then using assertion~\eqref{item-ssl-around} of Lemma~\ref{lem-ssl-across-around}. See Figure~\ref{fig-triangle}. 

In order to ensure the existence of the annulus $A_y^m$, we will work with a certain special sequence of rational circles nesting down to $y$, which we construct in the following lemma. 

\begin{lem} \label{lem-good-circles}
There is a deterministic constant $C>0$ depending only on $\xi$ such that the following is true almost surely. For each $z\in\BB C$, there exists a sequence $\{O^m\}_{m\in\BB N}$ of rational circles nesting down to $z$ (Definition~\ref{def-nest}) with the following properties. For $m\in\BB N$, let $\wh O^m$ be the circle with the same center as $O^m$ and twice the radius and let $A^m$ be the rational annulus between $O^m$ and $\wh O^m$.
Then for each $m\in\BB N$, $O^m$ surrounds $\wh O^{m+1}$ and 
\eqb \label{eqn-good-circles}
\rng D_h\left(\text{around $A^m$}\right)  \leq C  \rng D_h\left( \text{across $A^m$}   \right) .
\eqe
\end{lem}

As discussed above, the reason why the control on $\rng D_h(\text{around $A^m$})$ from condition~\eqref{eqn-good-circles} is useful is that it allows us to ``link up paths from $O^m$ to points outside of $O^m$" via assertion~\eqref{item-ssl-around} of Lemma~\ref{lem-ssl-across-around}. 
Before we prove Lemma~\ref{lem-good-circles}, we record the following supplementary lemma which will often be useful when we apply Lemma~\ref{lem-good-circles}.

\begin{lem} \label{lem-good-circles-conv}
Almost surely, the following is true. 
Let $z\in\BB C$ and let $\{O^m\}_{m\in\BB N}$ be a sequence of rational circles nesting down to $z$ satisfying the conditions of Lemma~\ref{lem-good-circles}. Also let $A^m$ be the rational annuli as in Lemma~\ref{lem-good-circles}. If there exists $w\in\BB C\setminus\{z\}$ such that $D_h(z,w) < \infty$, then 
\eqb
\lim_{m\rta\infty} \rng D_h\left(\text{around $A^m$}\right) = \lim_{m\rta\infty} \rng D_h\left( \text{across $A^m$}   \right) = 0.
\eqe
\end{lem}
\begin{proof}
By~\eqref{eqn-good-circles} we only need to show that $\lim_{m\rta\infty} \rng D_h\left( \text{across $A^m$}   \right) = 0$.
Fix a point $w\in\BB C \setminus \{z\}$ such that $D_h(z,w)$ is finite. 
By possibly ignoring finitely many of the $A^m$'s we can assume without loss of generality that $w$ is not contained in or disconnected from $\infty$ by $A^1$. 
By Lemma~\ref{lem-good-circles} the annuli $A^m$ for $m\in\BB N$ are disjoint. 
Therefore, assertion~\eqref{item-ssl-across} together with the definition~\eqref{eqn-dist-def} of $D_h(z,w)$ implies that
\eqb \label{eqn-good-circle-sum}
D_h(z,w) \geq \sum_{m=1}^\infty \rng D_h(\text{across $A^m$}) .
\eqe
Since $D_h(z,w)$ is finite by hypothesis and $\rng D_h(\text{across $A^m$}) \geq 0$ for every $m\in\BB N$, the sum on the right side of~\eqref{eqn-good-circle-sum} can have at most finitely many terms larger than any $\ep > 0$. It follows that $\lim_{m\rta\infty} \rng D_h\left( \text{across $A^m$}   \right) = 0$, as required.
\end{proof}

We now turn our attention to the proof of Lemma~\ref{lem-good-circles}. We first prove an analogous statement for the LFPP metrics $D_h^\ep$ using the tightness of distances around and across Euclidean annuli (Proposition~\ref{prop-gff-pt-tight}) together with the near-independence of the restriction of the GFF to disjoint concentric annuli (Lemma 3.1 of~\cite{local-metrics}). 

For $C>1$, $u \in\BB C$, $k\in\BB N_0$, and $\ep > 0$, let
\eqb
E_k^\ep(u;C) := \left\{   D_h^\ep\left(\text{around $B_{2^{-k+1}}(u) \setminus B_{2^{-k}}(u)$}\right)  
\leq C  D_h^\ep\left( \text{across $B_{2^{-k+1}}(u) \setminus B_{2^{-k}}(u)$} \right)  \right\} .
\eqe
Define $\wh E_k^\ep(u;C)$ similarly but with the localized LFPP metric $\wh D_h^\ep$ of~\eqref{eqn-localized-lfpp} used in place of $D_h^\ep$. 
We now check that $E_k^\ep(u;C)$ occurs with high probability when $C$ is large. 

\begin{lem} \label{lem-good-annulus-prob}
For each $p\in (0,1)$, there exists $C>1$ (depending only on $\xi$) such that for each $u\in\BB C$ and each $k\in\BB N_0$, 
\eqb \label{eqn-good-annulus-prob}
\BB P\left[  E_k^\ep(u;C) \right] \geq p - o_\ep(1),\quad \forall \ep\in (0,1)
\eqe
where the $o_\ep(1)$ tends to zero as $\ep\rta 0$ at a rate depending only on $\xi , k , p$ (not on $u$).
The same holds with $\wh E_k^\ep(u;C)$ in place of $E_k^\ep(u;C)$. 
\end{lem}
\begin{proof}
By Proposition~\ref{prop-gff-pt-tight}, we know that the random variables $\frk a_\ep^{-1} D_h^\ep\left(\text{around $B_2(0) \setminus B_1(0)$}\right)$ and $ \left(\frk a_\ep^{-1} D_h^\ep\left(\text{across $B_2(0)\setminus B_1(0)$}\right) \right)^{-1}$ are tight. Consequently, we can find $C   >1$ such that
\eqb \label{eqn-good-annulus-prob0}
\BB P\left[ E_0^\ep(0;C) \right] \geq p ,\quad\forall \ep \in (0,1) .
\eqe

By~\eqref{eqn-gff-scale}, we have the scaling relation
\eqb \label{eqn-good-annulus-scale}
\left( D_h^{2^{ k} \ep} (z,w)  \right)_{z,w\in\BB C} \eqD \left( 2^k e^{-\xi h_{2^{-k}}(0)} D_h^\ep(2^{-k} z , 2^{-k} w) \right)_{z,w\in\BB C},
\eqe
where $h_{2^{-k}}(0)$ is the average of $h$ over $\bdy B_{2^{-k}}(0)$. Since scaling $D_h^\ep$ by a constant factor does not affect the occurrence of $E_k^\ep(u;C)$, we infer from~\eqref{eqn-good-annulus-prob0}, \eqref{eqn-good-annulus-scale}, and the translation invariance of the law of $h$ modulo additive constant that  
\eqb
\BB P\left[ E_k^\ep(u;C ) \right] \geq p  ,\quad\forall \ep \in (0,2^{-k}) .
\eqe
This gives~\eqref{eqn-good-annulus-prob}. 

We obtain~\eqref{eqn-good-annulus-prob} with $\wh E_k^\ep(u; C)$ in place of $E_k^\ep(u;C)$ by replacing $C$ by $2C$ and applying Lemma~\ref{lem-localized-approx} together with the translation invariance of the law of $\wh h_\ep^*$, modulo additive constant.  
\end{proof}

The following lemma is the main technical estimated needed for the proof of Lemma~\ref{lem-good-circles}.

\begin{lem} \label{lem-good-annulus-exist}
Fix a bounded open set $U\subset\BB C$. There exists $C>1$ depending only on $\xi$ such that for each $K \in \BB N$ and $\ep\in (0,1)$, it holds with probability $1 - O_K(2^{-  K}) - o_\ep(1)$ that the following is true (here the implicit constant in the $O_K(\cdot)$ depends only on $U,\xi$ and the rate of convergence of the $o_\ep(1)$ depends only on $K,U,\xi$). 
For each $u \in (2^{-K-4} \BB Z^2) \cap U$, there exists $k\in [K/2 , K]_{\BB Z}$ for which $E_k^\ep(u;C)$ occurs.
\end{lem}
\begin{proof}
Fix $p  \in (0,1)$, close to 1, which we will choose later in a universal manner. Let $C$ be as in Lemma~\ref{lem-good-annulus-prob} for this choice of $p$, so that  
\eqb
\BB P\left[  E_k^\ep(u;C) \right] \geq p - o_\ep(1),\quad\forall \ep \in (0,1) ,\quad\forall u\in\BB C,\quad\forall k \in \BB N_0 .
\eqe 
By the locality property~\eqref{eqn-localized-property} of $\wh h_\ep^*(u)$, the event $\wh E_k^\ep(u;C)$ for $k\in [0,K]_{\BB Z}$ is determined by the restriction of $h$ to $B_{2^{-k+1}+\ep^{1/2}}(u) \setminus B_{2^{-k}-\ep^{1/2}}(u)$. In fact, $\wh E_k^\ep(u;C)$ is determined by this restriction of $h$ viewed modulo additive constant since adding a constant to $h$ results in scaling $\wh D_h^\ep$ by a constant. In particular, if $\ep$ is smaller than $2^{-2K-2}$, then $\wh E_k^\ep(u;C)$ is determined by the restriction of $h$ to $B_{2^{-k+2}}(u) \setminus B_{2^{-k-1}}(u)$, viewed modulo additive constant.

By a basic near-independence estimate for the restrictions of the GFF to disjoint concentric annuli (see~\cite[Lemma 3.1, assertion 1]{local-metrics}), it follows that if $p$ is chosen to be sufficiently close to 1 (in a universal manner) and $\ep$ is sufficiently small (depending on $K$), then the following is true. For each $u\in\BB C$, 
\eqb
\BB P\left[ \text{$E_k^\ep(u;C)$ occurs for at least one $k\in [K/2,K]_{\BB Z}$} \right] \geq 1 -  O_K(2^{-3 K}) .
\eqe
By a union bound over the $O_K(2^{ 2K})$ points $u\in (2^{- K - 4} \BB Z^2) \cap U$, we now obtain the lemma statement with $\wh E_k^\ep(u;C)$ in place of $E_k^\ep(u;C)$. 
The statement for $E_k^\ep(u;C)$ (with a slightly larger value of $C$) follows from the statement for $\wh E_k^\ep(u;C)$ together with Lemma~\ref{lem-localized-approx}.
\end{proof}

\begin{lem} \label{lem-ssl-good-annulus}
Fix a bounded open set $U\subset\BB C$. There exists $C>1$ depending only on $\xi$ such that for each $K \in \BB N$, it holds with probability $1 - O_K(2^{-  K})  $ that the following is true. 
For each $u\in (2^{-2K} \BB Z^2) \cap U$, there exists $k\in [K/2 , K]_{\BB Z}$ for which
\eqb \label{eqn-ssl-good-annulus}
\rng D_h\left(\text{around $B_{2^{-k+1}}(u) \setminus B_{2^{-k}}(u)$}\right)  \leq C  \rng D_h\left( \bdy B_{2^{-k+1}}(u)  , \bdy B_{2^{-k}}(u)   \right) .
\eqe
\end{lem}
\begin{proof}
This follows by passing to the (subsequential) limit in Lemma~\ref{lem-good-annulus-exist}. 
\end{proof}

\begin{proof}[Proof of Lemma~\ref{lem-good-circles}]
Fix a large $R>1$. It suffices to prove that the lemma statement holds a.s.\ for each $z\in B_R(0)$. 
By Lemma~\ref{lem-ssl-good-annulus} and the Borel-Cantelli lemma, it is a.s.\ the case that for each large enough $K \in\BB N$ and each $u \in (2^{-K-4} \BB Z^2) \cap B_R(0)$ there exists $k = k(u,K) \in [K/2,K]_{\BB Z}$ such that 
\eqb \label{eqn-rational-good-annulus}
\rng D_h\left(\text{around $B_{2^{-k+1}}(u) \setminus B_{2^{-k}}(u) $}\right)  \leq C  \rng D_h\left( \text{across $B_{2^{-k+1}}(u) \setminus B_{2^{-k}}(u) $}    \right) .
\eqe
For a given $z\in B_R(0)$ and $K\in\BB N$ large enough that the preceding statement holds, let $u_z^K \in (2^{-K-4} \BB Z^2) \cap B_R(0)$ be chosen so that $z\in B_{2^{-K}}(u_z^K)$ and let $\wt A_z^K := B_{2^{-k+1}}(u_z^K) \setminus B_{2^{-k}}(u_z^K) $. Then each $\wt A_z^K$ is an annulus surrounding $z$ which satisfies the bound~\eqref{eqn-good-circles} (with $\wt A_z^K$ in place of $A^m$) and the radii of the $\wt A_z^K$'s tends to zero as $K\rta\infty$. 
If we let $\{A^m\}_{m\in\BB N}$ be a sufficiently sparse subsequence of the $A_z^K$'s, then the $A^m$'s are disjoint and $A^m$ disconnects $A^{m+1}$ from $\infty$ for each $m\in\BB N$. Thus the lemma statement is true with $O^m$ equal to the inner boundary of $A^m$.
\end{proof}

\begin{figure}[t!]
 \begin{center}
\includegraphics[scale=1]{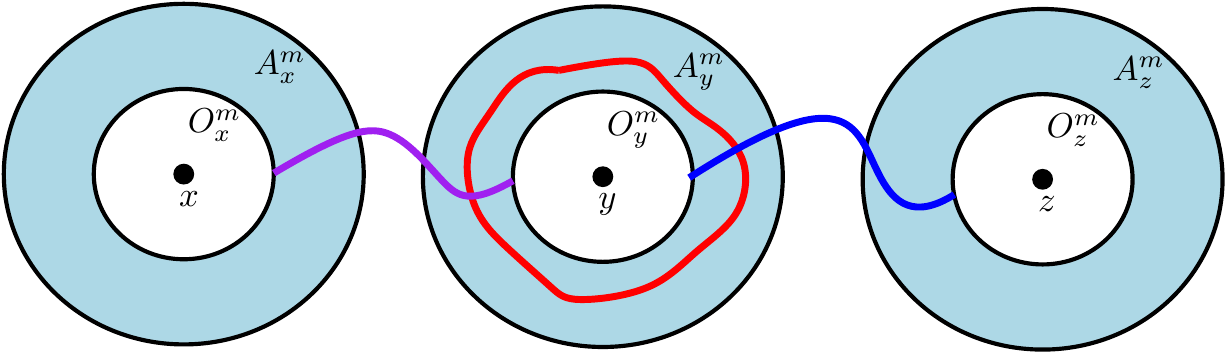}
\vspace{-0.01\textheight}
\caption{Illustration of the proof of Proposition~\ref{prop-triangle-inequality}. We use the short (by Lemma~\ref{lem-good-circles-conv}) red ``path" around $A_y^m$ to join together ``paths" from $O_x^m$ to $O_y^m$ and from $O_y^m$ to $O_z^m$ into a ``path" from $O_x^m$ to $O_z^m$. Note that the paths in the figure are meant to illustrate the intuition behind assertion~\eqref{item-ssl-around} rather than to represent literal paths, since we do not know that $\rng D_h$-distances are the same as $D_h$-distances or that $\rng D_h$ is an actual metric. 
}\label{fig-triangle}
\end{center}
\vspace{-1em}
\end{figure}

\begin{proof}[Proof of Proposition~\ref{prop-triangle-inequality}]
Since $D_h(w,w) =0$ for any $w\in\BB C$, we can assume without loss of generality that $x,y,z$ are distinct. 
We can also assume without loss of generality that $D_h(x,y)$ and $D_h(y,z)$ are both finite (otherwise,~\eqref{eqn-tri} holds trivially). 

Let $\{O_x^m\}_{m\in\BB N}$, $\{O_y^m\}_{m\in\BB N}$, and $\{O_z^m\}_{m\in\BB N}$ be sequences of rational circles nesting down to $x,y,z$, respectively, which satisfy the conditions of Lemma~\ref{lem-good-circles}. Also let $A_x^m , A_y^m , A_z^m$ be the corresponding annuli as in Lemma~\ref{lem-good-circles}. 
By possibly dropping a finite number of $m$-values, we can assume without loss of generality that $O_x^1$, $O_y^1$, and $O_z^1$ are disjoint. 
By assertion~\eqref{item-ssl-around} of Lemma~\ref{lem-ssl-across-around} applied with $O_1 = O_x^m$, $O_2 = O_z^m$, $O_1' = O_2' = O_y^m$, and $A  =A^m$, we have
\eqb \label{eqn-tri0}
\rng D_h(O_x^m , O_z^m ) \leq  \rng D_h(O_x^m , O_y^m ) +  \rng D_h(O_y^m , O_z^m )  + \rng D_h(\text{around $A_y^m$}) ; 
\eqe
see Figure~\ref{fig-triangle} for an illustration.  
By Lemma~\ref{lem-good-circles-conv} and our assumption that $D_h(x,y)  < \infty$, we have $\lim_{m\rta\infty} D_h(\text{around $A_y^m$}) = 0$. 
By the definition~\eqref{eqn-dist-def} of $D_h$, the triangle inequality now follows by sending $m\rta\infty$ in~\eqref{eqn-tri0}. 
Since $D_h$ is positive definite and symmetric by definition (see the discussion just after Lemma~\ref{lem-dist-lim}) it follows that $D_h$ is a metric. 
\end{proof}

\subsection{Consistency at rational points}
\label{sec-rational}

Recall that points in $\BB Q^2$ are considered as rational circles of radius 0. 
In particular, $\rng D_h(u,v)$ is defined for each $u,v\in\BB Q^2$. 
The goal of this subsection is to establish the following proposition, which implies assertion~\ref{item-rational} of Theorem~\ref{thm-lfpp-tight}. 

\begin{prop} \label{prop-rational-agree}
Almost surely, for each $u,v\in\BB Q^2$, we have $D_h(u,v) = \rng D_h(u,v)$. 
\end{prop}

The proof of Proposition~\ref{prop-rational-agree} is based on Lemma~\ref{lem-good-circles} together with the following lemma, which is a consequence of the explicit bounds for point-to-point distances from Lemma~\ref{lem-pt-dist}.
 
\begin{lem} \label{lem-gff-pt-conv}
For each fixed $z \in\BB C$, we have $\frk a_\ep^{-1} D_h^\ep(z,w) \rta 0$ in law as $\ep\rta 0$ and then $w\rta z$. 
\end{lem}
\begin{proof}
By Lemma~\ref{lem-pt-dist} and the fact that the random variables $\Phi_{0,k}(z)$ are Gaussian with variance $k \log 2$, it is easily seen that $\lambda_n^{-1} D_{0,n}(z,w) \rta 0$ in law as $n\rta\infty$ and then $w\rta z$. The lemma statement follows by combining this with Lemma~\ref{lem-gff-compare}. 
\end{proof}

As a consequence of Lemma~\ref{lem-gff-pt-conv}, we have the following. 

\begin{lem} \label{lem-rational-conv}
Almost surely, for each $u\in \BB Q^2$ and each sequence $\{O^m\}_{m\in\BB N}$ of rational circles nesting down to $u$, we have $\lim_{m\rta\infty} \rng D_h(u,O^m) = 0$. 
\end{lem}
\begin{proof}
By countability it suffices to prove the lemma for a fixed $u\in\BB Q^2$. 
Let $q\in\BB Q^2 \setminus \{u\}$. Since $\rng D_h(u,q) = \lim_{\ep\rta 0} \frk a_\ep^{-1} D_h^\ep(u,q)$, it follows from Lemma~\ref{lem-gff-pt-conv} that $\rng D_h(u,q) \rta 0$ in law (hence also in probability) as $q\rta u$.  
If $q $ is not surrounded by $O_u^m$, then assertion~\eqref{item-ssl-across} of Lemma~\ref{lem-ssl-across-around} (applied with $O_1 = \{u\}$, $O_2 = O^m$, and $O_3 = \{q\}$) implies that $\rng D_h(u, O^m) \leq \rng D_h(u , q)$. 
Since the radius of $O^m$ tends to zero as $m\rta\infty$, we infer that $\rng D_h(u,O^m) \rta 0$ in probability as $m\rta\infty$. Since $\rng D_h(u , O^m)$ is decreasing in $m$ (assertion~\eqref{item-ssl-across} of Lemma~\ref{lem-ssl-across-around}), it follows that in fact $\rng D_h(u,O^m) \rta 0$ a.s.\ as $m\rta\infty$.  
\end{proof}

\begin{figure}[t!]
 \begin{center}
\includegraphics[scale=1]{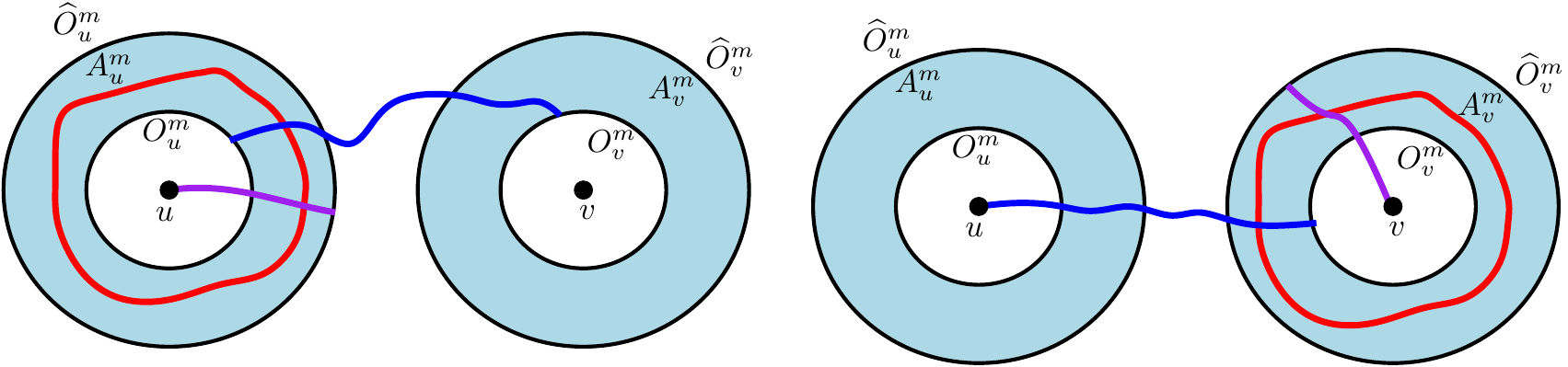}
\vspace{-0.01\textheight}
\caption{Illustration of the two applications of assertion~\eqref{item-ssl-around} in the proof of Proposition~\ref{prop-rational-agree}. In each case, the length of the purple and red ``paths" are each at most $\delta$. Note that the paths in the figure are meant to illustrate the intuition behind assertion~\eqref{item-ssl-around} rather than to represent literal paths, since we do not know that $\rng D_h$-distances are the same as $D_h$-distances or that $\rng D_h$ is an actual metric. 
}\label{fig-rational-agree}
\end{center}
\vspace{-1em}
\end{figure}

\begin{proof}[Proof of Proposition~\ref{prop-rational-agree}]
If $u=v$ then obviously $D_h(u,v) = \rng D_h(u,v) =0$. By countability it therefore suffices to prove the lemma for a fixed choice of distinct points $u,v\in\BB Q^2$ with $u\not= v$. 
Let $\{O_u^m\}_{m\in\BB N}$ and $\{O_v^m\}_{m\in\BB N}$ be sequences of rational circles nesting down to $u$ and $v$, respectively, which satisfy the conditions of Lemma~\ref{lem-good-circles}. Also let $\wh O_u^m$, $\wh O_v^m$ and $A_u^m , A_v^m $ be the corresponding outer circles and rational annuli as in Lemma~\ref{lem-good-circles}. 
Then $D_h(u,v) = \lim_{m\rta\infty} \rng D_h(O_u^m , O_v^m)$. 
It is obvious from assertion~\eqref{item-ssl-around} of Lemma~\ref{lem-ssl-across-around} that $  \rng D_h(O_u^m, O_v^m) \leq \rng D_h(u,v) $ for each $m\in\BB N$, whence $ D_h(u,v) \leq \rng D_h(u,v)$. We only need to prove that $ \rng D_h(u,v) \leq   D_h(u,v)$.

By Lemmas~\ref{lem-good-circles-conv} and~\ref{lem-rational-conv}, we find that for each $\delta>0$, it holds for large enough $m\in\BB N$ that
\eqb \label{eqn-rational-small}
 \rng D_h(u,\wh O_u^m) \leq \delta ,\quad 
 \rng D_h\left(\text{around $A_u^m$}\right) \leq \delta ,
\eqe
and the same is true with $v$ in place of $u$. 

See Figure~\ref{fig-rational-agree} for an illustration of the next steps of the proof. 
By assertion~\eqref{item-ssl-around} of Lemma~\ref{lem-ssl-across-around} (applied with $O_1  = \{u\}$, $O_1' =  \wh O_u^m $, $O_2 =   O_v^m$, $O_2' =  O_u^m$, and $A = A_u^m$) together with~\eqref{eqn-rational-small},  
\eqb \label{eqn-rational-around1}
\rng D_h(u,O_v^m) \leq   \rng D_h(O_u^m , O_v^m) +  \rng D_h(u,\wh O_u^m)   +   \rng D_h\left(\text{around $A_u^m$}\right) \leq  \rng D_h(O_u^m , O_v^m) + 2\delta .
\eqe 
By another application of assertion~\eqref{item-ssl-around} of Lemma~\ref{lem-ssl-across-around} (with $O_1 =  \{u\}$, $O_1' =  O_v^m $, $O_2 = \{v\}$, $O_2' = \wh O_v^m$, and $A = A_v^m$) together with~\eqref{eqn-rational-small},  
\eqb \label{eqn-rational-around2} 
 \rng D_h(u,v) \leq  \rng D_h(u , O_v^m) +  \rng D_h(v , \wh O_v^m)   +   \rng D_h\left(\text{around $A_v^m$}\right) \leq  \rng D_h(u , O_v^m) + 2\delta .
\eqe
Applying~\eqref{eqn-rational-around1} to bound the right side of~\eqref{eqn-rational-around2} gives $\rng D_h(u,v) \leq  D_h(O_u^m , O_v^m) + 4\delta$. Sending $m\rta\infty$ gives $\rng D_h(u,v) \leq D_h(u,v) + 4\delta$. Since $\delta>0$ is arbitrary this concludes the proof.
\end{proof}

\subsection{H\"older continuity and thick points}
\label{sec-holder}

In this subsection we will prove several quantitative properties of $D_h$ which are part of Theorems~\ref{thm-lfpp-tight} and~\ref{thm-ssl-properties}. 
We start with the following proposition, which is assertion~\ref{item-bounded} of Theorem~\ref{thm-ssl-properties}.  

\begin{prop} \label{prop-infty-dist-ssl}
Almost surely, for every compact set $K\subset\BB C$ we have $\lim_{R \rta\infty} D_h(K,\bdy B_R(0)) = \infty$. 
In particular, every $D_h$-bounded subset of $\BB C$ is also Euclidean bounded.
\end{prop}
\begin{proof}
Choose a rational $r > 0$ such that $K \subset B_r(0)$. By Lemma~\ref{lem-infty-dist} and~\eqref{eqn-ssl-circle}, for any $T>1$,  
\eqb \label{eqn-infty-dist-ssl}
\lim_{\BB Q\ni R \rta\infty} \BB P\left[\rng D_h \left( \bdy B_r(0) , \bdy B_{R/2}(0) \right) \geq T \right] = 1.
\eqe
By assertion~\eqref{item-lim-across} of Lemma~\ref{lem-lim-across-around} (applied with $z\in K$, $w \in \bdy B_{ R}(0)$, $O_0 = \bdy B_r(0)$, and $O_2 = \bdy B_{R/2}(0)$), we have $D_h(K, \bdy B_R(0)) \geq  \rng D_h \left( \bdy B_r(0) , \bdy B_{R/2}(0) \right)$. 
The proposition statement therefore follows from~\eqref{eqn-infty-dist-ssl}. 
\end{proof}

The next proposition is assertion~\ref{item-holder} of Theorem~\ref{thm-ssl-properties}. 

\begin{prop} \label{prop-holder}
Almost surely, the identity map from $\BB C$, equipped with metric $D_h$ to $\BB C$, equipped with the Euclidean metric is locally H\"older continuous with any exponent less than $[\xi(Q+2)]^{-1}$. 
\end{prop}

For the proof of Proposition~\ref{prop-holder} we need yet another variant of Proposition~\ref{prop-perc-estimate}. 

\begin{lem} \label{lem-perc-ssl}
Let $A\subset\BB C$ be a rational annulus and define the constants $\frk c_r$ for $r\in \BB Q\cap (0,\infty)$ as in~\eqref{eqn-ssl-constant}.  
There are constants $c_0,c_1 >0$ depending only on $A,\xi$ such that for each rational $r > 0$, each $q\in\BB Q^2$, and each $T > 3$, 
\eqb \label{eqn-perc-ssl-lower}
\BB P\left[ \rng D_h\left(\text{across $r A + q$}\right) < T^{-1} \frk c_r e^{\xi h_r(q)}  \right] \leq c_0 e^{-c_1 (\log T)^2}  
\eqe
and
\eqb \label{eqn-perc-ssl-upper}
\BB P\left[ \rng D_h\left(\text{around $r A + q$}\right)   > T \frk c_r e^{\xi h_r(q)} \right] \leq c_0 e^{-c_1 (\log T)^2 / \log\log T} .
\eqe 
\end{lem} 
\begin{proof}
We first note that by the translation invariance of the law of $h$, modulo additive constant, $e^{-\xi h_r(q)} \rng D_h\left(\text{across $r A + q$}\right)  \eqD  e^{-\xi h_r(0)} \rng D_h\left(\text{across $r A $}\right)  $, and the same is true for $\rng D_h(\text{around $r A + q$})$. Therefore, it suffices to prove the lemma in the case when $q = 0$. 
In this case, the lemma follows from Lemma~\ref{lem-perc-gff} upon sending $\ep\rta 0$ along $\mcl E$, using~\eqref{eqn-ssl-h} to deal with the convergence of circle averages, and using~\eqref{eqn-ssl-constant} to deal with the convergence of $r \frk a_{r/\ep} /\frk a_\ep$. 
\end{proof}

\begin{proof}[Proof of Proposition~\ref{prop-holder}]
Fix $R >1$. We will show that a.s.\ there is a random constant $c \in (0,1)$ such that
\eqb \label{eqn-holder}
D_h(z,w) \geq c |z-w|^{\xi (Q+2)  +\zeta} ,\quad\forall z,w\in B_R(0) .
\eqe
 
For $k \in\BB N$ and $u\in (2^{-k-4} \BB Z^2) \cap B_R(0)$, define the rational annulus 
\eqb
A_k(u) := B_{2^{-k }}(u) \setminus B_{2^{-k-1}}(u) . 
\eqe
By Lemma~\ref{lem-perc-ssl} (applied with $T = 2^{\zeta k}$ and $r = 2^{-k}$), there are constants $c_0,c_1 > 0$ depending only on $\xi$ and $\zeta$ such that for each such $k$ and $u$, 
\eqb
\BB P\left[ \rng D_h\left(\text{across $A_k(u)$}\right) <  2^{-\zeta k} \frk c_{2^{-k}}  e^{\xi h_{2^{-k}}(0)}  \right] \leq c_0 e^{-c_1 k^2} . 
\eqe
By taking a union bound over all $u\in (2^{-k-4} \BB Z^2) \cap B_R(0)$ and possibly adjusting $c_0,c_1$ to absorb a factor of $O_k(4^k)$, we get that for each $k\in\BB N$, it holds with probability at least $1-c_0 e^{-c_1 k^2}$ that
\eqb \label{eqn-holder-annuli}
\rng D_h\left(\text{across $A_k(u)$}\right) \geq  2^{-\zeta k} \frk c_{2^{-k}}  e^{\xi h_{2^{-k}}(0)} ,\quad\forall u \in (2^{-k-4} \BB Z^2) \cap B_R(0) .
\eqe

Each of the random variables $h_{2^{-k}}(u)$ is centered Gaussian with variance $k \log 2 + O_k(1)$. Therefore, the Gaussian tail bound and a union bound over all $u\in (2^{-k-2} \BB Z^2) \cap B_R(0)$ shows that with probability at least $1 - c_0 2^{-[ (2+\zeta)^2/2  - 2) k]}$, 
\eqb \label{eqn-holder-field}
|h_{2^{-k}}(u)|  \leq (2+\zeta) k \log 2  ,\quad\forall u \in (2^{-k-4} \BB Z^2) \cap B_R(0)  .
\eqe

By the Borel-Cantelli lemma, a.s.\ there exists $K\in\BB N$ such that~\eqref{eqn-holder-annuli} and~\eqref{eqn-holder-field} both hold for all $k\geq K$. 
Henceforth assume that this is the case. 
For distinct $z,w\in B_R(0)$ with $|z-w| \leq 2^{-K-2}$, choose $k = k(z,w) \geq K$ such that $2^{-k+1 } \leq |z-w| \leq 2^{-k+2}$. There is a $u \in (2^{-k-4} \BB Z^2) \cap B_R(0)$ such that $z$ and $w$ lie in different connected components of $\BB C\setminus A_k(u)$. By assertion~\eqref{item-lim-across} of Lemma~\ref{lem-lim-across-around}, followed by~\eqref{eqn-holder-annuli} and~\eqref{eqn-holder-field}, 
\eqb \label{eqn-holder-pt}
D_h(z,w) 
\geq \rng D_h\left(\text{across $A_k(u)$}\right)
\geq 2^{ -(2\xi + (1+\xi)\zeta) k } \frk c_{2^{-k}}  .
\eqe
Since $\frk c_{2^{-k}} = 2^{-\xi Q k  + o_k(k)}$ by~\eqref{eqn-ssl-constant-lim}, we obtain from~\eqref{eqn-holder-pt} that $D_h(z,w) \geq  2^{-[ \xi(Q+2) + (2+\xi)\zeta]k}$, where $c_0 > 0$ is a deterministic constant depending only on $\xi,\zeta$ (note that we have absorbed the $2^{o_k(k)}$ into an extra factor of $2^{\zeta k}$). By replacing $\zeta$ by $\zeta/(2+\xi)$ and recalling our choice of $k = k(z,w)$, we get that a.s.\ for each large enough $K \in\BB N$, 
\eqb
D_h(z,w) \geq c_0 |z-w|^{\xi(Q+2) + \zeta} ,\quad\forall z,w\in B_R(0) \quad \text{with $|z-w| \leq 2^{-K-1}$} .
\eqe
We now obtain~\eqref{eqn-holder} by replacing $c_0$ by a smaller, $K$-dependent constant $c$ to deal with the case when $|z-w| > 2^{-K-1}$. 
\end{proof}

Finally, we prove assertion~\ref{item-thick-pt} of Theorem~\ref{thm-ssl-properties}.

\begin{prop} \label{prop-thick-pt}
Let $\alpha  > Q$. Almost surely, for each $\alpha$-thick point $z$ of $h$, we have $D_h(z,w) = \infty$ for every $w\in\BB C\setminus z$. 
\end{prop}
\begin{proof}
Fix $R > 1$. We will prove that the condition in the lemma statement holds a.s.\ for each $\alpha$-thick point in $B_R(0)$. 
By Lemma~\ref{lem-perc-ssl} (applied with $T = 2^{k^{2/3} }$) and a union bound over all $u \in ( 2^{-2k} \BB Z^2) \cap B_R(0)$, there are constants $c_0,c_1 >0$ depending only on $R ,\zeta, \xi$ such that with probability at least $1-c_0 e^{-c_1 k^{4/3}}$, 
\eqb \label{eqn-thick-union}
\rng D_h\left( \text{across $B_{2^{-k}}(u) \setminus B_{2^{-k-1}}(u)$} \right) \geq 2^{-k^{2/3} } \frk c_{2^{-k}}  e^{\xi h_{2^{-k}}(0)} ,
\quad\forall u \in (2^{-2k} \BB Z^2) \cap B_R(0) .
\eqe
By the Borel-Cantelli lemma, a.s.~\eqref{eqn-thick-union} holds for all large enough $k\in\BB N$. 
By a basic continuity estimate for the circle average process (see, e.g.,~\cite[Proposition 2.1]{hmp-thick-pts}), a.s.\ there exists a random constant $C > 1$ such that for all $z,w\in B_R(0)$ and all $r \in (0,1)$, 
\eqb \label{eqn-thick-cont}
|h_r(z) - h_r(w)| \leq C r^{-1/2} |z-w|^{(1-\zeta)/2} .
\eqe

Henceforth assume that~\eqref{eqn-thick-union} holds for large enough $k\in\BB N$ and~\eqref{eqn-thick-cont} holds.
Let $z\in B_R(0)$ be an $\alpha$-thick point, so that $h_{2^{-k}}(z) \geq (\alpha + o_k(1)) k \log 2$ as $k\rta\infty$. 
For $k\in\BB N$, choose $u_k \in ( 2^{-2k} \BB Z^2) \cap B_R(0)$ such that $z\in B_{2^{-2k+1}}(u_k)$. By~\eqref{eqn-thick-cont},
\eqb
|h_{2^{-k}}(z) - h_{2^{-k}}(u_k)| \leq C 2^{-k/2} (2^{-2k+1})^{(1-\zeta)/2}  = o_k(1) ,\quad \text{as $k\rta\infty$}. 
\eqe
Therefore, $h_{2^{-k}}(u_k) \geq (\alpha + o_k(1)) k \log 2$. By~\eqref{eqn-thick-union},
\eqb \label{eqn-thick-lower}
\rng D_h\left( \text{across $B_{2^{-k}}(u_k) \setminus B_{2^{-k-1}}(u_k)$} \right)
\geq  \frk c_{2^{-k}} 2^{\xi \alpha k + o_k(k)} .
\eqe
By~\eqref{eqn-ssl-constant-lim}, $\frk c_{2^{-k}} = 2^{\xi Q k + o_k(k)}$, so the right side of~\eqref{eqn-thick-lower} is at least $2^{\xi (\alpha - Q) k + o_k(k)}$, which tends to $\infty$ as $k\rta\infty$. 
For any $w\in \BB C\setminus \{z\}$, the annulus $B_{2^{-k}}(u_k) \setminus B_{2^{-k-1}}(u_k)$ disconnects $z$ from $w$ for large enough $k$. Therefore, assertion~\eqref{item-lim-across} of Lemma~\ref{lem-lim-across-around} implies that 
\eqbn
D_h(z,w) \geq \limsup_{k\rta\infty} \rng D_h\left( \text{across $B_{2^{-k}}(u_k) \setminus B_{2^{-k-1}}(u_k)$} \right) = \infty .
\eqen 
\end{proof}

\subsection{Singular points, completeness, and geodesics}
\label{sec-complete}

We know from Proposition~\ref{prop-thick-pt} that $D_h$ takes on infinite values when $\xi > \xi_{\op{crit}}$. 
We now provide additional detail on which pairs of points can lie at infinite distance from others.  
As in Theorem~\ref{thm-ssl-properties}, we say that $z\in\BB C$ is a \emph{singular point} for $D_h$ if $D_h(z,w) = \infty$ for every $w\in\BB C\setminus \{z\}$. 

By Lemma~\ref{lem-finite-ae}, for a fixed $z\in\BB C$, a.s.\ $z$ is not a singular point for $D_h$. 
In particular, the set of singular points a.s.\ has Lebesgue measure zero. 
On the other hand, by Proposition~\ref{prop-thick-pt}, a.s.\ each $\alpha$-thick point of $h$ for $\alpha > Q$ is a singular point for $D_h$. 
In particular, if $Q < 2$ (equivalently, $\xi  > \xi_{\op{crit}}$) then a.s.\ the set of singular points is uncountable and dense.

We will now show that two points can be at infinite $D_h$-distance from each other only if at least one of them is a singular point. 
In particular, $D_h$ is a finite metric on $\BB C \setminus \{\text{singular points}\}$ (we already know that $D_h$ is a metric from Proposition~\ref{prop-triangle-inequality}).

\begin{lem} \label{lem-finite-dist}
Almost surely, the following is true. Suppose $z,z',w,w' \in\BB C$ such that $z\not=z'$, $w \not= w'$,  $D_h(z,z') <\infty$, and $D_h(w,w') < \infty$. Then $D_h(z,w) < \infty$.
In particular, if $z,w\in\BB C$ such that $D_h(z,w) = \infty$, then either $z$ or $w$ is a singular point for $D_h$. 
\end{lem}
\begin{proof}
See Figure~\ref{fig-singular-pt} for an illustration. 
Let $A_z$ be a rational annulus such that $z$ (resp.\ $z'$) lies in the bounded (resp.\ unbounded) connected components of $\BB C\setminus A_z$ and let $O_z$ be the inner boundary of $A_z$.  
Similarly define $A_w$ and $O_w$. 

By Lemma~\ref{lem-dist-sup}, we can find a sequences of points $z_\ep, z_\ep'$ for $\ep\in\mcl E$ such that $z_\ep \rta z$, $z_\ep' \rta z'$, and $\frk a_\ep^{-1} D_h(z_\ep  ,z_\ep') \rta D_h(z,z')$. For small enough $\ep   \in\mcl E$, $z_\ep$ (resp.\ $z_\ep'$) lies in the bounded (resp.\ unbounded) connected components of $\BB C\setminus A_z$. 
Using that $D_h^\ep$ is a length metric, we can concatenate paths (see Figure~\ref{fig-singular-pt}) to get
\eqb \label{eqn-finite-dist0-ep}
D_h^\ep(z_\ep , O_w)
\leq D_h^\ep(z_\ep , z_\ep') + D_h^\ep(O_z , O_w)  + D_h^\ep(\text{around $A_z$}). 
\eqe
Dividing by $\frk a_\ep $ in~\eqref{eqn-finite-dist0-ep}, taking the liminf of both sides, and applying Lemma~\ref{lem-dist-seq} on the left gives
\eqb \label{eqn-finite-dist0}
D_h(z, O_w) 
\leq  D_h(z,z') + \rng D_h(O_z,O_w) + \rng D_h(\text{around $A_z$})
< \infty .
\eqe

We will now use a similar argument to prove an upper bound for $D_h(z,w)$ in terms of $D_h(z,O_w)$.  
To this end, let $x\in O_w$ be chosen so that $D_h(z,x) \leq D_h(z,O_w) + 1$. 
We use Lemma~\ref{lem-dist-sup} to choose sequences of points $w_\ep \rta w$, $w_\ep' \rta w'$, $z_\ep \rta z$, and $x_\ep \rta x$ such that
\eqbn
\lim_{\mcl E \ni \ep\rta 0} D_h(w_\ep , w_\ep') = D_h(w,w') \quad\text{and} \quad \lim_{\mcl E \ni \ep\rta 0} D_h(z_\ep , x_\ep) = D_h(z,w) \leq D_h(z,O_w) + \delta .
\eqen
Using that $D_h^\ep$ is a length metric, we can concatenate paths (see Figure~\ref{fig-singular-pt}) to get
\eqb \label{eqn-finite-dist1-ep}
D_h^\ep(z_\ep , w_\ep) \leq D_h^\ep(w_\ep , w_\ep')  + D_h^\ep(\text{around $A_w$}) + D_h^\ep(w_\ep,x_\ep) .
\eqe
Dividing by $\frk a_\ep $ in~\eqref{eqn-finite-dist1-ep}, taking the liminf of both sides, and applying Lemma~\ref{lem-dist-seq} on the left gives 
\eqb \label{eqn-finite-dist1}
D_h(z,w) 
\leq  D_h(w,w') + \rng D_h(\text{around $A_w$}) + D_h(z,O_w)  + 1,
\eqe
which is finite due to~\eqref{eqn-finite-dist0}. 

To get the last assertion of the lemma, we note that if neither $z$ nor $w$ is a singular point then there exists $z'\not=z$ and $w'\not=w$ such that $D_h(z,z') < \infty$ and $D_h(w,w') < \infty$, which implies that $D_h(z,w) < \infty$ by the first assertion. 
\end{proof}

\begin{figure}[t!]
 \begin{center}
\includegraphics[scale=1]{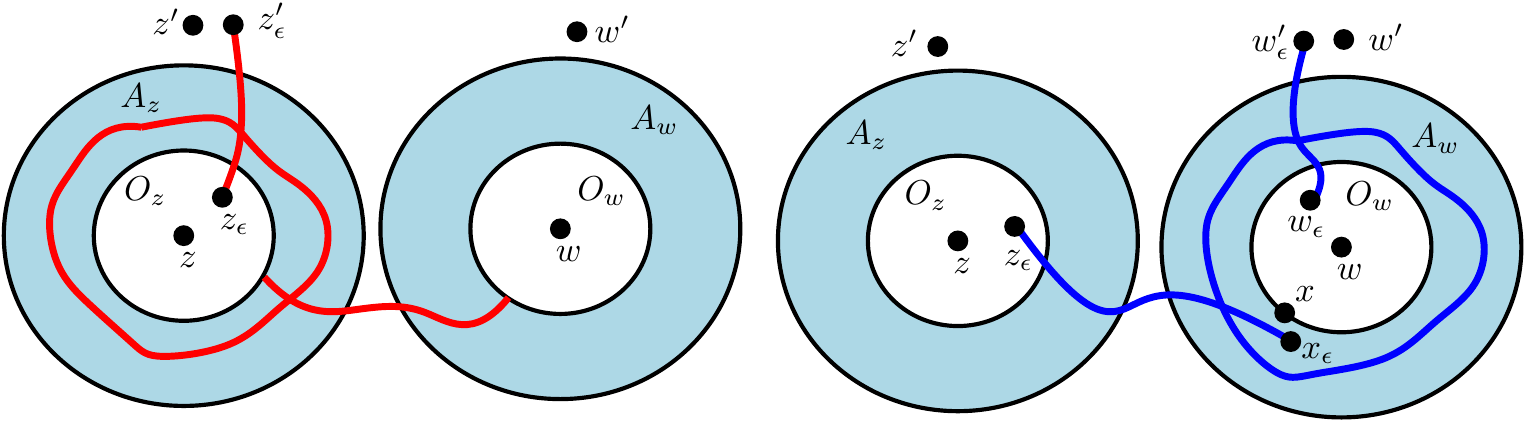}
\vspace{-0.01\textheight}
\caption{Illustration of the proof of Lemma~\ref{lem-finite-dist}. \textbf{Left:} The union of the red paths includes a path from $z_\ep$ to $O_w$, which leads to~\eqref{eqn-finite-dist0-ep}. \textbf{Right:} The union of the blue paths contains a path from $w_\ep$ to $z_\ep$, which leads to~\eqref{eqn-finite-dist1-ep}.  
}\label{fig-singular-pt}
\end{center}
\vspace{-1em}
\end{figure}

To complete the proofs of our main theorems, it remains to establish that $D_h$ is a complete geodesic metric on $\BB C\setminus \{\text{singular points}\}$. We start with completeness.

\begin{prop} \label{prop-complete}
Almost surely, every $D_h$-Cauchy sequence is convergent. In particular, the restriction of $D_h$ to $\BB C\setminus \{\text{singular points}\}$ is complete.
\end{prop}
\begin{proof}
Let $\{z_n\}_{n\in\BB N}$ be a Cauchy sequence w.r.t.\ $D_h$. Then $\{z_n\}_{n\in\BB N}$ is $D_h$-bounded, so by Proposition~\ref{prop-infty-dist-ssl}, $\{z_n\}_{n\in\BB N}$ is contained in some Euclidean-compact subset of $\BB C$. 
By Proposition~\ref{prop-holder}, $\{z_n\}_{n\in\BB N}$ is also Cauchy w.r.t.\ the Euclidean metric, so there is a $z\in\BB C$ such that $|z_n-z| \rta 0$. 
We need to show that $D_h(z_n,z) \rta 0$.
To this end, let $\ep  > 0$. By the Cauchy condition, we can find $n_* = n_*(\ep) \in \BB N$ such that $D_h(z_n , z_m) \leq \ep$ for each $n,m\geq n_*$. 
By lower semicontinuity (Proposition~\ref{prop-lower-semicont}), for each $n\geq n_*$ we have $D_h(z,z_n) \leq \liminf_{m\rta\infty} D_h(z_m,z_n) \leq \ep$. 
\end{proof}

\begin{prop} \label{prop-length-metric}
Almost surely, the restriction of $ D_h$ to $\BB C\setminus \{\text{singular points}\}$ is a geodesic metric, i.e., for any $z,w\in\BB C$ with $D_h(z,w) < \infty$, there is a path from $z$ to $w$ of $D_h$-length exactly $D_h(z,w)$. 
\end{prop}
\begin{proof}
Recall that $D_h$ is a complete metric on $\BB C\setminus \{\text{singular points}\}$ (Proposition~\ref{prop-complete}). 
So, by~\cite[Theorem 2.4.16]{bbi-metric-geometry}, it suffices to show that for any non-singular points $z,w\in \BB C$, there exists a midpoint between $z$ and $w$, i.e., a point $x\in\BB C$ such that $D_h(z,x) = D_h(w,x) = \frac12 D_h(z,w)$. Note that Lemma~\ref{lem-finite-dist} implies that $D_h(z,w) < \infty$. 

To product such a midpoint, let $z^\ep \rta z$ and $w^\ep\rta w$ be sequences as in Lemma~\ref{lem-dist-sup}, so that
\eqb \label{eqn-length-lim}
D_h(z,w) = \lim_{\mcl E \ni \ep} \frk a_\ep^{-1} D_h^\ep (z^\ep,w^\ep) .
\eqe
Since $D_h^\ep$ is a smooth Riemannian distance function, it follows that $D_h^\ep$ is a geodesic metric.
Therefore, for each $z,w \in\BB C$ there is a point $x^\ep \in \BB C$ (i.e., the midpoint of the $D_h^\ep$-geodesic from $z^\ep$ to $w^\ep$) such that $D_h^\ep(z^\ep , x^\ep) = D_h^\ep(w^\ep , x^\ep) = \frac12 D_h^\ep(z^\ep,w^\ep)$. 
By Proposition~\ref{prop-infty-dist-ssl}, it holds with probability tending to 1 as $\ep\rta 0$ and then $R \rta\infty$ that $x^\ep \in B_R(0)$. 
Since $\ol{B_R(0)}$ is compact, we can a.s.\ find a subsequence $\mcl E'$ of $\mcl E$ and a point $x\in \BB C$ such that $x^\ep \rta x$. 
By Lemma~\ref{lem-dist-seq},
\eqbn
D_h(z,x) 
\leq \liminf_{\mcl E' \ni \ep\rta 0} \frk a_\ep^{-1} D_h^\ep(z^\ep, x^\ep) 
 = \frac12 \liminf_{\mcl E' \ni \ep\rta 0} \frk a_\ep^{-1} D_h^\ep(z^\ep , w^\ep ) 
 = \frac12 D_h(z,w) 
\eqen
where the last equality is by~\eqref{eqn-length-lim}. 
Similarly, $D_h(w,x) \leq \frac12 D_h(z,w)$. By combining this with the triangle inequality (Proposition~\ref{prop-triangle-inequality}), we get
\eqb
D_h(z,w) \leq D_h(z,x) + D_h(w,x) \leq \frac12 D_h(z,w) + \frac12 D_h(z,w) =  D_h(z,w) .
\eqe
Hence both inequalities must in fact be equalities, which is only possible if $D_h(w,x) = D_h(z,x) = \frac12 D_h(z,w)$. 
\end{proof}

\subsection{Proofs of main theorems} 
\label{sec-thm-proofs}

We have now proven all of the assertions of our main theorems. Here, we record exactly where each assertion was proven.

\begin{proof}[Proof of Theorem~\ref{thm-lfpp-tight}]
Assertion~\ref{item-tight} is proven in Proposition~\ref{prop-lower-semicont}. 
Assertion~\ref{item-metric} is proven in Proposition~\ref{prop-triangle-inequality}.
Assertion~\ref{item-rational} follows from Proposition~\ref{prop-rational-agree} and the fact that our original coupling was chosen so that $\frk a_\ep^{-1} D_h^\ep(u,v) \rta D_h(u,v)$ for each $u,v\in\BB Q^2$.
Assertion~\ref{item-constant} follows from~\eqref{eqn-ssl-constant} and~\eqref{eqn-ssl-constant-lim}. 
\end{proof}

\begin{proof}[Proof of Theorem~\ref{thm-ssl-properties}]
Assertion~\ref{item-finite-ae} is proven in Lemma~\ref{lem-finite-ae}. 
Assertion~\ref{item-bounded} is proven in Proposition~\ref{prop-infty-dist-ssl}. 
Assertion~\ref{item-holder} is proven in Proposition~\ref{prop-holder}. 
Assertion~\ref{item-singularity} follows from Lemma~\ref{lem-finite-dist} and Proposition~\ref{prop-complete}.
Assertion~\ref{item-geodesic} follows from Proposition~\ref{prop-length-metric}.
Assertion~\ref{item-thick-pt} is proven in Proposition~\ref{prop-thick-pt}. 
\end{proof}

\appendix

\section{Appendix: Gaussian estimates}
\label{sec-appendix} 

Here we record some elementary estimates for Gaussian random variables which are needed for our proofs.

\begin{lem} \label{lem-gaussian-var}
Let $\bd X = (X_1,\dots,X_n)$ be a centered Gaussian vector with $\max_{i\in [1,n]_{\BB Z}} \op{Var} X_i =\sigma^2$. 
Let $C>0$ and let $F : \BB R^n \rta \BB R$ be a function which is $C$-Lipschitz continuous w.r.t.\ the $L^\infty$ norm. 
We have
\eqb \label{eqn-gaussian-var}
\op{Var} F(\bd X) \preceq C^2 \sigma^2  
\eqe
with a universal implicit constant. 
\end{lem}
\begin{proof}
Let $m$ be the median of $F(\bd X)$ and let $B := F^{-1}((-\infty,m])$ and $B' := F^{-1}([m,\infty))$. Then $\BB P[\bd X \in B] \geq 1/2$ and $\BB P[\bd X \in B'] \geq 1/2$. By a standard Gaussian concentration inequality (see, e.g.,~\cite[Lemma 2.1]{dzz-heat-kernel}), there is a universal constant $c_0   > 0$ such that for each $T \geq c_0 \sigma$, 
\eqb
\BB P\left[ \min_{\bd x\in B} |\bd X  - \bd x|_\infty  > T\right] \leq c_0 \exp\left( - \frac{(T-c_0 \sigma)^2}{2\sigma^2} \right) 
\eqe
and the same is true with $B'$ in place of $B$. Hence, with probability at least $1-2c_0 \exp\left( - \frac{(T-c_0 \sigma)^2}{2\sigma^2} \right) $, there exists $\bd x , \bd x' \in \BB R^n$ such that $F(\bd x) \leq m$, $F(\bd x') \geq m$, and $\max\{|\bd X - \bd x|_\infty , |\bd X - \bd x'|_\infty \} \leq T$. 
Since $F$ is $C$-Lipschitz, this means that 
\eqb
  F(\bd X) - m \leq F(\bd X) - F(\bd x) \leq C |\bd X - \bd x|_\infty \leq C T ,
\eqe
and similarly $F(\bd x) - m \geq - C T$. Hence for $T > c_0\sigma$, 
\eqb \label{eqn-gaussian-var-prob}
\BB P\left[ |F(\bd X) - m| > C T \right] \leq c_0 \exp\left( - \frac{(T-c_0 \sigma)^2}{2\sigma^2} \right)  .
\eqe

By substituting $S = C T$ in~\eqref{eqn-gaussian-var-prob}, we compute
\allb \label{eqn-gaussian-var-int}
\op{Var} F(\bd X) 
&\leq \BB E\left[ (F(\bd X) - m)^2 \right] \notag\\
&= 2 \int_0^\infty S \BB P\left[ |F(\bd X) - m| > S \right] \,dS \notag\\
&\leq C^2 c_0^2 \sigma^2 + 2 c_0 \int_{C c_0\sigma} S  \exp\left( - \frac{(S - C c_0 \sigma)^2}{2C^2 \sigma^2} \right) \,dS .
\alle
The last integral is equal to $2\pi C \sigma \BB E\left[ (Y + C c_0 \sigma) \BB 1_{(Y\geq 0)} \right]$ where $Y$ is a centered Gaussian random variable with variance $C^2 \sigma^2$. 
Hence this integral is bounded above by a universal constant times $C^2 \sigma^2$. 
Combining this bound with~\eqref{eqn-gaussian-var-int} now yields~\eqref{eqn-gaussian-var}.
\end{proof}

\begin{lem} \label{lem-gaussian-trunc}
Let $R > 1$ and let $X$ be a centered Gaussian random variable with variance $R$. Also let $\xi,\beta >0$. We have
\eqb \label{eqn-gaussian-trunc} 
\BB E\left[ e^{\xi X} \BB 1_{(X \leq \beta   R)} \right] = e^{(\xi (\xi\wedge\beta) - (\xi\wedge\beta)^2/2) R + o(R)} ,\quad \text{as $R \rta\infty$.}
\eqe
\end{lem}
\begin{proof} 
By the Gaussian tail bound,
\eqb
\BB E\left[ e^{\xi X} \BB 1_{(X \leq \beta   R)} \right]
\succeq e^{\xi(\xi\wedge \beta) R } \BB P\left[ X \in [ (\xi\wedge \beta) R - 1 ,  (\xi \wedge \beta) R ] \right]
\geq e^{(\xi (\xi\wedge\beta) - (\xi\wedge\beta)^2/2) R + o(R)} .
\eqe
This gives the lower bound in~\eqref{eqn-gaussian-trunc}.

To prove the upper bound, fix a small parameter $\zeta >0$. 
Let $0 = \alpha_0 < \alpha_1 < \dots < \alpha_N = \beta$ be a partition of $[0,\beta]$ with $\max_{i \in [1,N]_{\BB Z} } (\alpha_i - \alpha_{i-1}) \leq \zeta$. 
By the Gaussian tail bound, for each $i\in [1,N]_{\BB Z}$,
\eqb
\BB P\left[ X \in [\alpha_{i-1}  R , \alpha_i  R ] \right] \leq e^{-\alpha_{i-1}^2 R / 2} .  
\eqe
We can therefore compute 
\allb \label{eqn-gaussian-trunc-sum}
\BB E\left[ e^{\xi X} \BB 1_{(X \leq \beta   R)} \right]
&= \BB E\left[ e^{\xi X} \BB 1_{(X < 0)} \right] +  \sum_{i=1}^N \BB E\left[  e^{\xi X} \BB 1_{(X \in [\alpha_{i-1}  R , \alpha_i  R ])} \right] \notag\\
&\leq 1 +  \sum_{i=1}^N e^{(\xi \alpha_i - \alpha_{i-1}^2/2) R} \notag\\
&\preceq 1 + \max_{\alpha \in [0,\beta]} e^{(\xi \alpha - \alpha^2/2 + o_\zeta(1)) R} 
\alle
where in the last line we used that $\alpha_i - \alpha_{i-1} \leq \zeta$. Here, $o_\zeta(1)$ denotes a deterministic quantity which converges to 0 as $\zeta\rta0$ and depends only on $\xi,\beta$.
The maximum of $\xi \alpha - \alpha^2/2$ over all $\alpha \in [0,\beta]$ is attained at $\xi \wedge \beta$, where it equals $\xi (\xi\wedge \beta)  - (\xi\wedge \beta)^2/2$. 
Plugging this into~\eqref{eqn-gaussian-trunc-sum} and sending $\zeta\rta 0$ sufficiently slowly as $R\rta\infty$ now gives the upper bound in~\eqref{eqn-gaussian-trunc}.
\end{proof}

\bibliography{cibib}
%\bibliography{addbib}
%\bibliographystyle{alpha}
\bibliographystyle{hmralphaabbrv}

\end{document}